\documentclass{article}

\usepackage{arxiv}
\usepackage[colorlinks=true,breaklinks=true,bookmarks=true,urlcolor=blue,
     citecolor=blue,linkcolor=blue,bookmarksopen=false,draft=false]{hyperref}

\usepackage{graphicx}
\usepackage{amssymb}
\usepackage{epstopdf,comment}
\usepackage{  mathrsfs, enumerate}
\usepackage{amsmath}
\usepackage{booktabs} 
\usepackage{wrapfig}
\usepackage{color}
\usepackage{caption}
\newtheorem{assumption}{Assumption}

\newtheorem{theorem}{Theorem}
\newtheorem{proposition}{Proposition}
\newtheorem{lemma}{Lemma}
\newtheorem{corollary}{Corollary}
\newtheorem{remark}{Remark}
\newtheorem{proof}{Proof}

\newcommand{\tr}{\operatorname{tr}}

\newcommand{\divergence}{\operatorname{div}}

\newcommand{\req}[1]{Eq.\,(\ref{#1})}

\title{Quantification of Model Uncertainty on Path-Space via Goal-Oriented Relative Entropy}

\author{
  Jeremiah Birrell\\
  Department of Mathematics and Statistics\\
  University of Massachusetts Amherst\\
  Amherst, MA 01003,  USA \\
  \texttt{birrell@math.umass.edu} \\
  \And
    Markos A. Katsoulakis\\
    Department of Mathematics and Statistics\\
  University of Massachusetts Amherst\\
  Amherst, MA 01003,  USA \\
  \texttt{markos@math.umass.edu} \\
\And
    Luc Rey-Bellet\\
    Department of Mathematics and Statistics\\
  University of Massachusetts Amherst\\
  Amherst, MA 01003,  USA \\
  \texttt{luc@math.umass.edu} 
}

\begin{document}
\maketitle

\begin{abstract}
Quantifying the impact of parametric and model-form uncertainty on the predictions of stochastic models is a key challenge in many applications. Previous work has shown that the relative entropy rate is an effective tool for deriving path-space uncertainty quantification (UQ) bounds on ergodic averages.  In this  work we  identify  appropriate information-theoretic objects for a wider range of quantities of interest on path-space, such as hitting times and exponentially discounted observables, and develop the corresponding UQ bounds. In addition, our method yields tighter UQ bounds, even in cases where previous relative-entropy-based methods also apply, e.g., for ergodic averages.  We illustrate these results with examples from option pricing, non-reversible diffusion processes, stochastic control, semi-Markov queueing  models, and expectations and distributions of hitting times.
\end{abstract}

% keywords can be removed

\section{Introduction}

Probabilistic models are widely used in  engineering, finance, physics, chemistry, and many other fields.  Models of real-world systems inherently carry uncertainty, either in the value of model parameters, or more general non-parametric (i.e., model-form) uncertainty. Such uncertainty can stem from fitting a model to data or from approximating a more complicated, intractable model with one that is simpler and computationally tractable. It is often important to estimate or bound the effect of such uncertainties on quantities-of-interest (QoIs) computed from the model, i.e, to obtain uncertainty quantification (UQ) bounds. This can be especially difficult for dynamical/time-series problems  over a long time-horizon, such as the computation of hitting times, reaction rates, cost functions, and option values. When the uncertainty is  parametric and small,  linearization and perturbative methods are effective tools for quantifying the impact of uncertainty on model predictions;  see, e.g., \cite{10.2307/41713721,doi:10.1063/1.4868649,10.1145/84537.84552,DKPR,Hairer_2010,KIM2007379,PLYASUNOV2007724,doi:10.1021/jp061858z,doi:10.1063/1.3677230}.  However, new general methods are needed when the uncertainties are  large or non-parametric. In this paper, we  present a framework for  addressing this problem for a wide range of path-space QoIs.

To  illustrate the challenges inherent in model-form uncertainty on path-space, consider the following examples  (see Section \ref{sec:examples} below for further detail):
\begin{enumerate}
\item  In mathematical finance,  the value of an asset, $X_t$, in a variable-interest-rate environment  may be modeled by a stochastic differential equation of the form
\begin{align}\label{eq:intro_option}
  d{X}_t=(r+\Delta r(t,{X}_t,Y_t)){X}_tdt+\sigma {X}_td{W}_t,\,\,\,\,r,\sigma>0.
\end{align}
Many QoIs, such as option values, involve hitting times for $X_t$, and hence are naturally phrased on path-space with a potentially large time horizon. When the variable  rate perturbation, $\Delta r$,  is influenced by some other process, $Y_t$, that is incompletely known (such as in the Vasicek model), one may wish to approximate \req{eq:intro_option} with a simpler, tractable model, e.g., geometric Brownian motion (obtained by setting $\Delta r\equiv 0$).  If the true $\Delta r$ is potentially large or has an incompletely known structure then perturbative methods are unable to provide reliable error estimates. One needs new, non-perturbative methods to obtain   bounds on option values, and other QoIs.

\item In queuing systems, Markovian models are widely used due to their mathematical simplicity, even in cases where  real-world data shows that  waiting times are non-exponential, and hence the systems are non-Markovian \cite{doi:10.1198/016214504000001808}. To quantify the model uncertainty, one must estimate/bound   the error incurred from using a Markovian approximation  to a non-Markovian system.     More generally, an approximate model may have a fundamentally different mathematical structure than the `true' model; such scenarios  can be difficult to study through perturbative methods.
\end{enumerate}

 Variational-principle-based methods, relating expectations of a quantity of interest to information-theoretic divergences,  have proven to be effective tools for deriving   UQ bounds.  While they are applicable to parametric uncertainty, such methods are particularly powerful when applied to systems with model-form uncertainty. Methods have been developed using  relative entropy \cite{chowdhary_dupuis_2013,BREUER20131552,GlassermanXu,doi:10.1287/moor.2015.0776,GKRW} and R{\'e}nyi divergences \cite{doi:10.1137/130939730,DKPR}. For stochastic processes, the relative entropy rate  has been used to derive UQ bounds on ergodic averages \cite{doi:10.1063/1.4789612,DKPP,doi:10.1137/15M1047271,KRW,doi:10.1137/19M1237429}.  Ergodic averages constitute an important class of QoIs on path space, but  many other key QoIs are not covered by these previous methods, such as the expectation of hitting times and  exponentially discounted observables.  Here, we develop a new general framework  for path-space UQ  by identifying a larger class of information-theoretic objects (the so-called goal-oriented relative entropies) and proving an accompanying new bound, Theorem \ref{thm:goal_info_ineq3}. This framework can be applied to qualitatively new regimes,  such as the aforementioned hitting times or to ergodic averages in the presence of unbounded perturbations. In addition, even when the previous methods apply (e.g., for ergodic averages with  bounded  perturbations), the  framework developed here  provides quantitatively tighter bounds.

In this work, UQ will refer to the following mathematical problem: One has a baseline model, described by  a probability measure $P$, and an alternative model, $\widetilde{P}$. For a given real-valued function (i.e., QoI), $F$,  one's goal is  to bound the error incurred from using $P$ in place of $\widetilde{P}$, i.e.,
\begin{align}\label{main_problem1}
\!\!\!\!\!\!\!\!\!\!\!\!\!\text{\bf Bound the bias: }\,\,\,\,\,\,\,\,  E_{\widetilde{P}}[F]-E_P[F].
\end{align}
 (Here, and in the following, $E_Q$ will denote the expectation under the probability measure $Q$.) $F$ is the quantity one is interested in, such as a hitting time or an ergodic average.   The baseline model, $P$, is typically an approximate but `tractable' model, meaning one can calculate QoIs exactly, or it is  relatively inexpensive to simulate. However, it generally contains many sources of error and uncertainty; it may depend on parameters with uncertain values (obtained from experiment, Monte-Carlo simulation, variational inference, etc.) or is obtained via some approximation procedure (coarse graining, neglecting memory terms, linearization, asymptotic approximation, etc.).   Any QoI computed from $P$ therefore has significant uncertainty associated with it. In contrast, $\widetilde{P}$, is thought of as the `true' (or at least, a more precise) model, but due to its complexity or a lack of knowledge, it is  intractable.   For instance, in the finance example discussed above, $\widetilde{P}$ is the distribution on path space of the solution to \req{eq:intro_option}, and one possible choice of   $P$  is the distribution  of geometric Brownian motion (i.e., the solution to \req{eq:intro_option} with $\Delta r\equiv 0$).  In the queuing example, $\widetilde{P}$ is the distribution of the `true' non-Markovian model on path-space, while $P$ is the distribution of the Markovian approximation.

The results in \cite{chowdhary_dupuis_2013,BREUER20131552,GlassermanXu,doi:10.1287/moor.2015.0776,GKRW} produce  bounds on the general UQ problem (\ref{main_problem1}) in terms of the relative entropy, i.e., Kullback-Leibler (KL) divergence, $R(\widetilde{P}\|P)$, which quantifies the discrepancy between the models $\widetilde{P}$ and $P$.  When comparing stochastic  processes, $\{X_t\}_{t\in[0,\infty)}$ and  $\{\widetilde X_t\}_{t\in[0,\infty)}$,   one often finds that their distributions on path-space, $P_{[0,\infty)}$ and $\widetilde{P}_{[0,\infty)}$,    have infinite relative entropy.  Hence, when dealing with an unbounded time horizon, these previous methods do not  produce (nontrivial) UQ bounds. As we will see, this infinite relative entropy is often an indication that one is using the wrong information-theoretic quantity. In \cite{DKPP}  it was shown how to overcome this issue when the QoI is an ergodic average;  there, the relative entropy rate  was found to be the proper information-theoretic quantity for bounding ergodic averages. See Section \ref{sec:UQ_background} for further  discussion of this, and  other  background material.   In this paper, we  derive a  general procedure for identifying  information-theoretic quantities, matched to a particular QoI, that can then be used to obtain UQ bounds.   

The  intuition underlying our approach is conveyed by the following example: Suppose one wants to bound the $\widetilde{P}$-expectation of a stopping time, $\tau$.  Even if $\tau$ is unbounded, the QoI $F=\tau$ depends only on the path of the process up to time $\tau$, and not on the  path for all $t\in[0,\infty)$.  Hence, one would expect  UQ bounds on $\tau$ to only require knowledge of the  dynamics up to the stopping time;   a bound on $R(\widetilde{P}_{[0,\infty)}\|P_{[0,\infty)})$ should not be necessary. Rather, one should be able to construct UQ bounds using the relative entropy between the baseline and alternative processes  {\em up to the stopping time}:
\begin{align}\label{eq:goal_oriented_KL}
R(\widetilde{P}|_{\mathcal{F}_{\tau}}\|P|_{\mathcal{F}_{\tau}}),
\end{align}
where $\mathcal{F}_\tau$ is the filtration for the process up to the stopping-time $\tau$; intuitively, $\mathcal{F}_\tau$ captures the information content of the process up to the stopping time. \req{eq:goal_oriented_KL} is an example of a goal-oriented relative entropy, for the QoI $\tau$. 

\begin{wraptable}{r}{.25\textwidth}
\centering
{\renewcommand{\arraystretch}{1.5} 
    \begin{tabular}{l|l}
     & \textbf{QoI} \\
     \hline
      1&$\widetilde{P}(\tau\leq T)$\\      
      2& $E_{\widetilde{P}}[\tau]$\\      
      3& $E_{\widetilde{P}}[\int_0^\infty f_s \lambda e^{-\lambda s}ds]$\\
     4& $E_{\widetilde{P}}[T^{-1}\int_0^T f(X_t)dt]$\\     
    \end{tabular}}
\caption{}\label{tab:QoIs}
\end{wraptable}

 Our results will make the above intuition rigorous and  general, and we will provide appropriate information-theoretic objects for computing UQ bounds on a variety of  path-space QoIs, generalizing \req{eq:goal_oriented_KL}. Moreover, we will show how partial information regarding  the structure of the alternative model, $\widetilde{P}$, can be used to  make such bounds {\em computable}. The method is developed in an abstract framework in Section \ref{sec:main_result}, before being specializing to path-space QoIs in Section \ref{sec:UQ_processes}. The following list shows several  classes of QoIs to which our methods apply (see also  Table \ref{tab:QoIs}):
\begin{enumerate}
\item Cumulative distribution function (CDF) of a stopping time.
\item Expected value of a stopping time.
\item Expectation of exponentially discounted QoIs.
\item Expectation of time averages.
\end{enumerate}
We will also show how one can bound goal-oriented relative entropies, such as \req{eq:goal_oriented_KL}, often a key step in obtaining computable UQ bounds.  Finally, we illustrate these  results  through several examples in Section \ref{sec:examples}:
  \begin{enumerate}
  \item The distribution and expectation of hitting times of Brownian motion with constant drift, as compared to Brownian motion perturbed by a non-constant drift; see Section \ref{sec:BM_const_drift} and Appendix \ref{app:BM_Etau}.
\item UQ bounds for invariant measures of non-reversible stochastic differential equations (SDEs); see Section \ref{example:non-rev}.
    \item Robustness of linear-quadratic stochastic control under non-linear perturbations; see Section \ref{example:control}.
      \item Robustness of continuous-time queueing models under non-exponential waiting-time perturbations (i.e.,  semi-Markov  perturbations); see Section \ref{example:birth_death}.
  \item The value of American put options in a variable-interest-rate environment (including the Vasicek model); see Section \ref{sec:options} and Appendix \ref{app:options}.
  \end{enumerate}

\section{Background on Uncertainty Quantification  via Information-Theoretic Variational Principles}\label{sec:UQ_background}
In this section we record important background  on the information-theoretic-variational-principle approach to UQ; the new methods presented in this paper (starting in Section \ref{sec:main_result}) will be built upon these foundations.

First, we fix some notation.   Let $P$ be a probability measure on a measurable space $(\Omega,\mathcal{F})$. $\overline{\mathbb{R}}$ will denote the extended reals  and we will refer to a random variable  $F: \Omega \to \overline{\mathbb{R}}$ as a quantity-of-interest (QoI).  If $F\in L^1(P)$ we write
\begin{align}
\widehat{F} = F - E_P[F]
\end{align}
for the centered quantity of mean $0$.  The cumulant generating function of $F$ is defined by 
\begin{equation}
\Lambda_P^F(c) = \log E_P[e^{ c F}],
\end{equation} 
where $\log$ denotes the natural logarithm and we use the continuous extensions of $\exp(x)$ to $\overline{\mathbb{R}}$ and of $\log(x)$ to $[0,\infty]$. 

Recall also that the relative entropy (i.e., KL divergence) of another probability measure, $\widetilde{P}$, with respect to $P$  is defined by 
\begin{equation}\label{rel_ent_def}
R( \widetilde{P} \| P)= \left\{ \begin{array}{cl}  E_{\widetilde{P}} \left[\log\left(\frac{d\widetilde{P}}{dP}\right)\right]  & \textrm{ if } \widetilde{P} \ll P  \\  + \infty & \textrm{ otherwise}
\end{array} \,.
\right. 
\end{equation}
It has the property of a divergence, that is $R( \widetilde{P} \| P) \ge 0$  and $R( \widetilde P \| P) = 0$ if and only if $\widetilde{P}=P$; see, for example, \cite{dupuis2011weak} for this and further properties.

The starting point of our approach to UQ on path-space is a KL-based UQ bound, which we here call the {\em UQ information inequality}.  The bound was derived in \cite{chowdhary_dupuis_2013,DKPP}; a similar inequality is used in the context of concentration inequalities, see e.g. \cite{BLM}, and was also used independently in  \cite{BREUER20131552,GlassermanXu}.  We summarize the proof  in Appendix \ref{app:base_inequality_proof} for completeness: Let  $F:\Omega\to {\mathbb{R}}$, $F\in L^1(\widetilde P)$.  Then 
\begin{align}\label{goal_oriented_bound}
\pm E_{\widetilde P}[F]\leq \inf_{c>0}\left\{\frac{1}{c} \Lambda_{P}^{{F}}(\pm c)+\frac{1}{c}R(\widetilde P\|P)\right\}.
\end{align}
If $F$ is also in $L^1(P)$ then one can apply \req{goal_oriented_bound} to the centered QoI, $\widehat{F},$ to obtain  centered UQ bounds, i.e.,  bounds on $E_{\widetilde P}[F]-E_P[F]$.  \req{goal_oriented_bound} is to be understood as two formulas, one with all the upper signs and another with all  the lower signs.  This remark applies to all other uses of $\pm$ in this paper. The bounds (\ref{goal_oriented_bound}) have many appealing properties, including tightness over relative-entropy neighborhoods \cite{chowdhary_dupuis_2013,DKPP}, a  linearization formula in the case of small relative entropy \cite{doi:10.1287/moor.2015.0776,DKPP}, and a divergence property \cite{DKPP}; see also \cite{doi:10.1111/mafi.12050,BREUER20131552,DKPR}.

 In \cite{DKPP,KRW,doi:10.1137/19M1237429}, \req{goal_oriented_bound} was used to derive UQ bounds for ergodic averages on path-space, both in discrete and continuous time. The goal of the present work is to extend these  methods to apply to more general  path-space QoIs. As motivation, here we provide a summary of one of the results from \cite{DKPP}:    Let $\mathcal{X}$ be a Polish space and consider two distributions, $P_{[0,T]}$ and $\widetilde{P}_{[0,T]}$, on path-space up to time $T$, $C([0,T],\mathcal{X})$; these are the baseline and alternative models respectively.  Assuming that the alternative process starts in an invariant distribution, $\widetilde{\mu}^*$, and the baseline process starts in an arbitrary distribution, $\mu$, one obtains the following UQ bound on ergodic averages (see Section 3.2 in  \cite{DKPP} for this result and further generalizations):\\
Let $f:\mathcal{X}\to\mathbb{R}$ be a bounded observable and  for $T>0$ define the time average $A_T=T^{-1}\int_0^Tf(X_t)dt$, where $X_t$ is evaluation at time $t$ for a path in  $C([0,T],\mathcal{X})$.  Then
\begin{align}\label{eq:ergodic_avg}
\pm\left(E_{\widetilde{P}_{[0,T]}}[A_T]-E_{P_{[0,T]}}[A_T]\right)\leq \inf_{c>0}\left\{\frac{1}{cT}\Lambda_{P_{[0,T]}}^{\widehat{A}_T}(\pm cT)+\frac{1}{c}H(\widetilde{P}\|P)+\frac{1}{cT}R(\widetilde{\mu}^*\|\mu)\right\},
\end{align}
where $H(\widetilde{P}\|P)$ is the relative entropy rate:
\begin{align}\label{eq:rel_ent_rate}
H(\widetilde{P}\|P)\equiv\lim_{T\to\infty}\frac{1}{T}R\left(\widetilde{P}_{[0,T]}\|P_{[0,T]}\right).
\end{align}
To arrive at  \req{eq:ergodic_avg}, one applies \req{goal_oriented_bound} to $T\widehat{A}_T$, divides by $T$, and employs the chain-rule of relative entropy to write the result in terms of the relative entropy rate. Under appropriate assumptions, one can show that the upper bound remains finite as $T\to\infty$. In particular, the relative entropy rate  between the alternative and baseline processes is positive and finite in many cases (see the online supplement to \cite{DKPP}), and hence provides an appropriate information-theoretic quantity for controlling ergodic averages in the long-time regime.

\section{UQ Bounds via Goal Oriented Information Theory}\label{sec:main_result}
\req{eq:ergodic_avg} provides UQ bounds on ergodic averages, but there are many other classes of path-space QoIs that are important in applications; see Table \ref{tab:QoIs}. Naively applying the UQ information inequality, \req{goal_oriented_bound}, to one of these other path-space QoIs that has an infinite time horizon (e.g., $F=\tau$ where $\tau$ is an unbounded stopping-time) will generally lead to trivial (i.e., infinite) bounds.  This is due to the fact that  $R(\widetilde{P}_{[0,\infty)}\|P_{[0,\infty)})$ is  infinite in most cases (see \req{eq:rel_ent_rate} and the subsequent discussion). For ergodic averages, the  argument leading to \req{eq:ergodic_avg} circumvents this difficulty by expressing the bound in terms of the relative entropy rate.  Although the derivation of  \eqref{eq:ergodic_avg} cannot be generalized to other QoIs, the idea of obtaining finite bounds by utilizing an alternative information theoretic quantity can be generalized, and we will do so in this section.

 A naive application of  \req{goal_oriented_bound}  ignores a hidden degree of freedom: the choice of sigma algebra on which one defines the models.  More specifically, if the QoI, $F$, is measurable with respect to  $\mathcal{G}$, a sub sigma-algebra of the full sigma algebra $\mathcal{F}$ on $\Omega$, then one can write
\begin{align}\label{eq:goal_exp}
E_{\widetilde{P}}[F]-E_{P}[F]=E_{\widetilde{P}|_{\mathcal{G}}}[F]-E_{P|_{\mathcal{G}}}[F],
\end{align}
where $Q|_{\mathcal{G}}$ denotes the restriction of a measure $Q$ to the sub sigma-algebra $\mathcal{G}$; $\mathcal{G}$ captures the information content of $F$ and working on $\mathcal{G}$ will lead to UQ bounds that are targeted at the QoI.  More specifically, one can  apply \req{goal_oriented_bound} to the baseline model $P|_{\mathcal{G}}$ and alternative model $\widetilde{P}|_{\mathcal{G}}$, thereby obtaining a UQ bound in terms of a {\em goal-oriented relative entropy} $R(\widetilde{P}|_{\mathcal{G}}\|P|_{\mathcal{G}})$.  We use the term goal-oriented in order to emphasize the fact that the information-theoretic quantity is tailored to the QoI under consideration, through the use of $\mathcal{G}$.  As we will show, the data-processing inequality (see \cite{1705001}) implies
\begin{align}
R(\widetilde{P}|_{\mathcal{G}}\|P|_{\mathcal{G}})\leq R(\widetilde{P}\|P),
\end{align}
and there are many situations where the inequality is strict.  Thus one often obtains tighter UQ bounds by using a goal-oriented relative entropy. The above idea is quite general, and so  we develop the  theory in this section without any reference to path-space. We also prove new (quasi)convexity properties regarding optimization problems of the form \eqref{goal_oriented_bound}.
\begin{theorem}[Goal-Oriented Information Inequality]\label{thm:obs_adap_info_ineq}
Let $P$ and $\widetilde{P}$ be probability measures on $(\Omega,\mathcal{F})$ and $F:\Omega\to\overline{\mathbb{R}}$.
\begin{enumerate}

\item Let $\mathcal{G}$ be a sub sigma algebra of $\mathcal{F}$ and $F\in L^1(\widetilde P)$  be $\mathcal{G}$-measurable. Then
\begin{align}
\pm E_{\widetilde P}[F]\leq& \inf_{c>0}\left\{\frac{1}{c} \Lambda_{P}^{F}(\pm c)+\frac{1}{c} R(\widetilde P|_{\mathcal{G}}\|P|_{\mathcal{G}})\right\}\label{eq:goal_info_ineq2}\\
=&\inf_{\eta>0}\left\{\eta \Lambda_{P}^{F}(\pm \eta^{-1})+\eta R(\widetilde P|_{\mathcal{G}}\|P|_{\mathcal{G}})\right\}, \label{eq:goal_info_ineq2_convex}
\end{align}
where we interpret $-\infty+\infty\equiv \infty$.  The objective function $c\mapsto \frac{1}{c} \Lambda_{P}^{F}(\pm c)+\frac{1}{c} R(\widetilde P|_{\mathcal{G}}\|P|_{\mathcal{G}})$ is quasiconvex and the objective function $\eta\mapsto \eta \Lambda_{P}^{F}(\pm \eta^{-1})+\eta R(\widetilde P|_{\mathcal{G}}\|P|_{\mathcal{G}})$ is convex.

\item If   $\mathcal{G}\subset\mathcal{H}\subset\mathcal{F}$ are sub sigma-algebras   then
\begin{align}\label{eq:data_proc}
 R(\widetilde{P}|_{\mathcal{G}}\|P|_{\mathcal{G}})\leq R(\widetilde{P}|_{\mathcal{H}}\|P|_{\mathcal{H}})\leq R(\widetilde{P}\|P).
\end{align}

\end{enumerate}

\end{theorem}
\begin{proof}
\begin{enumerate}
\item   For a real-valued $F\in L^1(\widetilde P)$, write
\begin{align} 
E_{\widetilde{P}}[F]=E_{\widetilde{P}|_{\mathcal{G}}}[F]
\end{align}
and then apply \req{goal_oriented_bound} to arrive at \req{eq:goal_info_ineq2}.  The bound can then be extend to $\overline{\mathbb{R}}$-valued $F$ by taking limits.  Next, by changing variables $c=1/\eta$ in the infimum we arrive at \req{eq:goal_info_ineq2_convex}. To see that the objective function in \eqref{eq:goal_info_ineq2_convex} is convex, first recall that the cumulant generating function $\Lambda_P^F(\pm c)$ is convex (this can be proven via H{\"o}lder's inequality).  The perspective of a convex function is convex (see, e.g., page 89 in \cite{boyd2004convex} and note that the result can easily be exteded to $\overline{\mathbb{R}}$-valued convex functions), and therefore $\eta\mapsto \eta \Lambda_P^F(\pm \eta^{-1})$ is convex.  Adding a linear term preserves convexity, hence $\eta \Lambda_{P}^{F}(\pm \eta^{-1})+\eta R(\widetilde P|_{\mathcal{G}}\|P|_{\mathcal{G}})$ 
is convex in $\eta$.  Finally, quasiconvexity of the objective function in \eqref{eq:goal_info_ineq2} follows from the general fact that if $h:(0,\infty)\to \overline{\mathbb{R}}$ is convex then $c\mapsto h(1/c)$ is quasiconvex on $(0,\infty)$.  To see this, let $c_1,c_0>0$ and $\lambda\in(0,1)$. Then $1/(\lambda c_1+(1-\lambda)c_0)= s c_1^{-1}+(1-s)c_0^{-1}$ for some $s\in(0,1)$.  Therefore, convexity of $h$ implies
\begin{align}\label{eq:quasiconvex}
h\left(1/(\lambda c_1+(1-\lambda)c_0)\right)\leq sh( c_1^{-1})+(1-s)h(c_0^{-1})\leq \max\{h(c_1^{-1}),h(c_0^{-1})\}.
\end{align}

 \item If   $\mathcal{G}\subset\mathcal{H}\subset\mathcal{F}$ are sub sigma-algebras   then we can write $P|_{\mathcal{G}}=(id)_*(P|_{\mathcal{H}})$ (we let $id$ denote the identity on $\Omega$, thought of as a measurable map from $(\Omega,\mathcal{H})$ to $(\Omega,\mathcal{G})$, and $\psi_*P$ denotes the distribution of the random quantity $\psi$), and similarly for $\widetilde{P}$. The data processing inequality (see Theorem 14 in \cite{1705001}) implies $R(\psi_*\widetilde{P}\|\psi_*P)\leq R(\widetilde{P}\|P)$ for any random quantity $\psi$, and so we arrive at \req{eq:data_proc}.
\end{enumerate}

\end{proof}
\begin{remark}\label{remark:unimodal}
  The objective functions in both \eqref{eq:goal_info_ineq2} and \eqref{eq:goal_info_ineq2_convex} are  unimodal, due to (quasi)convexity. In the majority of our theoretical discussions we will use the form \eqref{eq:goal_info_ineq2} of the UQ bound, but the convex form \eqref{eq:goal_info_ineq2_convex} is especially appealing  for computational purposes, as the optimization can be done numerically via standard techniques and with  guarantees on the convergence. In the examples in Section \ref{sec:examples} we will generally perform the optimization numerically, using either a  line search or  Nelder-Mead simplex method \cite{10.1093/comjnl/7.4.308,doi:10.1137/S1052623496303470}.
\end{remark}
\begin{remark}
\req{eq:data_proc} is a version of the data processing inequality, which holds in much greater generality than the form  stated above.  In particular, it holds for any $f$-divergence \cite{1705001}. For relative entropy, the data-processing inequality can be obtained from the chain rule, together with the fact that marginalization can only reduce the relative entropy.
\end{remark}
In particular, one can bound the probability of an event  $A\in\mathcal{G}$ by applying Theorem \ref{thm:obs_adap_info_ineq} to $F=1_A$ (see the example in Section \ref{sec:Vasicek}):
\begin{corollary}\label{cor:event_prob_UQ}
Let $\mathcal{G}$ be a sub sigma-algebra of $\mathcal{F}$ and $A\in\mathcal{G}$.  Then
\begin{align}\label{eq:event_UQ_bound}
&\pm \left(\widetilde P(A)-P(A)\right)\\
\leq& \inf_{c>0}\left\{\frac{1}{c} \log \left(P(A)e^{\pm c(1-P(A))}+(1-P(A))e^{\mp cP(A)}\right)+\frac{1}{c}R(\widetilde{P}|_{\mathcal{G}}\|{P}|_{\mathcal{G}})\right\}.\notag
\end{align}
 
\end{corollary} 
\begin{remark}
UQ bounds based on Corollary \ref{cor:event_prob_UQ} are far from  optimal for very rare events (i.e., when $P(A)\ll 1$). See \cite{doi:10.1137/130939730,DKPR} for a related approach to UQ, using R{\'e}nyi divergences instead of relative entropy, which produces tighter bounds for rare events.
 \end{remark}

  Theorem \ref{thm:obs_adap_info_ineq} is particularly useful because of the manner in which it  splits the UQ problem into two sub-problems:
\begin{enumerate}
\item Compute or bound the cumulant generating function of the QoI with respect to the baseline model, $\Lambda_{P}^{{F}}(c)$. (Note that there is no dependence on the alternative model, $\widetilde P$.)
\item Bound the relative entropy of the alternative model with respect to the baseline model, $R(\widetilde P|_{\mathcal{G}}\|P_{\mathcal{G}})$, on an appropriate sub sigma-algebra, $\mathcal{G}$. (The primary difficulty here is the presence of the intractable/unknown model $\widetilde{P}$.)
\end{enumerate}

We generally consider the baseline model, $P$, and hence sub-problem (1), to be `tractable'.  In practice, computing or bounding the CGF under $P$  may still be a difficult task. Bounds on CGFs are often used in the derivation of   concentration inequalities; see \cite{BLM} for an overview of this much-studied problem. For prior uses of CGF bounds in UQ, see \cite{doi:10.1137/10080782X,GKRW,doi:10.1137/19M1237429}.     With these points in mind, our focus  will primarily  be on the second problem; our examples  in Section \ref{sec:examples} will largely involve  baseline models for which the CGF is relatively straightforward to compute.  

Theorem \ref{thm:goal_info_ineq3}, below, shows how one can use  partial information/bounds on the intractable model, $\widetilde{P}$, to produce a computable UQ bound, thus addressing sub-problem 2.  The key step is to select a sub sigma algebra $\mathcal{G}$ and a  $\mathcal{G}$-measurable $G\in L^1(\widetilde{P})$ for which the following relative entropy bound holds:
\begin{align}\label{eq:rel_ent_assump}
R(\widetilde P|_{\mathcal{G}}\|P|_{\mathcal{G}})\leq E_{\widetilde{P}}[G].
\end{align}

\begin{theorem}\label{thm:goal_info_ineq3}
Let  $F:\Omega\to\overline{\mathbb{R}}$, $F\in L^1(\widetilde P)$, be $\mathcal{G}$-measurable, where 
 $\mathcal{G}\subset\mathcal{F}$ is a sub sigma-algebra, and suppose we have a $\mathcal{G}$-measurable real-valued $G\in L^1(\widetilde{P})$ that satisfies $R(\widetilde P|_{\mathcal{G}}\|P|_{\mathcal{G}})\leq E_{\widetilde{P}}[G]$.   Then 
\begin{align} 
\pm E_{\widetilde P}[F]\leq& \inf_{c>0}\left\{\frac{1}{c} \log E_P\left[\exp\left(\pm c F+G\right)\right]\right\}\label{eq:goal_info_ineq3}\\
=&\inf_{\eta>0}\left\{\eta \log E_P\left[\exp\left(\pm \eta^{-1} F+G\right)\right]\right\}.\label{eq:goal_info_ineq3_convex}
\end{align}
The objective function $(0,\infty)\to (-\infty,\infty]$, $c\mapsto \frac{1}{c} \log E_P\left[\exp\left(\pm c F+G\right)\right]$ is quasiconvex and the objective function $(0,\infty)\to (-\infty,\infty]$, $\eta\mapsto \eta \log E_P\left[\exp\left(\pm \eta^{-1} F+G\right)\right]$ is convex.
\end{theorem}
\begin{remark}
 The bounds \eqref{eq:goal_info_ineq3} (or, equivalently, \eqref{eq:goal_info_ineq3_convex}) are tight over the set of allowed $G$'s; see Appendix \ref{app:tight} for a precise statement and proof of this fact.
\end{remark}
\begin{proof}

Given $c_0>0$, apply \req{eq:goal_info_ineq2} from Theorem \ref{thm:obs_adap_info_ineq}  and  the assumption (\ref{eq:rel_ent_assump}) to the $\mathcal{G}$-measurable QoI, $F_{\pm c_0}\equiv F\pm c_0^{-1}G$. Doing so yields
\begin{align}
\pm E_{\widetilde P}[F\pm c_0^{-1}G]\leq& \inf_{c>0}\left\{\frac{1}{c} \Lambda_{P}^{F_{\pm c_0}}(\pm c)+\frac{1}{c}R(\widetilde P|_{\mathcal{G}}\|P|_{\mathcal{G}})\right\}\\
\leq& \inf_{c>0}\left\{\frac{1}{c} \Lambda_{P}^{F_{\pm c_0}}(\pm c)+\frac{1}{c}E_{\widetilde{P}}[G]\right\}.\notag
\end{align}
In particular, bounding the right-hand-side by the value at $c=c_0$, we can cancel the $E_{\widetilde{P}}[G]$ term (which is assumed to be finite) to find
\begin{align}
\pm E_{\widetilde P}[F]\leq&  \frac{1}{c_0} \Lambda_{P}^{F_{\pm c_0}}(\pm c_0).
\end{align}
Taking the infimum over all $c_0>0$ yields \eqref{eq:goal_info_ineq3}.  \req{eq:goal_info_ineq3_convex} then follows by changing variables $c=1/\eta$.  To prove convexity of the objective function in \req{eq:goal_info_ineq3_convex}, first note that the map $c\mapsto \log\int e^{\pm c F}d\mu$ is convex on $(0,\infty)$ for all positive measures, $\mu$; this can be proven using H{\"o}lder's inequality in the same way that one proves convexity of the CGF.  Applying this to $d\mu=e^{G}dP$ we see that  $\log E_P\left[\exp\left(\pm c F+G\right)\right]$ is convex in $c$.  Therefore, by again using convexity of the perspective of a convex function, we find that $\eta\mapsto \eta\log E_P\left[\exp\left(\pm \eta^{-1} F+G\right)\right]$ is convex. Finally, using the quasiconvexity property proven above in \req{eq:quasiconvex}, we see that $c\mapsto c^{-1}\log E_P\left[\exp\left(\pm c F+G\right)\right]$ is quasiconvex, thus completing the proof.
\end{proof}
\begin{remark}\label{remark:old_strat}
 Previously, a primary strategy for using the non-goal-oriented bound \eqref{goal_oriented_bound} in UQ was to find a function $G$ such that $R(\widetilde{P}\|P)\leq E_{\widetilde{P}}[G]$, compute an explicit upper  bound $G\leq D\in\mathbb{R}$,  bound $R(\widetilde{P}\|P)\leq D$, and use this to obtain the (in principle) computable bounds
\begin{align} 
\pm E_{\widetilde P}[F]\leq \inf_{c>0}\left\{\frac{1}{c} \Lambda_{P}^{{F}}(\pm c)+\frac{D}{c}\right\}.
\end{align}
  Of course, if $G$ is not bounded then the resulting UQ bound is uninformative.  The new result \eqref{eq:goal_info_ineq3} is  tighter than the above described strategy, as can be seen by bounding $G\leq D$ in the exponent of \req{eq:goal_info_ineq3}.  In addition, \eqref{eq:goal_info_ineq3} can give finite results even when $G$ is not bounded, as we demonstrate in the examples in Sections \ref{example:non-rev} and \ref{sec:Vasicek}.
\end{remark}
 
Theorem \ref{thm:goal_info_ineq3} takes a relative entropy bound of the form (\ref{eq:rel_ent_assump}) (which involves an expectation under the intractable model $\widetilde{P}$) and produces a UQ bound which only involves expectations under the tractable model, $P$ (see  \req{eq:goal_info_ineq3}).    However, in practice  it is not enough to use just any $G$ that satisfies \req{eq:rel_ent_assump}. If the relative entropy on $\mathcal{G}$ is finite then the (trivial) choice of $G=\log(d\widetilde{P}|_{\mathcal{G}}/dP|_{\mathcal{G}})$ always satisfies  \req{eq:rel_ent_assump} (with equality).   Rather, the practical question is whether one can find a {\em tractable} $G$  that satisfies \req{eq:rel_ent_assump}; the trivial choice  $G=\log(d\widetilde{P}|_{\mathcal{G}}/dP|_{\mathcal{G}})$  typically {\em does not} satisfy this tractability requirement.  Given the fact that $\widetilde{P}$ may only be partially known,  a tractable $G$ will generally need to incorporate bounds/partial information regarding $\widetilde{P}$.   

The choice of an appropriate  $\mathcal{G}$ goes hand-in-hand with the choice of $G$.  The bounds in Theorem \ref{thm:obs_adap_info_ineq} become tighter as one makes $\mathcal{G}$ smaller, with the tightest bound obtained when $\mathcal{G}=\sigma(F)$, where $\sigma(F)$ is the sigma algebra generated by $F$ (i.e., generated by  $\{F^{-1}((a,\infty]): a\in\mathbb{R}\}$), in which case the relative entropy is  $R(\widetilde{P}|_{\sigma(F)}\|P|_{\sigma(F)})=R(F_*\widetilde{P}\|F_*P)$.  However, one must again balance the desire for tightness against the need for {\em computable} bounds. It is generally very difficult to find a tractable $G$ for the relative entropy on $\sigma(F)$, and so a larger sub sigma-algebra must be used.  While the task of finding a suitable $\mathcal{G}$ and $G$ is  problem specific,  for QoIs on path space there are  general strategies one can follow.  We discuss such strategies in Section \ref{sec:UQ_processes} and Appendix \ref{sec:markov_rel_ent} below; see also the concrete examples in Section \ref{sec:examples}.

In some cases,  the right-hand-side of the UQ bound (\ref{eq:goal_info_ineq3}) is still difficult to compute, and so  it can be useful to first  bound  $R(\widetilde{P}|_{\mathcal{G}}\|P|_{\mathcal{G}})$  and then use the UQ bound from Theorem \ref{thm:obs_adap_info_ineq} directly; for example, see the analysis of the Vasicek model  in Section \ref{sec:Vasicek}.  The following corollary, obtained by applying Theorem \ref{thm:goal_info_ineq3} to $F=G$ and then reparametrizing the infimum, is useful for this purpose:
\begin{corollary}\label{lemma:gen_rel_ent_bootstrap}
Suppose $G:\Omega\to {\mathbb{R}}$, $G\in L^1(\widetilde{P})$ is $\mathcal{G}$-measurable for some sub sigma-algebra $\mathcal{G}\subset\mathcal{F}$, and $R(\widetilde{P}|_{\mathcal{G}}\|P|_{\mathcal{G}})\leq E_{\widetilde{P}}\left[G\right]$. Then
\begin{align}\label{eq:gen_rel_ent_bootstrap}
 R(\widetilde{P}|_{\mathcal{G}}\|P|_{\mathcal{G}})\leq& \inf_{\lambda>1}\left\{(\lambda-1)^{-1}\Lambda_{P}^{G}( \lambda)\right\}.
\end{align}
\end{corollary} 
  Note that the optimization in \eqref{eq:gen_rel_ent_bootstrap} is over $\lambda>1$, as opposed to over $c>0$ like in Theorems \ref{thm:obs_adap_info_ineq} and \ref{thm:goal_info_ineq3}.
\begin{remark}
 Sensitivity analysis of a QoI to perturbations within a parametric family, $P_\theta$, $\theta\in U\subset\mathbb{R}^d$, was also studied in  \cite{DKPP}.  The extension to goal-oriented relative entropy proceeds similarly, assuming one can find a density for the restricted measure with respect to some common dominating measure: $dP_\theta|_{\mathcal{G}}=p_\theta^{\mathcal{G}} d\mu|_{\mathcal{G}}$.  Examples of such densities   for various classes of Markov processes can be found in Appendix \ref{sec:markov_rel_ent}.    Once one has $p_\theta^{\mathcal{G}}$, the computation of sensitivity bounds proceeds as in \cite{DKPP}, and so we make no further comments on sensitivity analysis here.
\end{remark}
\begin{remark}
When $P$ is part of a tractable parametric family, $\{P_\theta\}_{\theta\in U}$,  the following technique can be used to improve the UQ bounds:  If one can carry out the procedure of Theorem \ref{thm:goal_info_ineq3} for each $\theta\in U$ then, by minimizing over $\theta$, one obtains
 \begin{align}\label{eq:goal_info_ineq_base_param}
\pm E_{\widetilde P}[F]\leq \inf_{c>0}\left\{\frac{1}{c} \inf_{\theta\in U}\log E_{P_\theta}\left[\exp\left(\pm c F+G_\theta\right)\right]\right\}.
\end{align}
See Figure \ref{fig:options_plot_r_drop} in Appendix \ref{app:options} for an example that employs this idea.
\end{remark}

\section{Goal-Oriented UQ for QoIs up to a Stopping Time}\label{sec:UQ_processes}
We now specialize the methods of Section \ref{sec:main_result}  to  the study of   path-space QoIs. Concrete examples can be found in Section \ref{sec:examples}, but we first discuss the guiding principles and common themes that underlie these applications. Section \ref{sec:path_space_setting}  sets up the framework for UQ on path-space.  Specifically,  we  consider QoIs up to a stopping time.  Many important problem types fit under this umbrella, such as discounted observables, the CDF of a stopping time, expectation of a stopping time, as well as the previously studied ergodic averages (see Table \ref{tab:QoIs}); Sections \ref{subsec:exp_stop_time} - \ref{sec:time_avg} discuss several of these problem types   in further detail.

\subsection{General Setting for UQ on Path Space up to a Stopping Time}\label{sec:path_space_setting}

 The general setting in which we  derive path-space UQ bounds is as follows:
\begin{assumption}\label{assump:time_dep_QoI}
Suppose:
\begin{enumerate}
\item  $(\Omega,\mathcal{F}_\infty,\{\mathcal{F}_t\}_{t\in\mathcal{T}})$ is a filtered probability space, where $\mathcal{T}=[0,\infty)$, or $\mathcal{T}=\mathbb{Z}_0$.  We define $\overline{\mathcal{T}}\equiv \mathcal{T}\cup\{\infty\}$.
\item $P$ and $\widetilde P$ are probability measures on $(\Omega,\mathcal{F}_\infty)$.
\item  $F:\overline{\mathcal{T}}\times\Omega\to\overline{\mathbb{R}}$, written $F_t(\omega)$, is progressively measurable (progressive), i.e., $F$ is measurable and $F|_{[0,t]\times\Omega}$ is $\mathcal{B}([0,t])\bigotimes \mathcal{F}_t$-measurable for all $t\in \mathcal{T}$ (intervals refer to subsets of $\mathcal{T}$).
\item $\tau:\Omega\to \overline{\mathcal{T}}$ is a $\mathcal{F}_t$-stopping time.
\end{enumerate}
\end{assumption}

We will derive UQ bounds for a process $F_t$, stopped at $\tau$, which we denote by $F_\tau$.   First recall:
\begin{lemma}\label{lemma:F_tau}
$F_\tau$ is $\mathcal{F}_{\tau}$-measurable, where
\begin{align}
\mathcal{F}_\tau=\{A\in\mathcal{F}_\infty: \forall t\in \mathcal{T}, A\cap\{\tau\leq t\}\in\mathcal{F}_t\}
\end{align}
is the filtration up to the stopping-time $\tau$, a sub sigma-algebra of $\mathcal{F}_\infty$. 
\end{lemma}

For a general stopped QoI, $F_\tau$,  the following (uncentered) UQ bounds immediately follow from Theorems  \ref{thm:obs_adap_info_ineq} and  \ref{thm:goal_info_ineq3}  with the choices $\mathcal{G}=\sigma(F_\tau)$ or $\mathcal{G}=\mathcal{F}_\tau$:
\begin{corollary}\label{cor:targeted_info_stopped}
Suppose $F_\tau\in L^1(\widetilde P)$.  Then 
\begin{align}\label{eq:goal_info_ineq2_stopped}
\pm E_{\widetilde P}[F_\tau]\leq& \inf_{c>0}\left\{\frac{1}{c} \Lambda_{P}^{F_\tau}(\pm c)+\frac{1}{c}R((F_\tau)_*\widetilde P\|(F_{\tau})_*P)\right\}\\
\leq&\inf_{c>0}\left\{\frac{1}{c} \Lambda_{P}^{F_\tau}(\pm c)+\frac{1}{c}R(\widetilde P|_{\mathcal{F}_\tau}\|P|_{\mathcal{F}_\tau})\right\}.\notag
\end{align}
If $R(\widetilde P|_{\mathcal{F}_\tau}\|P|_{\mathcal{F}_\tau})\leq E_{\widetilde{P}}[G]$ for some $\mathcal{F}_{\tau}$-measurable real-valued $G\in L^1(\widetilde{P})$ then
\begin{align}\label{eq:goal_info_ineq3_stopped}
\pm E_{\widetilde{P}}[F_\tau]\leq\inf_{c>0}\left\{\frac{1}{c} \log E_P\left[\exp\left(\pm c F_\tau+G\right)\right]\right\}.
\end{align}
The objective functions in \eqref{eq:goal_info_ineq2_stopped} and \eqref{eq:goal_info_ineq3_stopped} are quasiconvex
\end{corollary}
\begin{remark}
 This result shows that for stopped QoIs, $R( \widetilde{P}|_{\mathcal{F}_\tau}\|P|_{\mathcal{F}_\tau})$ is an appropriate goal-oriented relative entropy.  One can similarly obtain  centered variants of \eqref{eq:goal_info_ineq2_stopped} and  \eqref{eq:goal_info_ineq3_stopped} by working with $F_\tau-E_P[F_\tau]$ (we omit the details). As in Theorems \ref{thm:obs_adap_info_ineq} and \ref{thm:goal_info_ineq3}, one can change variables $\eta=1/c$ to obtain convex objective functions.
\end{remark}

Corollary \ref{cor:targeted_info_stopped} provides general-purpose UQ bounds for path-space QoIs up to a stopping time. In the following subsections, we discuss the nuances involved in deriving UQ bounds for several of the classes of  QoIs from Table \ref{tab:QoIs}.   In some cases we  directly apply  Corollary \ref{cor:targeted_info_stopped}, while others benefit from a slightly altered approach. Not all of our examples in  Section \ref{sec:examples} below will fit neatly into one of these  categories, but rather, they illustrate several important use-cases and strategies.

\subsection{Expectation of a Stopping Time}\label{subsec:exp_stop_time}
Here we consider the expectation of a stopping time, $\tau$.  For $\tau\in L^1(P)\cap L^1(\widetilde{P})$, (the centered variant of) \req{eq:goal_info_ineq2_stopped} in Corollary \ref{cor:targeted_info_stopped}  applied to $F_t=t$ yields
\begin{align}
\pm(E_{\widetilde{P}}[\tau]-E_P[\tau])\leq \inf_{c>0}\left\{\frac{1}{c}\Lambda_P^{\widehat{\tau}}(\pm c)+\frac{1}{c}R(\widetilde{P}|_{\mathcal{F}_\tau}\|P|_{\mathcal{F}_\tau})\right\}.
\end{align}
For unbounded $\tau\in L^1(P)\cap L^1(\widetilde{P})$, it is often useful to use Lemma \ref{lemma:lim_rel_ent} from Appendix \ref{sec:rel_ent} to obtain:
\begin{align}\label{eq:E_tau_bound0}
\pm( E_{\widetilde{P}}[\tau]-E_P[\tau])\leq& \inf_{c>0} \left\{\frac{1}{c}\Lambda_P^{\widehat{\tau}}(\pm c)+\frac{1}{c}\liminf_{n\to\infty}R(\widetilde{P}|_{\mathcal{F}_{\tau\wedge n}}\|P|_{\mathcal{F}_{\tau\wedge n}})\right\}.
\end{align}
To obtain a computable bound from \req{eq:E_tau_bound0}, first note that a relative entropy bound of the following form often follows from Girsanov's theorem (see Appendix \ref{sec:markov_rel_ent}):
\begin{align}\label{eq:rel_ent_lin_bound}
R(\widetilde P|_{\mathcal{F}_{\tau\wedge n}}\|P|_{\mathcal{F}_{\tau\wedge n}})\leq \eta_0+K E_{\widetilde{P}}[\tau\wedge n]
\end{align}
for some $\eta_0,K\in[0,\infty)$ and all $n\in\mathbb{Z}^+$. One can then use Theorem \ref{thm:goal_info_ineq3} with $F=\tau\wedge n$ and $G=\eta_0+K(\tau\wedge n)$. Taking $n\to\infty$ via the monotone and dominated convergence theorems and then  reparameterizing the infimum results in the following:
\begin{corollary}\label{cor:E_tau_bound}
Assume the relative entropy satisfies a bound of the form \req{eq:rel_ent_lin_bound} for all $n\in\mathbb{Z}^+$. Then
\begin{align}\label{E_tau_bound}
& -\inf_{c>0}\left\{(c+K)^{-1}\left(\Lambda_{P}^{\tau}(- c)+ \eta_0\right)\right\} \leq E_{\widetilde{P}}[\tau] \leq\inf_{\lambda>K}\left\{ (\lambda-K)^{-1}\left(\Lambda_{P}^{\tau}( \lambda)+ \eta_0\right)\right\}.
\end{align}
\end{corollary}
Again, note that the upper and lower bounds only involve the tractable model, $P$. We illustrate this technique in a concrete example in Appendix \ref{app:BM_Etau}.

\subsection{Time-Integral/Discounted  QoIs}\label{sec:UQ_time_integral}
Here we consider exponentially discounted QoIs.  In the language of Assumption \ref{assump:time_dep_QoI}, given a progressive process $f_s$ and some $\lambda>0$ we define
\begin{align}\label{eq_discounted}
F_t=\int_0^t f_s \lambda e^{-\lambda s}ds
\end{align}
 and will consider the stopped process $F_\tau$,  with stopping-time $\tau=\infty$.  Such QoIs are often used in control theory (see page 38 in \cite{kushner2013numerical}) and economics (see page 64 in \cite{page2013applications} and  page 147 in \cite{peter2014uncertainty}). Here we derive an alternative UQ bound to Corollary \ref{cor:targeted_info_stopped} that is applicable to such QoIs:
\begin{theorem}\label{thm:UQ_integral_QoI}
Let $\lambda>0$ and define $\pi(ds)= \lambda e^{-\lambda s}ds$. Suppose  $f:[0,\infty)\times\Omega\to\overline{\mathbb{R}}$ is progressively measurable and
\begin{align}\label{eq:discounted_tilde_P_assump}
\text{  either }f_s\geq 0\,\,\,\text{  for $\pi$-a.e. $s$ or  } \,\,\,E_{\widetilde{P}}\left[\int_0^\infty |f_s|d\pi\right]<\infty.
\end{align}
Then
\begin{align}\label{eq:UQ_integral_bound}
&\pm E_{\widetilde{P}}\left[\int_0^\infty f_s \lambda e^{-\lambda s}ds\right]\\
\leq&\int_0^\infty\inf_{c>0}\left\{\frac{1}{c}\Lambda_P^{{f_{s}}}(\pm c)+\frac{1}{c} R(\widetilde{P}|_{\mathcal{F}_{s}}\|P|_{\mathcal{F}_{s}})\right\} \lambda e^{-\lambda s}ds\notag\\
\leq&\inf_{c>0}\left\{\int_0^\infty\frac{1}{c}\Lambda_P^{{ f_{s}}}(\pm c)\lambda e^{-\lambda s}ds+\frac{1}{c} \int_{0}^\infty R(\widetilde{P}|_{\mathcal{F}_{s}}\|P|_{\mathcal{F}_{s}}) \lambda e^{-\lambda s}ds\right\}.\notag
\end{align}
\end{theorem}
\begin{remark}
Theorem \ref{thm:UQ_integral_QoI} can be generalized to other discounting measures, $\pi$, and other stopping times; see Appendix \ref{app:discounted}.
\end{remark}
\begin{proof}

First suppose  $E_{\widetilde{P}}\left[\int_0^\infty |f_s|d\pi\right]<\infty$. This allows us to use Fubini's theorem to compute
\begin{align}
& E_{\widetilde{P}}\left[\int_0^\infty f_s \lambda e^{-\lambda s}ds\right]=\int_{0}^\infty E_{\widetilde{P}}\left[ f_s\right]\pi(ds)
\end{align}
and implies $E_{\widetilde{P}}\left[ f_s\right]$ is finite for $\pi$-a.e. $s$. For each $s$, $f_{s}$ is $\mathcal{F}_{ s}$-measurable, hence Theorem \ref{thm:obs_adap_info_ineq} implies
\begin{align}
\pm E_{\widetilde{P}}\left[ f_{s}\right]\leq & \inf_{c>0}\left\{\frac{1}{c}\Lambda_P^{{  f_{s}}}(\pm c)+\frac{1}{c}R(\widetilde{P}|_{\mathcal{F}_{s }}\|P|_{\mathcal{F}_{s }})\right\}
\end{align}
for $\pi$-a.e. $s$. Integrating over $s$ gives the first inequality in \req{eq:UQ_integral_bound}.  The second follows by pulling the  infimum outside of the integral.

If   $f_s\geq 0$ then repeat the above calculations for $f_s^N\equiv f_s1_{f_s\leq N}+N1_{f_s>N}$ and take $N\to\infty$, using the dominated and monotone convergence theorems.
\end{proof}
 The last line in \req{eq:UQ_integral_bound} consists of two terms, a discounted moment generating function and a discounted relative entropy,
\begin{align}\label{eq:exp_discount_eta}
D\equiv\int_0^\infty R(\widetilde{P}|_{\mathcal{F}_{s}}\|P|_{\mathcal{F}_{s}})\lambda e^{-\lambda s}ds.
 \end{align}
 \req{eq:exp_discount_eta}  is the same information-theoretic quantity  that was defined in \cite{peter2014uncertainty} (see page 147), where it was proposed as a measure of model uncertainty for control problems. Theorem \ref{thm:UQ_integral_QoI} provides a rigorous justification for its use in UQ for exponentially discounted QoIs.

  \begin{remark}\label{remark:integral_QoI}
$\pi(dx)=\lambda e^{-\lambda s}ds$ is a probability measure, and so one one could alternatively apply Theorem \ref{thm:obs_adap_info_ineq} to the product measures $\pi\times P$ and $\pi\times \widetilde{P}$. However,  Jensen's inequality implies that the bound \req{eq:UQ_integral_bound} is tighter. Alternatively, one might  attempt to use Corollary \ref{cor:targeted_info_stopped} in place of Theorem \ref{thm:UQ_integral_QoI}, in which case one generally finds that the relative-entropy term, $R(\widetilde{P}|_{\mathcal{F}_\infty}\|{P}|_{\mathcal{F}_\infty})$,  is infinite and so the corresponding UQ bound is trivial and uninformative.  On the other hand, the bound \req{eq:UQ_integral_bound} is often nontrivial; see the example in Section \ref{example:control} below.
\end{remark}

\subsection{Time-Averages}\label{sec:time_avg}
Finally, we remark that  ergodic averages are a special case of  the current framework: Let $T>0$, $F_t=t^{-1}\int_0^{t} f_sds$, and $\tau=T$ (one can  similarly treat the discrete-time case).  $F_\tau$ is $\mathcal{F}_T$-measurable and so, if $F_\tau\in L^1(P)\cap L^1(\widetilde{P})$, one can apply   Corollary \ref{cor:targeted_info_stopped} to $\widehat{F}_\tau$.  Reparameterizing the infumum $c\to cT$, one finds that the bound is  the same as \req{eq:ergodic_avg}, which was previously derived in  \cite{DKPP} and used in \cite{doi:10.1137/19M1237429}. As we will demonstrate in the example in Section \ref{example:non-rev} below, one can often use Theorem \ref{thm:goal_info_ineq3} to obtain tighter and more general UQ bounds on ergodic averages than those obtained from \req{eq:ergodic_avg}.

\section{Examples}\label{sec:examples}
We now apply the goal-oriented UQ methods developed above to several examples:
\begin{enumerate}
\item Hitting times for perturbations of Brownian motion in Section \ref{sec:BM_const_drift}.
\item UQ bounds for invariant measures  of non-reversible SDEs in Section \ref{example:non-rev}.
\item Expected cost in stochastic control systems in Section \ref{example:control}.
\item Long-time behavior of semi-Markov perturbations of an $M/M/\infty$ queue in Section \ref{example:birth_death}.
\item Pricing of American put options with variable interest rate in Section \ref{sec:options}.
\end{enumerate}

\subsection{Perturbed Brownian Motion}\label{sec:BM_const_drift}
We first  illustrate the use of our method, and several of its features, with a simple example wherein many computations can be done explicitly and exactly; the example consists of  perturbations to Brownian-motion-with-constant-drift. Specifically, take  the baseline model to be the distribution on path-space, $C([0,\infty),\mathbb{R})$, of a $\mathbb{R}$-valued Wiener process  (Brownian motion), $W_t$, with constant drift, $\mu>0$, starting from $0$, i.e., the distribution of the following process:
\begin{flalign}
\text{\bf Baseline Model:} && X_t=\mu t+W_t&& &&
\end{flalign}
The alternative model will be the distribution of the solution to a SDE of the form:
\begin{flalign}\label{eq:BM_alt}
\text{\bf  Alternative Models:} && d\widetilde X_t=(\mu \!+\!\beta(t,\widetilde{X}_t,\widetilde{Y}_t,\widetilde{Z}))dt+d\widetilde{W}_t,\,\,\,\,\,\widetilde X_0=0, && &&
\end{flalign}
where $\widetilde{W}_t$ is also a Wiener process. We allow the drift perturbation, $\beta$, to depend on an additional $\mathbb{R}^k$-valued process, $\widetilde{Y}_t$,  and also on independent external data, $\widetilde{Z}$ (see Appendix \ref{sec:rel_ent_SDE} for further discussion of the type of perturbations we have in mind). More specific assumptions on $\beta$ will be stated later.  We note that the methods developed here do not allow one to perturb the diffusion, as the corresponding relative entropy is infinite due to lack of absolute continuity (see, e.g., page 80 in \cite{freidlin2016functional}).

In Section \ref{sec:BM_dist} we will study the cumulative distribution function of $\tau_a$, the level-$a$ hitting time: 
\begin{flalign}\label{eq:BM_QoI1}
\text{\bf  QoI 1:} && \widetilde{P}^0(\tau_a\leq T),\,\,\,T>0,\text{ where }\tau_a[y]\equiv\inf\{t: y_t=a\}. && 
\end{flalign}
Note that $\tau_a$ is a stopping time on path-space, $C([0,\infty),\mathbb{R})$. $\widetilde{P}^0$ denotes the distribution of $\widetilde{X}_t$ on path-space and similarly for $P^0$; the superscript $0$ indicates that the initial condition is $0$.

We will also  derive UQ bounds for the expected level-$a$ hitting time,
\begin{flalign}\label{eq:BM_QoI2}
\text{\bf  QoI 2:} &&E_{\widetilde{P}^0}[\tau_a],\,\,\,a>0. && 
\end{flalign}
This second QoI is simpler, as we can use the general result of Section \ref{subsec:exp_stop_time}. Therefore we relegate the details to  Appendix \ref{app:BM_Etau}.

\subsubsection{Hitting Time Distribution: Benefits of Being Goal-Oriented}\label{sec:BM_dist}

We use the QoI (\ref{eq:BM_QoI1}) to illustrate the benefits of using a goal-oriented relative entropy, versus a non-goal-oriented counterpart. Here we specialize to perturbed SDEs,  \req{eq:BM_alt}, whose distributions, $\widetilde{P}^0$, satisfy the following.
\begin{flalign}\label{eq:BM_Ambiguity1}
\text{\bf  Alternative Models:} && &\text{ Perturbations by drifts, $\beta$,  such that, for some $\alpha>0$:}&& && \\
&& &R(\widetilde{P}^0|_{\mathcal{F}_{\tau_a\wedge T}}\|{P}^0|_{\mathcal{F}_{\tau_a\wedge T}})\leq \frac{1}{2}\alpha^2E_{\widetilde{P}^0}[\tau_a\wedge T].&& &&\notag
\end{flalign}

For example, \req{eq:BM_Ambiguity1} holds if $\|\beta\|_\infty\leq\alpha$ (assuming that the  models  satisfy the assumptions required to obtain the relative entropy via Girsanov's theorem; see  Appendix \ref{sec:rel_ent_SDE}  for details).

  The QoI, $F\equiv1_{\tau_a\leq T}$, is $\mathcal{F}_{\tau_a\wedge T}$-measurable, and so $R(\widetilde{P}^0|_{\mathcal{F}_{\tau_a\wedge T}}\|\widetilde{P}^0|_{\mathcal{F}_{\tau_a\wedge T}})$ is an appropriate goal-oriented relative entropy; the bound \req{eq:BM_Ambiguity1} implies that we can take $\mathcal{G}=\mathcal{F}_{\tau_a\wedge T}$ and $G=\frac{1}{2}\alpha^2(\tau_a\wedge T)$ in Theorem \ref{thm:goal_info_ineq3}. This gives the following goal-oriented UQ bound:
\begin{align}\label{eq:BM_goal_oriented}
\pm \widetilde{P}^0(\tau_a\leq T)\leq \inf_{c>0}\left\{\frac{1}{c}\log E_{P^0}\left[\exp\left(\pm c1_{\tau_a\leq T}+\frac{1}{2}\alpha^2(\tau_a\wedge T)\right)\right]\right\}.
\end{align}

\begin{figure}
 \begin{minipage}[b]{0.5\linewidth}
\centerline{\includegraphics[height=6cm]{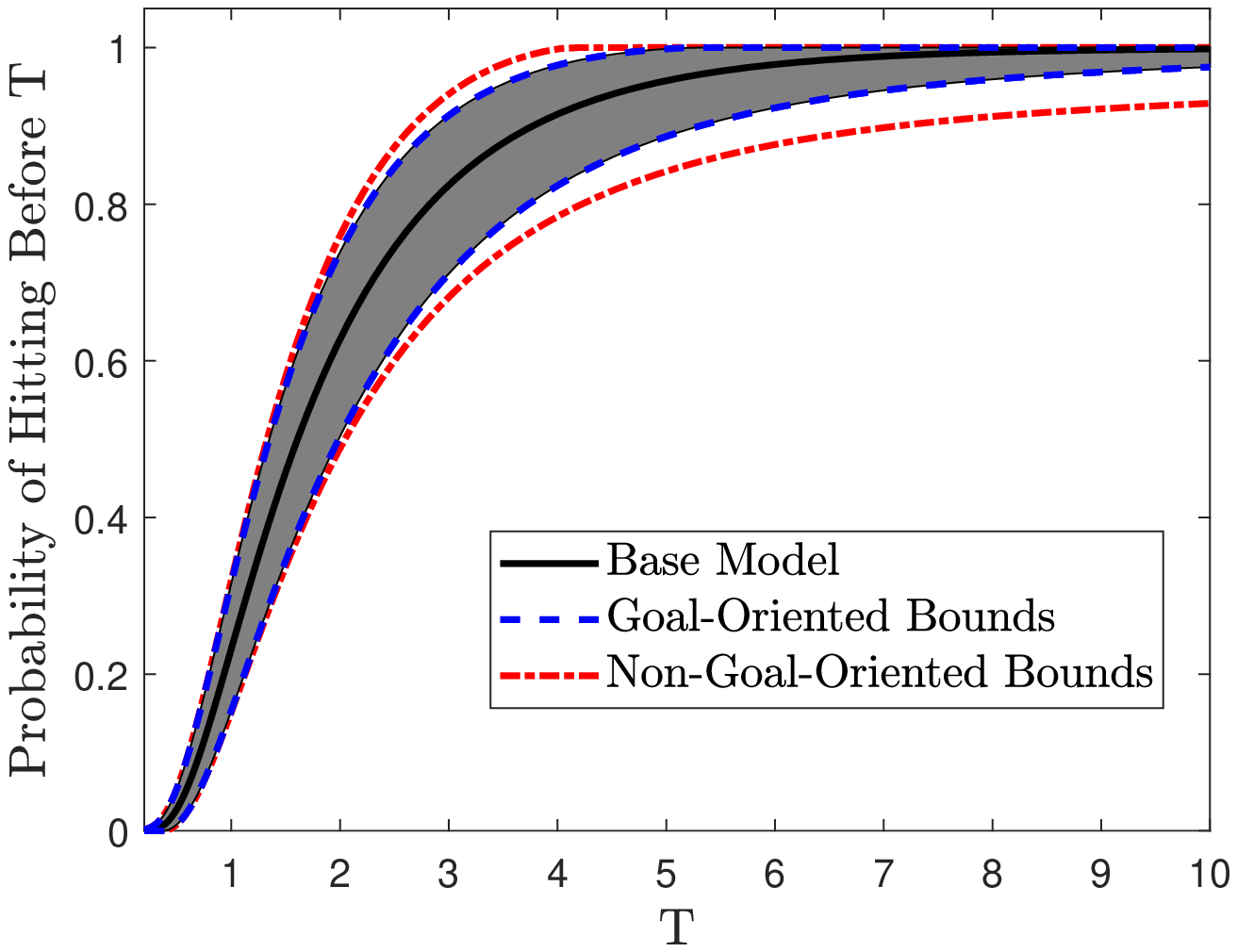}}
 \end{minipage}
 \begin{minipage}[b]{0.5\linewidth}
\centerline{\includegraphics[height=6cm]{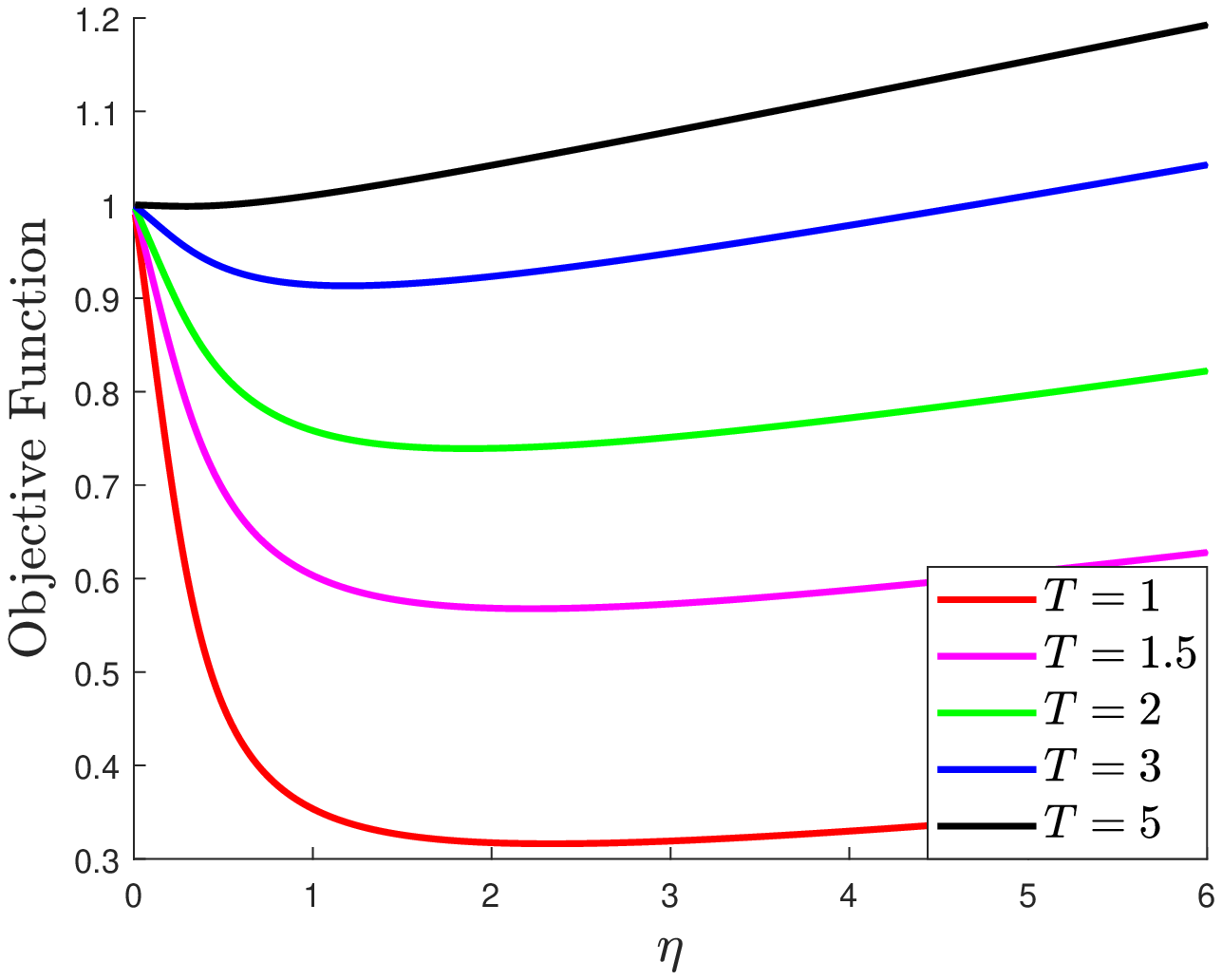}}
 \end{minipage}
\caption{Left: Bound on the CDF of $\tau_a$, the level-$a$ hitting time, for a perturbation of  Brownian-motion-with-constant-drift. The solid black line shows the distribution under the baseline model. The red curves (dot-dashed, \req{eq:BM_dist_non_goal}) are obtained from a non-goal-oriented relative entropy, while the tighter bounds, in blue (dashed, \req{eq:BM_dist_goal}), were obtained by our method with an appropriate goal-oriented relative entropy. Right: The (convex) objective function from \eqref{eq:goal_info_ineq3_convex} for the goal-oriented upper UQ bound (i.e., corresponding to the  upper blue dashed curve on the left). }\label{fig:bm_dist}
 \end{figure}

For comparison, we also compute the  non-goal-oriented bound that follows from the UQ information inequality \eqref{goal_oriented_bound}. The QoI is a function of the path up to time $T$, so we apply this proposition to  the relative entropy $R(\widetilde{P}^0_T\|P^0_T)$ (the subscript $T$ denotes taking the distribution on path-space $C([0,T],\mathbb{R})$).
\begin{remark}
One can  attempt to use the full distributions on $C([0,\infty),\mathbb{R})$ in \req{goal_oriented_bound}, but that only results in trivial bounds, as the maximum of  $R(\widetilde{P}^0\|P^0)$ over the collection of alternative models (\ref{eq:BM_Ambiguity1}) is infinite; this can be seen by taking a constant perturbation $\beta=\alpha$.
 \end{remark}
Repeating these computations under the assumption $\|\beta\|_\infty\leq\alpha$, but using the non-goal-oriented relative entropy  bound $R(\widetilde{P}|_{\mathcal{F}_T}\|P|_{\mathcal{F}_T})\leq \alpha^2T/2$ (once again using Girsanov's theorem),  we obtain the UQ bound
\begin{align}\label{eq:BM_dist_non_goal}
\pm \widetilde{P}^0(\tau_a\leq T)\leq&  \inf_{c>0}\left\{\frac{1}{c}\log E_{P^0}\left[\exp\left(\pm c1_{\tau_a\leq T}\right)\right]+\frac{1}{c}\frac{\alpha^2 T}{2}\right\}.
\end{align}
Note that the goal-oriented bound on the right hand side of \req{eq:BM_goal_oriented} is tighter than  the non-goal-oriented bound \eqref{eq:BM_dist_non_goal}, as can be seen by bounding $\frac{1}{2}\alpha^2(\tau_a\wedge T)\leq \frac{1}{2}\alpha^2  T$ in the exponent.

The $P^0$-expectations in \eqref{eq:BM_goal_oriented} and \eqref{eq:BM_dist_non_goal} can be computed using the known formula for the distribution of $\tau_a$ under the baseline model (see page 196 in \cite{karatzas2014brownian} and also Appendix \ref{app:BM_hitting_times}).  Assuming that $a>0$ and $\mu>0$, the goal-oriented bound is
\begin{align}\label{eq:BM_dist_goal}
&\pm \widetilde{P}^0(\tau_a\leq T)\leq \inf_{c>0}\left\{\frac{1}{c}\log\left[ \frac{a}{\sqrt{2\pi}} \int_0^\infty e^{\pm c1_{z\leq T}+\frac{1}{2}\alpha^2(z\wedge T)}e^{-(a-\mu z)^2/(2z)} z^{-3/2}dz\right]\right\}.
\end{align}
Similarly, one can  compute the non-goal-oriented bound, \req{eq:BM_dist_non_goal}, and the probability under the baseline model.

Figure \ref{fig:bm_dist} shows numerical results comparing these bounds, with parameter values $\mu=1$, $a=2$, $\alpha=0.2$;  we compute the integrals numerically using a quadrature method and, as discussed in Remark \ref{remark:unimodal}, we solve the 1-D optimization problems numerically.  The black curve shows the distribution of $\tau_a$ under the baseline model. The blue curves show the goal-oriented bounds \req{eq:BM_dist_goal}; these constrain $\widetilde{P}^0(\tau_a\leq T)$ to be within the gray region. The red curves show the non-goal-oriented bounds \eqref{eq:BM_dist_non_goal}; we  see that the non-goal-oriented bounds significantly overestimate the uncertainty, especially as $P^0(\tau_a\leq T)$ approaches $1$.

    \subsection{UQ for Time-Averages and Invariant Measures of Non-Reversible SDEs}\label{example:non-rev}

In this example we show how Theorem \ref{thm:goal_info_ineq3} can be combined with   the functional inequality methods from \cite{doi:10.1137/19M1237429} to obtain improved UQ bounds for invariant measures of non-reversible SDEs.  Specifically, consider the following class of SDEs on $\mathbb{R}^d$:
\begin{flalign}\label{eq:Langevin_baseline}
\text{\bf  Baseline Model:} &&dX_t=(-\nabla V( X_t) +a(X_t))dt +\sqrt{2}dW_t,\,\,\, X_0\sim \mu^*=\frac{1}{Z}e^{-V(x)}dx, && 
\end{flalign}
where the non-gradient portion of the drift, $a$, satisfies
\begin{align}\label{eq:a_div_condition}
\divergence(e^{-V}a)=0.
\end{align}
If $a=0$ then \req{eq:Langevin_baseline} is an overdamped Langevin equation with potential $V$.  In general,  \req{eq:Langevin_baseline} is non-reversible but \req{eq:a_div_condition} ensures that $\mu^*$ remains  the invariant measure for $a\neq 0$; such non-reversible terms are used for variance reduction in  Langevin samplers  \cite{Rey_Bellet_2015}.

The alternative model will have a general non-reversible drift:
\begin{flalign}\label{eq:Langevin_alt}
\text{\bf   Alternative Models:} &&d\widetilde{X}_t=(-\nabla  V(\widetilde{X}_t)+ b(\widetilde{X}_t)) dt +\sqrt{2}d\widetilde{W}_t,\,\,\,\widetilde X_0\sim \widetilde\mu. && &&
\end{flalign}
Denote their respective distributions on path-space, $C([0,\infty),\mathbb{R}^d)$, by $P$ and $\widetilde{P}$. We will obtain UQ bounds for ergodic averages of some observable $f$:
\begin{flalign}
\text{\bf    QoI:} &&F_T\equiv\frac{1}{T}\int_0^T f(x_t) dt. && 
\end{flalign}
\begin{remark}
While the invariant measure of the baseline process is known explicitly (up to a normalization factor; see \req{eq:Langevin_baseline}), the invariant measure of the alternative model (\ref{eq:Langevin_alt}) is {\em not} known for  general $b$; physically, such non-reversible drifts model external forces.   A general drift can be split into $-\nabla V$ and $b$ via a Helmholz decomposition (i.e., $b$ is the divergence-free component).  Such a decomposition can inform the choice of $V$ in the baseline model.
\end{remark}

Assuming that Girsanov's theorem applies  (see Appendix \ref{sec:rel_ent_SDE}), we have the following expression for the relative entropy
\begin{align}
R(\widetilde{P}|_{\mathcal{F}_T}\|P|_{\mathcal{F}_T})=E_{\widetilde{P}}[G_T],\,\,\,\, G_T\equiv R(\widetilde{\mu}\|\mu^*)+\frac{1}{4}\int_0^T\|b(x_t)-a(x_t)\|^2dt.
\end{align}
The quantity $G_T$ is $\mathcal{F}_T$-measurable and if $R(\widetilde{\mu}\|\mu^*)$ is finite and $b-a$ grows sufficiently slowly at infinity  then $G_T\in L^1(\widetilde{P})$ and hence Theorem \ref{thm:goal_info_ineq3} implies
\begin{align}\label{eq:imp_ergodic_UQ}
\pm E_{\widetilde{P}}[F_T]\leq \inf_{c>0}\left\{\frac{1}{cT}\log E_P[\exp(\pm c TF_T+G_T)]\right\},
\end{align}
where we re-indexed $c\to cT$.  

We can now use the functional inequality approach of \cite{doi:10.1137/19M1237429} to bound \req{eq:imp_ergodic_UQ}. First write
\begin{align}
&\pm c TF_T+G_T=\int_0^T h_c(x_t)dt\equiv TH_{c,T},\\
& h_c(x)\equiv R(\widetilde{\mu}\|\mu^*)/T\pm cf(x)+\frac{1}{4}\|b(x)-a(x)\|^2.\notag
\end{align}
Utilizing a connection with the Feynman Kac semigroup, the cumulant generating function in (\ref{eq:imp_ergodic_UQ}) can be bounded as follows  (see Theorem 1 in \cite{WU2000435} or Proposition 6 in  \cite{doi:10.1137/19M1237429}):
\begin{align}
&\frac{1}{T}\Lambda_{P}^{{H}_{c,T}}( T)\leq \kappa(h_c),\\
&\kappa(h_c)\equiv\sup\left\{\int g A[g] d\mu^*\!+\!\!\int h_c|g|^2d\mu^*:g\in D(A,\mathbb{R}),\,\|g\|_{L^2(\mu^*)}=1\right\},\notag
\end{align}
where $A$ denotes the generator of the baseline SDE (\ref{eq:Langevin_baseline}). Note that $\kappa$ depends only on the symmetric part of $A$, while the non-reversible component of the drift, $a$, only contributes to the antisymmetric part of $A$.  Therefore, it suffices to consider the reversible case $a=0$ when bounding $\kappa(h)$. (The precise form of $a$ will impact the quality of the final UQ bound; see Remark \ref{remark:a_param}).

  $\kappa(h_c)$ can be bounded using functional inequalities for $A$ (e.g., Poincar{\'e}, log-Sobolev, Liapunov functions).    For example, a log-Sobolev inequality for $A$, 
\begin{align}
\int g^2 \log(g^2)d\mu^{*}\leq -\beta\int gA[g]\,d\mu^{*}\, \,\text{ for all }\,g\in D(A,\mathbb{R}),
\end{align}
where $\beta>0$, implies
\begin{align}\label{eq:Lambda_log_sob}
\frac{1}{T}\Lambda_{P}^{{H}_{c,T}}( T)\leq \beta^{-1}\log\left(\int \exp(\beta h_c(x))\mu^*(dx)\right)
\end{align}
(see Corollary 4 in \cite{WU2000435} or Proposition 7 in  \cite{doi:10.1137/19M1237429}) and hence
\begin{align}\label{eq:imp_ergodic_UQ2}
\pm E_{\widetilde{P}}[F_T]\leq \inf_{c>0}\left\{\frac{1}{c\beta} \log\left(\int \exp(\beta h_c(x))\mu^*(dx)\right)\right\}.
\end{align}
If the alternative process is started in a known initial distribution, $\widetilde{\mu}$, for which $R(\widetilde{\mu}\|\mu^*)$ is finite and can be computed (or bounded) then the right-hand-side of \req{eq:imp_ergodic_UQ2} only involves the baseline process and other known quantities.  It therefore provides (in principle) computable UQ bounds on the expectations of  time-averages at finite-times.   Alternatively, by taking $T\to\infty$ the dependence on the initial distribution in \eqref{eq:imp_ergodic_UQ} disappears and we can obtain UQ bounds on the invariant measure of the alternative process.  In the remainder of this subsection we explore this second direction.

By starting the alternative process in an invariant measure $\widetilde{\mu}=\widetilde{\mu}^*$  and taking $T\to\infty$ in \req{eq:imp_ergodic_UQ2}  we obtain the following UQ bounds on the invariant measure of the alternative model:
\begin{align}\label{eq:non_rev_UQ}
&\pm\int f(x)\widetilde{\mu}^*(dx)\leq \inf_{c>0}\left\{\frac{1}{c\beta} \log\left(\int \exp\left(\pm c\beta f(x)+\frac{\beta}{4}\|b(x)-a(x)\|^2\right)\mu^*(dx)\right)\right\}.
\end{align}
Note that the  bounds (\ref{eq:non_rev_UQ}) only involve expectations under the baseline invariant measure, $\mu^*$, and are   finite for appropriate unbounded $f$ and $b$. The ability to treat unbounded perturbations is due to our use of Theorem \ref{thm:goal_info_ineq3}, as opposed to  \req{eq:ergodic_avg}, and constitutes  a significant improvement on the results in  \cite{DKPP,doi:10.1137/19M1237429}, which are  tractable only when $b-a$ is bounded.   In Figure \ref{fig:non_rev_UQ} we compare \req{eq:non_rev_UQ}  with the UQ bound \eqref{eq:ergodic_avg} from \cite{DKPP},  where in the latter we employ the relative entropy rate bound 
\begin{align}\label{eq:non_rev_rel_ent_bound}
H(\widetilde{P}\|P)\leq\|b-a\|^2_\infty/4,
\end{align}
i.e., we use the strategy described in Remark \ref{remark:old_strat}.  Even in the case where $b-a$ is bounded, the new result (\ref{eq:non_rev_UQ}) obtained using Theorem \ref{thm:goal_info_ineq3}  is tighter than the bounds obtained from the earlier result \eqref{eq:ergodic_avg}.

\begin{figure}
\centering \includegraphics[height=6cm]{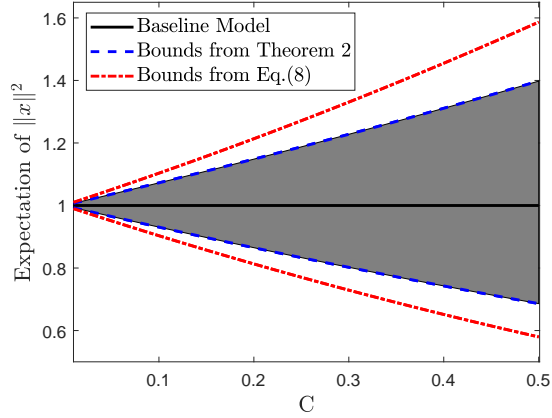}
 \caption{Baseline model: $V(x)=\|x\|^2$, $a=0$ ($d=2$);  $D^2V\geq 2I$ and so  a log-Sobolev inequality holds with $\beta=1$ (see \cite{BakryEmery}).   Perturbation: $b(x)=C(-\sin(x_2),\cos(x_1))$. Observable: $f(x)=\|x\|^2$. The baseline QoI, $\int fd\mu^*$, is shown in the horizontal black line.  The bounds developed in Theorem \ref{thm:goal_info_ineq3}  constrain the mean under the alternative model to the gray region (see  \req{eq:non_rev_UQ}).  These bounds are tighter than those obtained from \req{eq:ergodic_avg} (see  \cite{DKPP}) together with \req{eq:non_rev_rel_ent_bound} (shown in red).}\label{fig:non_rev_UQ}
 \end{figure}

\begin{remark}\label{remark:a_param}
The left hand side in \req{eq:non_rev_UQ} is independent of the choice of $a$ in the baseline model.  Therefore,  one can view $a$ in \req{eq:non_rev_UQ} as a parameter to be optimized over (minimizing over the set of $a$'s satisfying  \req{eq:a_div_condition}).
\end{remark}

\subsection{Discounted Quadratic Observable for Non-Gaussian Perturbations of a Gaussian Process}\label{example:control}
Here we consider a $\mathbb{R}^n$-valued progressively measurable process, $X_t$, that has a Gaussian distribution for every $t$:
\begin{flalign}\label{eq:gaussian_base}
\text{\bf     Baseline Model:} &&(X_t)_*P=N(\mu_t,\Sigma_t),\,\,\,t\geq 0. && &&
\end{flalign}
The mean and variance of the baseline model are considered to be known (i.e., obtained by some approximation or measurement).  We will compute UQ bounds on non-Gaussian perturbations of \eqref{eq:gaussian_base}. We assume  $\Sigma_t$ is (strictly) positive-definite and study the discounted QoI:
\begin{flalign} 
\text{\bf     QoI:} &&F\equiv\int_{\mathcal{T}} \left( \frac{1}{2}\widehat{X}_s^TC_s\widehat{X}_s+d_s^T \widehat{X}_s\right)\pi(ds), && 
\end{flalign}
where $C_s$ is a positive-definite $n\times n$-matrix-valued function, $d_s$ is $\mathbb{R}^n$-valued, and $\pi$ is a probability measure. Without loss of generality (i.e., neglecting an overall constant factor and after a redefinition of $d_s$), we have chosen to express $F$ in terms of the centered process $\widehat{X}_s=X_s-\mu_s$.  The techniques involved here are those of Section \ref{sec:UQ_time_integral}.

At this point, we are being purposefully vague about the nature of the alternative model and even the details of the baseline model, as a large portion of the following computation is independent of these details. We also emphasize that the alternative models do {\em not} need to be Gaussian.   In Section \ref{sec:discounted_control} we will make both base and alternative models concrete, as we study an application to control theory.

Assuming that the conditions of Theorem \ref{thm:UQ_integral_QoI2} of Appendix \ref{app:discounted} are met, we obtain the following UQ bound (note that the appropriate sub sigma-algebras are $\mathcal{G}=\mathcal{F}_s$):
\begin{align}\label{eq:discount_ex_2}
&\pm E_{\widetilde{P}}\left[F\right]\leq\int_{\mathcal{T}}\inf_{c>0}\left\{\frac{1}{c}\Lambda_P^{{ f_{s}}}(\pm c)+\frac{1}{c}R(\widetilde{P}|_{\mathcal{F}_s}\|P|_{\mathcal{F}_s})\right\}\pi(ds),\\
&f_s\equiv \frac{1}{2} \widehat{X}_s^TC_s\widehat{X}_s+d_s^T \widehat{X}_s.\notag
\end{align}

$X_t$ is Gaussian under $P$, so we can compute 
\begin{align}\label{eq:Lambda_gaussian}
&\Lambda_P^{{ f_{s}}}( c)=\log\left(\!((2\pi)^n|\det(\Sigma_s)|)^{-1/2}\!\! \int e^{cd^T_sz-z^T(\Sigma^{-1}_s-cC_s)z/2}dz\right),\\
&E_P[F]=\int_{\mathcal{T}} \frac{1}{2}\tr(\Sigma_sC_s)\pi(ds).\notag
\end{align}
For $c$ small enough, $\Sigma_s^{-1}-cC_s$ is positive-definite.  More specifically, $C_s$ is positive-definite, so we can compute a Cholesky factorization  $C_s=M_s^TM_s$ with $M_s$ nonsingular.  $\Sigma_s^{-1}-cC_s$ is positive-definite  if and only if $y^T (M_s^{-1})^T\Sigma_s^{-1} M_s^{-1}y>cy^Ty$ for all nonzero $y$, which is true if and only if $c<1/\|M_s\Sigma_sM_s^T\|$ ($\ell^2$-matrix norm).

For $c<1/\|M_s\Sigma_sM_s^T\|\equiv c^*_s$, the integral in \req{eq:Lambda_gaussian} can be computed in terms of the moment generating function of a Gaussian with covariance $(\Sigma^{-1}_s-cC_s)^{-1}$, resulting in the UQ bound
\begin{align}\label{eq:Gaussian_UQ}
&-\!\int_{\mathcal{T}}\inf_{c>0}\bigg\{\!\frac{1}{c}\Lambda_P^{f_s}(-c)+\frac{\eta_s}{c}\bigg\}\pi(ds)\leq E_{\widetilde{P}}\left[F\right]
\leq\int_{\mathcal{T}}\inf_{0<c<c^*_s}\bigg\{\!\frac{1}{c}\Lambda_P^{f_s}(c)+\frac{\eta_s}{c}\bigg\}\pi(ds),
\end{align}
where  $\eta_s\equiv R(\widetilde{P}|_{\mathcal{F}_s}\|P|_{\mathcal{F}_s})$ and for the above ranges of $c$ we have
\begin{align}
&\Lambda_P^{{ f_{s}}}( \pm c)=-\frac{1}{2}\log\left(|\det\left(I\mp c\Sigma_sC_s\right)|\right)+\frac{1}{2}c^2d_s^T(I\mp c\Sigma_sC_s)^{-1}\Sigma_sd_s.\notag
\end{align}

To make these bounds computable, one needs a bound on the relative entropy; this requires more specificity regarding the nature of the family of alternative models. We study an application to control theory in the following subsection.

\subsubsection{Linear-Quadratic Stochastic Control}\label{sec:discounted_control}
Here, \req{eq:Gaussian_UQ} will be used to study robustness for linear-quadratic stochastic control (robustness under nonlinear perturbations).  Specifically, suppose one is interested in controlling some nonlinear system,
\begin{align}\label{eq:nonlin_control}
d\widetilde{X}_t=(B\widetilde{X}_t+\sigma\beta(t,\widetilde{X}_t)+Du_t)dt+\sigma dW_t,\,\,\, \widetilde{X}_0\sim N(0,\Sigma_0),
\end{align}
where  $B,D,\sigma$ are $n\times n$ matrices, and $W$ is a $\mathbb{R}^n$-valued Wiener process. We write the non-linear term with an explicit factor of $\sigma$ to simplify the use of Girsanov's theorem later on. The control variable is denoted by $u_t$ and $\widetilde{X}_t$ is the state; we take $\widetilde{X}_t$ to be an observable quantity  that can be used for feedback.

 The perturbation $\beta$ may be unknown and even when it is known, optimal control of \req{eq:nonlin_control} is a difficult problem, both analytically and numerically \cite{kushner2013numerical}. Therefore, one option is to consider the linear approximation
\begin{align}\label{eq:lin_control}
dX_t=(BX_t+Du_t)dt+\sigma dW_t, \,\,\, {X}_0\sim N(0,\Sigma_0),
\end{align}
obtain an explicit formula for the optimal feedback control (under a cost functional to be specified below) for \req{eq:lin_control}, and use that same feedback function to (suboptimally) control the original system \req{eq:nonlin_control}.  To evaluate the performance, one must bound the cost functional when the control for the linearized system is used on the nonlinear system.

To make the above precise, we must first specify the cost functional, which we take to be an exponentially discounted quadratic cost:
\begin{align}\label{eq:quadratic_cost}
\mathcal{C}\equiv E\left[\int_0^\infty \frac{1}{2}\left( {X}_s^TQX_s+u_s^TRu_s\right)\lambda e^{-\lambda s}ds\right],
\end{align}
(and similarly with $X_t$ replaced by $\widetilde{X}_t$ for the nonlinear system). Here, $Q$ and $R$ are positive-definite $n\times n$-matrices.

 The optimal feedback control for the the linear system \req{eq:lin_control} with cost \req{eq:quadratic_cost} is  given by $u_t=K_\lambda X_t$ where $K_\lambda$ is obtained as follows (see \cite{anderson2007optimal,BIJL2019471} for details): Define $B_\lambda\equiv B-\frac{\lambda }{2}I$ and solve   the following algebraic Riccati equation for $Y_\lambda$:
 \begin{align}
 B_\lambda^T Y_\lambda+Y_\lambda B_\lambda+Q-Y_\lambda DR^{-1}D^TY_\lambda=0.
 \end{align}
 The optimal control is then
 \begin{align}\label{eq:optimal_control}
 u_t=K_\lambda X_t,\,\,\, K_\lambda\equiv-R^{-1}D^TY_\lambda.
 \end{align}
 
Having obtained \req{eq:optimal_control}, we can finally fit the above problem into our UQ framework:
\begin{flalign} 
\text{\bf    Baseline Model:} && &dX_t=\!A_\lambda X_tdt+\sigma dW_t,\,A_\lambda\!\equiv B\!+\!DK_\lambda,\,X_0\!\sim\! N(0,\Sigma_0), && \label{eq:control_SDE}\\
\text{\bf     Alternative Models:} && &d\widetilde{X}_t=(A_\lambda\widetilde{X}_t+\sigma\beta(t,\widetilde{X}_t))dt+\sigma dW_t,\,\widetilde{X}_0\sim N(0,\Sigma_0), &&\label{eq:Control_Ambiguity}\\
\text{\bf     QoI:} && &F\equiv\int_0^\infty \frac{1}{2}x_s^TC_\lambda x_s\lambda e^{-\lambda s}ds,\,\,\, C_\lambda\equiv Q+K_\lambda^TRK_\lambda, &&\label{eq:control_QoI}\\
&& & \text{where $x\in C([0,\infty),\mathbb{R}^n)$ is the path of the state variable.}&&\notag
\end{flalign}
Note that, by Girsanov's theorem, the collection of alternative models contains all bounded-drift-perturbations of the baseline model with $\|\beta\|_\infty\leq\alpha$.

The UQ bounds \eqref{eq:Gaussian_UQ} applied to the above problem are  bounds on the cost of controlling the nonlinear system \eqref{eq:nonlin_control} by using the  feedback control function that was derived to optimally control the linear system \eqref{eq:lin_control}. We can make the bounds explicit as follows:  \req{eq:control_SDE} has the solution
\begin{align}
X_t=e^{tA_\lambda}\left(X_{0}+\int_0^te^{-sA_\lambda}\sigma dW_s\right).
\end{align}
  In particular, $X_t$ is Gaussian with mean $0$ and covariance
\begin{align}
\Sigma_t^{\alpha\beta}=(e^{tA_\lambda})^\alpha_i (e^{tA_\lambda})^\beta_j \Sigma^{ij}_0
+(e^{tA_\lambda})^\alpha_i(e^{tA_\lambda})^\beta_j\int_0^t (e^{-sA_\lambda}\sigma)^i_k\delta^{k\ell}(e^{-sA_\lambda}\sigma)^{j}_\ell ds.
\end{align}
\begin{figure}
\centerline{\includegraphics[height=6cm]{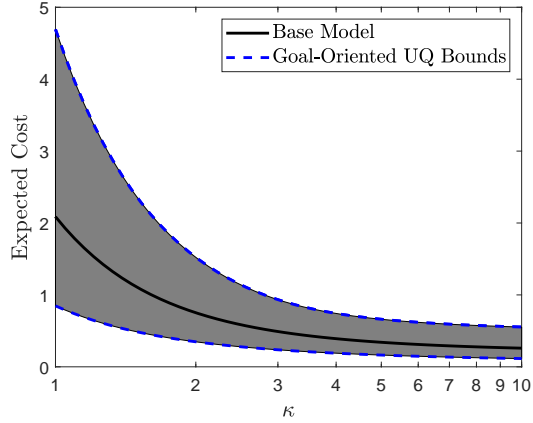}}
\caption{ Bound on the expected quadratic cost, \req{eq:control_QoI}, for a nonlinear  stochastic control system (alternative model) \req{eq:nonlin_control}, (suboptimally) controlled via the optimal control for an approximating linear system (baseline model) \req{eq:lin_control}.  The blue dashed lines, computed from \req{eq:control_UQ_final}, constrain the QoI to the gray region. The parameter $\kappa$ (defined in \req{eq:B_D_def}) governs the strength of the controller. }\label{fig:control_plot} 
 \end{figure}
Next, compute a Cholesky factorization of $C_\lambda$, $C_\lambda=M_\lambda^TM_\lambda$. The integrand in \req{eq:quadratic_cost} is non-negative, hence the hypotheses of Theorem \ref{thm:UQ_integral_QoI} are met and we can use \req{eq:Gaussian_UQ} to obtain
\begin{align}\label{eq:control_UQ_final}
&-\!\int_{\mathcal{T}}\inf_{c>0}\bigg\{\!-\frac{1}{2c}\log\left(|\det\left(I+c\Sigma_sC_\lambda\right)|\right)+\frac{1}{2c}\alpha^2s\bigg\}\pi(ds)\notag\\
\leq&E_{\widetilde{P}}\left[F\right]\\
\leq&\int_{\mathcal{T}}\inf_{0<c<1/\|M_\lambda\Sigma_sM_\lambda^T\|}\bigg\{\!-\frac{1}{2c}\log\left(|\det\left(I-c\Sigma_sC_\lambda\right)|\right)+\frac{1}{2c}\alpha^2s\bigg\}\pi(ds).\notag
\end{align}

 Figure \ref{fig:control_plot}  shows numerical results,  corresponding to the following simple example system, with a 2-D state variable and a 1-D control variable:
\begin{align}\label{eq:B_D_def}
B=\begin{bmatrix}
2 &0.1\\
0.1&-1
\end{bmatrix},\,\,\, D=\begin{bmatrix}
\kappa\\
0
\end{bmatrix},
\end{align}
$Q=I$, $R=1$, $\Sigma_0=0$, $\lambda=1/2$, $\sigma=I$, $\alpha=1/2$. Note that the size of the perturbation, $\sigma\beta$, is not required to be `small' for the method to produce non-trivial bounds, as this is a non-perturbative calculation. Finally,  if one instead attempts to use the non-goal-oriented relative entropy  to derive UQ bounds, the results are again trivial (infinite), due to the unbounded time-horizon.

\subsection{Semi-Markov Perturbations of a $M/M/\infty$ Queue}\label{example:birth_death}

Continuous-time jump processes with non-exponential waiting times (i.e., waiting times with memory) are found  in  many real-world systems; see, for example, the telephone call-center waiting time distribution in Figure 2 of \cite{doi:10.1198/016214504000001808}. However, these systems are often approximated by a much simpler Markovian (i.e., exponentially-distributed waiting time) model. Here we derive  robustness bounds for such approximations.

More specifically, we take the baseline model to be a  $M/M/\infty$ queue (i.e., continuous-time birth-death process) with arrival rate  $\alpha>0$ and service rate $\rho>0$:
\begin{flalign} \label{eq:CTMC_base}
\text{\bf      Baseline Model:} &&P(X_{t+dt}=x+1|X_t=x)=\alpha dt,\,\,\, P(X_{t+dt}=x-1|X_t=x)=\rho x dt,\,\,\,x\in\mathbb{Z}_0.&& 
\end{flalign}
The embedded jump Markov-chain  has transition probabilities, $a(x,y)$, given by
\begin{align}\label{eq:jump_chain}
a(x,x+1)=\alpha/(\alpha+\rho x),\,\,\,\,a(x,x-1)=\rho x/(\alpha+\rho x),
\end{align}
 and the waiting-times are exponentially distributed with jump rates 
 \begin{align}\label{eq:lambda_def}
 \lambda(x)=\alpha+\rho x.
 \end{align}
  We consider the stationary case, i.e., $X_0$ is Poisson with parameter $\alpha/\rho$.  Also, note that the jump chain with transition matrix  (\ref{eq:jump_chain}) has the stationary distribution
  \begin{align}\label{eq:queue_base_pi}
  \pi(x)= \frac{(\alpha+\rho x)(\alpha/\rho)^x }{2\alpha x!} e^{-\alpha/\rho}.
\end{align}

The alternative models will be semi-Markov processes; these describe jump-processes with non-exponential waiting times (i.e., with memory).  Mathematically, a semi-Markov model is a piecewise-constant continuous-time process  defined by a jump chain, $X^J_n$, and jump times, $J_n$, and waiting times (i.e., jump intervals), $\Delta_{n+1}\equiv  J_{n+1}- J_n$, that satisfy
\begin{flalign} 
\text{\bf        Alternative Models:} &&\widetilde{P}( X^J_{n+1}=y, \Delta_{n+1}\leq t| X^J_1,..., X^J_{n-1}, X^J_n, J_1,..., J_n)&& \label{eq:semi_markov_P}\\
&& = \widetilde{P}( X^J_{n+1}=y, \Delta_{n+1}\leq t| X^J_n)\equiv \widetilde{Q}_{X^J_n,y}(t). && &&\notag
\end{flalign}
  $\widetilde{Q}_{x,y}(t)$ is called the semi-Markov kernel (see \cite{janssen2006applied,limnios2012semi} for further background and applications). Note that the baseline process (\ref{eq:CTMC_base}) is described by the semi-Markov kernel  
  \begin{align}\label{eq:semi_markov_Q}
  Q_{x,y}(t)\equiv \int_0^t q_{x,y}(s)ds= a(x,y)\int_0^t \lambda(x)e^{-\lambda(x)s}ds.
  \end{align}
We emphasize that  the alternative models \eqref{eq:semi_markov_P} are not Markov processes in general; our methods make no such assumption on $\widetilde{P}$.
  
An important QoI in such a model is the average queue-length,
\begin{flalign}\label{eq:semiMarkov_QoI}
\text{\bf       QoI:} &&F_T\equiv\frac{1}{T}\int_0^T x_t dt.&& 
\end{flalign}
  More specifically, we will be concerned with the long-time behavior (limit of $E_{\widetilde{P}}[F_T]$ as $T\to\infty$). Time-averages such as \req{eq:semiMarkov_QoI} were discussed in Section \ref{sec:time_avg},  and for these the UQ bound \eqref{eq:goal_info_ineq2_stopped} implies
\begin{align}\label{semiMarkov_time_avg_UQ}
&\pm\left( E_{\widetilde{P}}\left[F_T\right]-E_P\left[F_T\right]\right)
\leq \inf_{c>0}\left\{\frac{1}{cT}\Lambda_P^{\widehat{F}_T}(\pm c T)+\frac{1}{cT}R( \widetilde{P}|_{\mathcal{F}_T}\|P|_{\mathcal{F}_T})\right\}.
\end{align}
$E_P\left[F_T\right]=\alpha/\rho$ is  the average queue-length in the (stationary) baseline process.   Note that one does not need to first prove $F_T\in L^1(\widetilde{P})$ to obtain \req{semiMarkov_time_avg_UQ}; non-negativity of the QoI allows one to first work with a truncated QoI and then take limits.
 
 The  cumulant generating function of the time-averaged queue-length has the following limit  (see Section 4.3.4 in \cite{doi:10.1137/19M1237429}):
 \begin{align}\label{eq:queue_baseline}
 \lim_{T\to\infty}T^{-1}\Lambda^{\widehat{F}_T}_{P}(cT)=\frac{\alpha c^2}{\rho^2(1-c/\rho)},\,\,\,c<\rho.
 \end{align}
 Combining \req{semiMarkov_time_avg_UQ} and \req{eq:queue_baseline}, we obtain the following UQ bound:
\begin{proposition}
Let $F_T$ denote the average queue-length (\ref{eq:semiMarkov_QoI}) and $\widetilde{P}$ be a semi-Markov perturbation (see \req{eq:semi_markov_P} - \req{eq:semi_markov_Q}).  Then
 \begin{align}\label{eq:time_avg_UQ_limit}
\limsup_{T\to\infty}\left( E_{\widetilde{P}}\left[F_T\right]-\alpha/\rho\right)\leq &\inf_{0<c<\rho}\left\{\frac{1}{c}\frac{\alpha c^2}{\rho^2 (1-c/\rho)}+\frac{1}{c}H(\widetilde{P}\|P)\right\}\\
=&\left(2\sqrt{ H(\widetilde{P}\|P)/\alpha}+H(\widetilde{P}\|P)/\alpha\right)\frac{\alpha}{\rho},\notag\\
\liminf_{T\to\infty}\left( E_{\widetilde{P}}\left[F_T\right]-\alpha/\rho\right)\geq&-\inf_{c>0}\left\{\frac{1}{c}\frac{\alpha c^2}{\rho^2 (1+c/\rho)}+\frac{1}{c}H(\widetilde{P}\|P)\right\}\notag\\
=&\begin{cases} 
      -\left(2\sqrt{H(\widetilde{P}\|P)/\alpha}-H(\widetilde{P}\|P)/\alpha\right)\frac{\alpha}{\rho} & H(\widetilde{P}\|P)<\alpha \\
      -\frac{\alpha}{\rho} & H(\widetilde{P}\|P)\geq\alpha,
   \end{cases}\notag
\end{align}
where $H(\widetilde{P}\|P)$ is the relative entropy rate:
\begin{align}\label{eq:semimarkov_eta}
H(\widetilde{P}\|P)\equiv \limsup_{T\to\infty} \frac{1}{T} R( \widetilde{P}|_{\mathcal{F}_T}\|P|_{\mathcal{F}_T}).
\end{align}
 \end{proposition}
 
A formula for the relative entropy rate between semi-Markov processes was obtained in \cite{10.2307/3216060} under the appropriate ergodicity assumptions: 
 \begin{align}
&H(\widetilde{P}\|P)=\frac{1}{ \widetilde{m}_{\widetilde\pi}}\sum_{x,y}\widetilde{\pi}(x) \int_0^\infty  \widetilde{q}_{x,y}(t)\log\left(\widetilde{q}_{x,y}(t)/q_{x,y}(t)\right)dt,\\
&\widetilde{m}_{\widetilde{\pi}}\equiv\sum_{x}\widetilde{\pi}(x)\int_0^\infty (1-\widetilde{H}_x(t))dt,\,\,\,\widetilde{H}_x(t)=\sum_y \widetilde{Q}_{x,y}(t),\notag
 \end{align}
 where $\widetilde{\pi}$ is the invariant distribution for the embedded jump chain of $\widetilde{P}$.
 
 In particular, an alternative process with the same jump-chain as the base process  but a different  waiting-time distribution, $\widetilde{H}_x(t)=\int_0^t \widetilde{h}_x(s)ds$, is described by the semi-Markov kernel
 \begin{align}\label{eq:semiMarkov_wait_perturb}
 \widetilde{Q}_{x,y}(t)=a(x,y)\widetilde{H}_x(t).
 \end{align}
  In this case,  $\widetilde{\pi}=\pi$ and $H(\widetilde{P}\|P)$ can be expressed in terms of the relative entropy of the waiting-time distributions:
 \begin{align}\label{eq:semimarkov_rel_ent}
 H(\widetilde{P}\|P)=& \limsup_{T\to\infty} \frac{1}{T} R( \widetilde{P}|_{\mathcal{F}_T}\|P|_{\mathcal{F}_T})=\frac{1}{\widetilde{m}_{ \pi}}E_{\pi(dx)}\left[R(\widetilde{H}_x\|H_x)\right],\\
 \widetilde{m}_{{\pi}}=&\sum_{x}{\pi}(x)\int_0^\infty (1-\widetilde{H}_x(t))dt,
 \end{align}
 where $\pi$ is the invariant distribution for $a(x,y)$; recall the stationary distribution for  the  jump chain  of the $M/M/\infty$-queue was given in \req{eq:queue_base_pi}.
\begin{remark}
The quantity $\widetilde{m}_\pi$ is the mean sojourn time under the invariant distribution, $\pi$, and the expecation $E_{\pi(dx)}\left[R(\widetilde{H}_x\|H_x)\right]$ can be viewed as the mean relative entropy of a single jump (comparing the alternative and baseline model waiting-time distributions). The formula for $H(\widetilde{P}\|P)$, \req{eq:semimarkov_rel_ent}, therefore has the clear intuitive meaning of an information loss per unit time.
\end{remark}

 Other than $\widetilde{H}_x$, the formulas \req{eq:time_avg_UQ_limit} and \req{eq:semimarkov_rel_ent} include only baseline model quantities.  Next we show one can bound $\widetilde{H}_x$ for a particular class of semi-Markov processes, and thus obtain computable UQ bounds.

\subsubsection{UQ Bounds for Phase Type Distributions} Phase-type distributions constitute a semi-parametric description of waiting-time distributions, going beyond the exponential case to describe systems with memory; see  \cite{doi:10.1002/asm.3150100403,10.2307/4616418,doi:10.1198/016214504000001808} for examples and information on fitting  such distributions to data.  Probabilistically, they can be characterized in terms the time to absorption for a continuous-time Markov chain with a single absorbing state.   The density and distribution function of a phase type-distribution, $\text{PH}_k(\nu,T)$, are characterized by a $k\times k$ matrix $T$, and a probability vector $\nu\in\mathbb{R}^k$ (see \cite{bladt2017matrix} for  background):
 \begin{align}\label{eq:phase_denstity}
 \widetilde{h}(t)=\nu e^{Tt}(-T1),\,\,\,\widetilde{H}(t)\equiv\int_0^t \widetilde{h}(s)ds=1-\nu e^{Tt}1,
 \end{align}
  where $1$ is the vector of all $1$'s and $T$ satisfies:
  \begin{enumerate}
  \item $T_{j,j}< 0$ for all $j$,
 \item $T_{j,\ell}\geq 0$ for all $j\neq \ell$,
 \item $-T1$ has all non-negative entries.
  \end{enumerate}
 
Combining \req{eq:phase_denstity} with \req{eq:semimarkov_rel_ent}, we obtain a formula for the relative entropy rate of a semi-Markov perturbation:
\begin{proposition} 
Let the baseline model consist of a $M/M/\infty$ queue  (\ref{eq:CTMC_base}) and let $\widetilde{P}$ be a semi-Markov perturbation with state-$x$ waiting-time distributed as $\text{PH}_{k(x)}(\nu(x),T(x))$ (and with the same jump-chain, $a(x,y)$, as the base process).  Then the relative entropy rate is
\begin{align}
&H(\widetilde{P}\|P)\\
=&\frac{1}{\widetilde{m}_\pi}\!\sum_x \pi(x) \! \int_0^\infty \!\!\! \log\left(\lambda(x)^{-1}\nu(x) e^{(T(x)+\lambda(x))t}(-T(x)1)  \right) \nu(x) e^{T(x)t}(-T(x)1) dt,\notag\\
&\widetilde{m}_\pi=\sum_x{\pi}(x)(-\nu(x) T^{-1}(x) 1).\notag
\end{align}
\end{proposition}

As a simple example,  consider the case where the waiting-time at $x$ is distributed as a convolution of $\exp(\lambda(x))$ and $\exp(\gamma(x))$-distributions (convolutions of exponential distributions are  examples of phase-type distributions; again, see \cite{bladt2017matrix}), for some choices of $\gamma(x)>\lambda(x)$. Recall the baseline model rates were given in \req{eq:lambda_def}. The semi-Markov kernel for the alternative process is then
\begin{align}\label{eq:semi_markov_kernel_convolution}
&\widetilde{Q}_{x,y}(t)= a(x,y)\int_0^t \lambda(x)e^{-\lambda(x)s}\frac{1-e^{-(\gamma(x)-\lambda(x))s}}{1-\lambda(x)/\gamma(x)}ds.
\end{align}
Note that the Markovian limit occurs at $\gamma(x)= \infty$, where the integrand is non-analytic.  Hence, it appears unlikely that one could perform a perturbative study of this system.  In contrast, our method will produce computable UQ bounds.

\begin{figure}
 \begin{minipage}[b]{0.45\linewidth}
\centerline{\includegraphics[height=6cm]{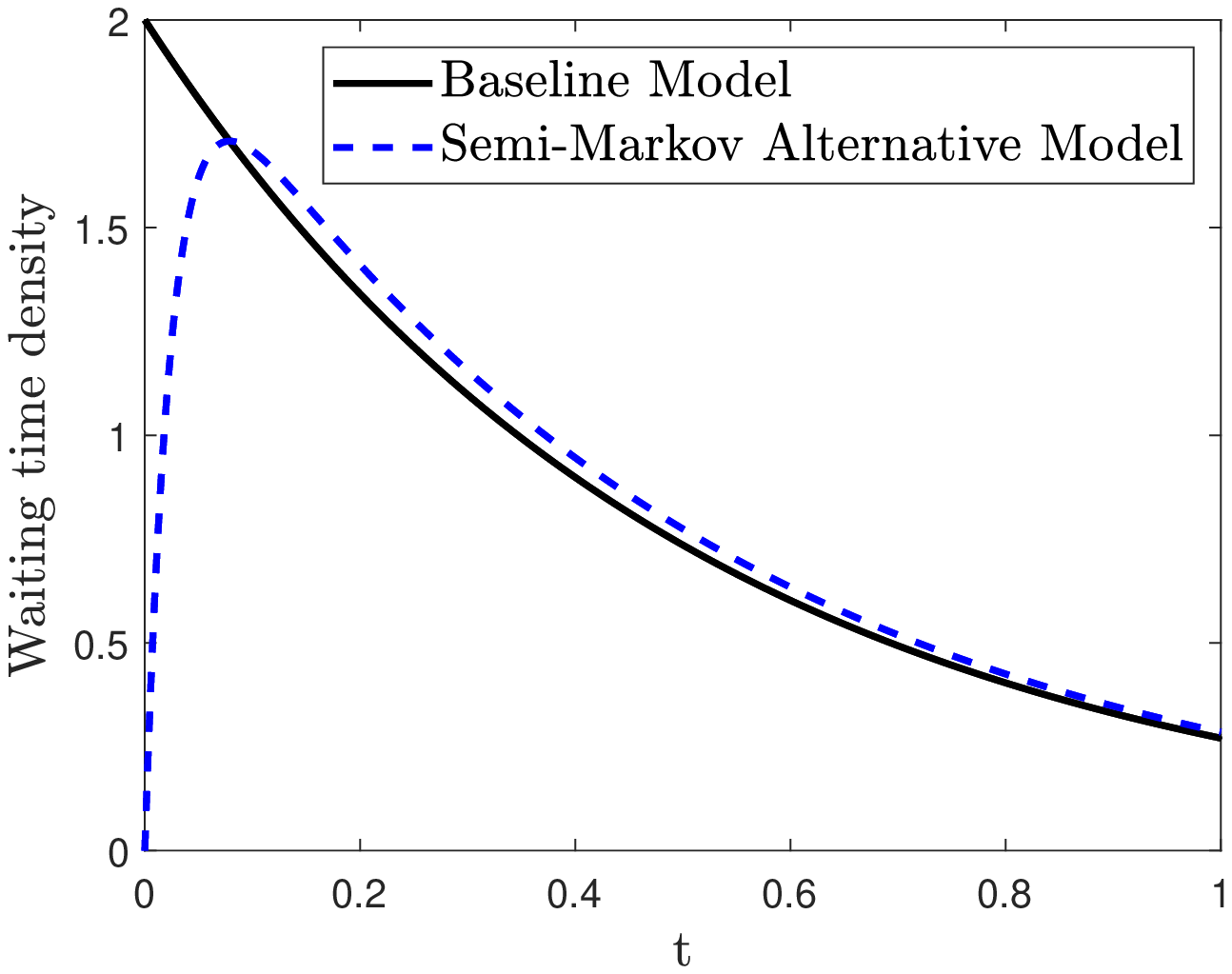}}
 \end{minipage}
 \hspace{0.5cm}
 \begin{minipage}[b]{0.5\linewidth}
\centerline{\includegraphics[height=6cm]{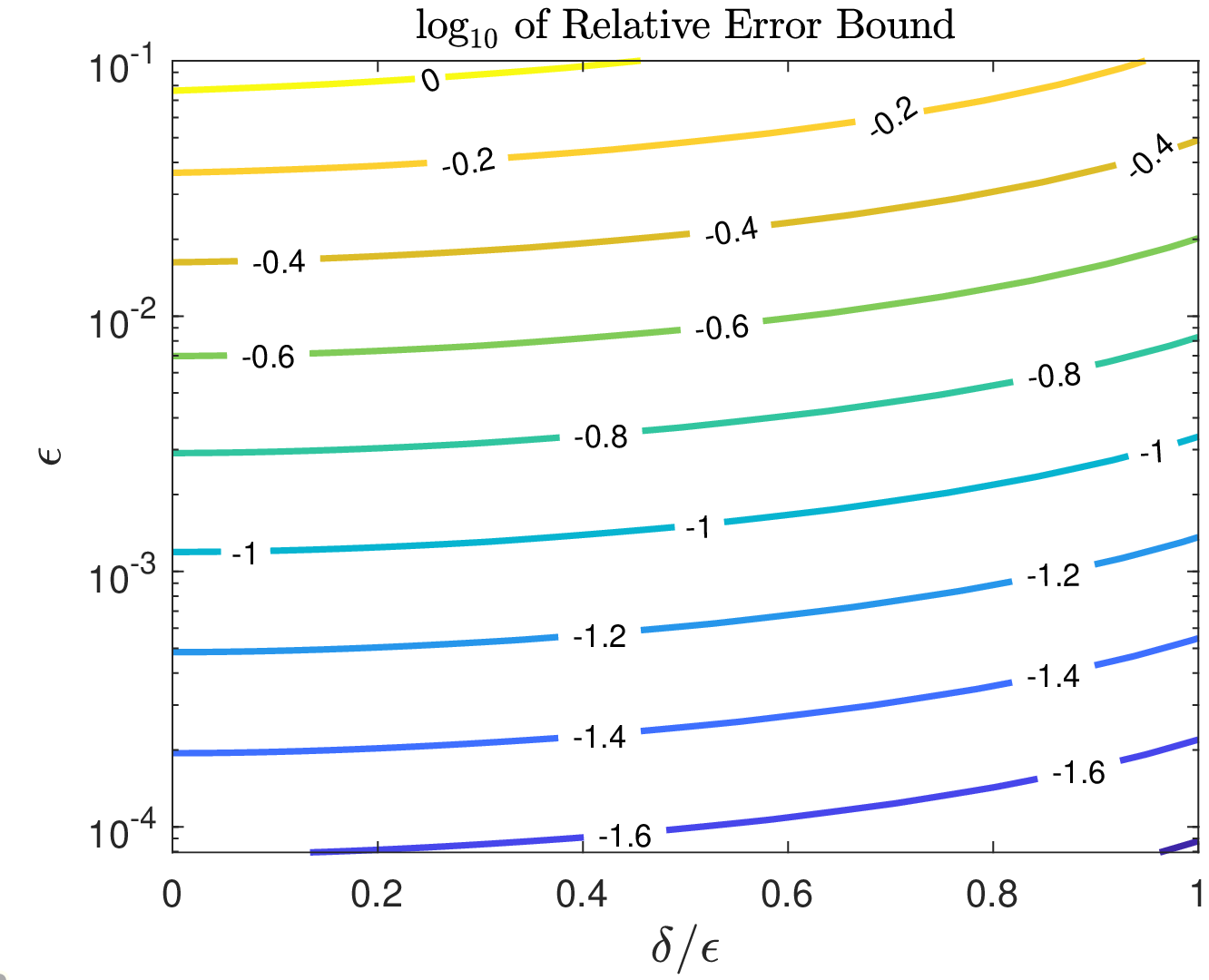}}
 \end{minipage}
\caption{Left: Waiting time densities for a baseline (exponential waiting time) model and a semi-Markov perturbation of the form \req{eq:semimarkov_perturbation}, corresponding to  $\delta=\epsilon=0.05$. Right: Upper bound on the relative error for the expected time-averaged-queue-length in the long-time limit, \req{eq:semimarkov_bound}, for a semi-Markov perturbation of a $M/M/\infty$ queue. The parameters $\delta$ and  $\epsilon$ quantify the size of the  perturbation; see \req{eq:semimarkov_perturbation} for their definitions.  As $\epsilon\to 0$, the semi-Markov waiting times approach exponential waiting times and the relative error approaches $0$. Increasing $\delta$ reduces the size of the model uncertainty; this is also reflected in the upper bound.}\label{fig:semimarkov}
 \end{figure}

To proceed, one must assume some partial knowledge of $\gamma(x)$. Here we consider the alternative family   consisting of  processes of the form \req{eq:semi_markov_kernel_convolution} with the following parameter bounds (of course, other choices are possible):
\begin{flalign}\label{eq:semimarkov_perturbation}
\text{\bf       Alternative Models:} &&0\leq \delta\leq \lambda(x)/\gamma(x)\leq \epsilon<1 \,\,\text{ for all $x$, i.e., } \frac{1}{\epsilon}\lambda(x)\leq \gamma(x)\leq \frac{1}{\delta}\lambda(x),&& 
\end{flalign}
where the jump-rates for the baseline model are $\lambda(x)=\alpha+\rho x$.  As a function of $x$, \req{eq:semimarkov_perturbation} therefore constrains $\gamma(x)$ to lie between two lines. This set describes waiting-times that are approximately  $\exp(\lambda(x))$-distributed when $\epsilon$ is small (i.e., when $\gamma(x)\gg\lambda(x)$). The parameter $\delta$ is a lower-bound on the perturbation. See the left panel of Figure \ref{fig:semimarkov} for an illustrative  example of this class of alternative models.

Using \req{eq:semimarkov_perturbation} we can bound
\begin{align}\label{eq:semiMarkov_eta_bound}
H(\widetilde{P}\|P)=&\frac{1}{ \widetilde{m}_\pi}\!\sum_{x}\pi(x) \!\int_0^\infty\! \!\!\lambda(x)e^{-\lambda(x)t}\frac{1-e^{-(\gamma(x)-\lambda(x))t}}{1-\lambda(x)/\gamma(x)}\log\left(\frac{1-e^{-(\gamma(x)-\lambda(x))t}}{1-\lambda(x)/\gamma(x)}\right)\!dt\\
\leq& \alpha r(\delta,\epsilon),\notag\\
r(\delta,\epsilon)\equiv& \frac{2}{ 1+\delta} \int_0^\infty  e^{-u}\!\left(\frac{1-e^{-(\epsilon^{-1}-1)u}}{1-\delta}1_{u<\frac{\log(1/\epsilon)}{\delta^{-1}-1}}+\frac{1-e^{-(\delta^{-1}-1)u}}{1-\epsilon}1_{u\geq \frac{\log(1/\epsilon)}{\delta^{-1}-1}}\right)\notag\\
&\hspace{1cm}\times\log\left(\frac{1-e^{-(\delta^{-1}-1)u}}{1-\epsilon}\right) du.\notag
\end{align}
To obtain the above, we used the inequality
\begin{align}\label{eq:semiMarkov_m}
 \widetilde{m}_\pi=&\sum_{x}\pi(x)\lambda(x)^{-1}(1+\lambda(x)/\gamma(x))   \geq(1+\delta)\sum_x\pi(x)\lambda(x)^{-1} =\frac{1+\delta}{2\alpha}.
\end{align}
Combining \req{eq:semiMarkov_eta_bound} and \req{eq:semiMarkov_m} with \req{eq:time_avg_UQ_limit}  gives the relative-error bounds
\begin{align}\label{eq:semimarkov_bound}
&\limsup_{T\to\infty}\left( E_{\widetilde{P}}\left[F_T\right]/\left(\frac{\alpha}{\rho}\right)-1\right)\leq 2\sqrt{ r(\delta,\epsilon)}+r(\delta,\epsilon),\\
&\liminf_{T\to\infty}\left( E_{\widetilde{P}}\left[F_T\right]/\left(\frac{\alpha}{\rho}\right)-1\right)\geq -\left(2\sqrt{r(\delta,\epsilon)}-r(\delta,\epsilon)\right)1_{r(\delta,\epsilon)<1}- 1_{r(\delta,\epsilon)\geq 1}.\notag
\end{align}
These bounds depend only on  the uncertainty parameters $\epsilon$ and $\delta$, and not on the base-model parameters. See the right pane in Figure \ref{fig:semimarkov} for a contour plot of the logarithm of the upper bound \req{eq:semimarkov_bound}.  Note that the error decreases as $\delta$ increases, due to the decreasing uncertainty in the form of the perturbation $\gamma$ (see \req{eq:semimarkov_perturbation}). As $\epsilon\to 0$, the semi-Markov perturbation approaches the baseline model (see \req{eq:semi_markov_kernel_convolution} and \req{eq:semimarkov_perturbation}) and so the relative error approaches $0$.  As previously noted,  relative entropy based UQ methods  are sub-optimal for   risk-sensitive quantities (i.e., rare events); see \cite{2020arXiv200102110A} for a treatment of risk-sensitive UQ for queueing models.

\subsection{Value of American Put Options with Variable Interest Rate}\label{sec:options}
As a final example, we consider the value of an asset in a variable-interest-rate environment, as compared to a constant-interest-rate baseline model.   For background see, for example,  Chapter 8 of \cite{shreve2004stochastic} or Chapter 8 of \cite{glasserman2013monte}. This example does not fit neatly into any of the problem categories from Section \ref{sec:UQ_processes}, though the technique of Corollary \ref{lemma:gen_rel_ent_bootstrap} will play a critical role. In particular, we will see that  the natural QoI for the base and alternative models is different; see \req{eq:finance_base_QoI} and \req{eq:finance_goal}.

Specifically, the baseline model in this example is geometric Brownian motion:
\begin{flalign}\label{eq:options_base}
\text{\bf        Baseline Model:} && dX_t^r=rX_t^rdt+\sigma X_t^rdW_t,&& 
\end{flalign}
where  $r,\sigma>0$ are constants.  This has the explicit solution
\begin{align}
X_t^r=X^r_0 \exp\left(\sigma W(t)+(r-\sigma^2/2)t\right).
\end{align} 
The variable-interest-rate alternative model has the general form
\begin{flalign}\label{eq:finance_SDE_tilde}
\text{\bf       Alternative Models:} &&d\widetilde{X}_t=(r+\Delta r(t,\widetilde{X}_t,\widetilde{Y}_t))\widetilde{X}_tdt+\sigma \widetilde{X}_td{W}_t.&& 
\end{flalign}
The initial values $X^r_0=\widetilde{X}_0\equiv X_0$ will be fixed  (here, the superscript on $X^r$ will always denote the constant interest rate, $r$, that is being used).  $r$ will be thought of as a known, fixed value, while $\Delta r$ will be considered an unknown  perturbation  that depends on time, on the asset price, as well as  on an additional driving process, $\widetilde{Y}_t$, that solves some auxiliary SDE; see Appendix \ref{sec:rel_ent_SDE} for more precise assumptions and \cite{shreve2004stochastic} for  discussion of various interest rate models. In Section \ref{sec:Vasicek} below, we will consider the Vasicek model.

The quantity of interest we consider is the value of a perpetual American put option:  Let $K>0$ be the option strike price.   The  payout if the option is exercised at time $t\geq 0$ is $K-x_t$, where $x_t$ is the asset price at time $t$. The relevant QoI is then the value, discounted to the present time.  In the baseline model, this takes the form   $V:[0,\infty]\times C([0,\infty),\mathbb{R})\to\mathbb{R}$, 
\begin{align}
V_t[x]=e^{-rt}(K-x_t) 1_{t<\infty},
\end{align}
while for the alternative model it is $\widetilde{V}:[0,\infty]\times C([0,\infty),\mathbb{R}\times\mathbb{R}^k)\to\mathbb{R}$,
\begin{align}
\widetilde{V}_t[x,y]=e^{-\int_0^t r+\Delta r(s,x_s,y_s)ds}(K-x_t)1_{t<\infty}.
\end{align}
The option-holder's strategy is to exercise the option when the asset price hits some level $L$, assumed to satisfy $L<X_0$ and $0<L<K$, i.e.,  we consider the   stopping time $\tau[x,y]=\tau[x]=\inf\{t\geq 0: x_t\leq L\}$. Therefore the baseline model QoI is
\begin{flalign}\label{eq:finance_base_QoI}
\text{\bf       Baseline QoI:} &&E[V_\tau[X^r]]=(K-L)(L/X_0)^{2r/\sigma^2}.&& 
\end{flalign}
To evaluate this, one uses the fact that 
 \begin{align}\label{eq:tau_Xr}
\tau\circ X^r=\inf\{t:X_t^r=L\}=\inf\{t: W_t+(r/\sigma-\sigma/2)t=-\sigma^{-1}\log(X_0/L)\}
\end{align}
 is the level $a\equiv -\sigma^{-1}\log(X_0/L)$ hitting time of a Brownian motion with constant drift $\mu\equiv r/\sigma-\sigma/2$, combined with the formula for the distribution of  such hitting times (see \cite{karatzas2014brownian} and also Appendix \ref{app:BM_hitting_times}).

The goal of  our analysis  is now to bound the expected option value in the alternative model:
\begin{flalign}\label{eq:finance_goal}
\text{\bf      Alternative QoI:} &&{E}\left[\widetilde{V}_{\tau\circ\widetilde{X}}[\widetilde{X},\widetilde{Y}]\right]= (K-L){E}\left[1_{\tau\circ\widetilde{X}<\infty}e^{-\int_0^{\tau\circ\widetilde{X}} r+\Delta r(s,\widetilde{X}_s,\widetilde{Y}_s)ds}\right].&& 
\end{flalign}
Note that both QoIs have unbounded time-horizon; therefore, to obtain non-trivial bounds one must again use a goal-oriented relative entropy.

\begin{remark}\label{remark:perturb_sigma}
The methods developed in this paper, applied to the goal oriented relative entropy $R(\widetilde{P}|_{\mathcal{F}_\tau}\|{P}|_{\mathcal{F}_\tau})$, are still not capable of comparing the  baseline model \req{eq:options_base}, with asset-price volatility $\sigma$, to an alternative model with perturbed asset-price volatility, $\sigma+\Delta\sigma$, due to the loss of absolute continuity  (again, see page 80 in \cite{freidlin2016functional}).  It is likely that one does have absolute continuity on a smaller sigma algebra than $\mathcal{F}_\tau$,  but we do not currently know how to bound the relative entropy on any such smaller sigma algebra.   An alternative approach to   robust option pricing under uncertainty in $\sigma$, utilizing $H_\infty$-control methods, was developed in \cite{doi:10.1287/moor.22.1.202}.
\end{remark}

To obtain UQ bounds for \req{eq:finance_goal}, it is useful to define the modified QoI
\begin{align}
F_t[x,y]=(K-L)\exp\left(-\int_0^t r+\Delta r(s,x_s,y_s)ds\right)1_{t<\infty},
\end{align}
and note that $E\left[\widetilde{V}_{\tau\circ\widetilde{X}}[\widetilde{X},\widetilde{Y}]\right]= {E}_{\widetilde{X},\widetilde{Y}}\left[F_{\tau}\right]$, where  ${E}_{\widetilde{X},\widetilde{Y}}$ denotes the expectation with respect to the joint distribution of $(\widetilde{X},\widetilde{Y})$ on path-space. (The notation $E_{X^r,Y}$ will similarly be used below.) In the next  subsection, we will bound  ${E}_{\widetilde{X},\widetilde{Y}}\left[F_{\tau}\right]$ for one class of rate perturbations.

\subsubsection{Vasicek Interest Rate Model}\label{sec:Vasicek}
Here we study a specific type of dynamical  interest rate model, known as the Vasicek model (see, for example, page 150 in \cite{shreve2004stochastic}):
\begin{flalign} 
\text{\bf      Alternative Models:} && &d\widetilde{X}_t=(r+\Delta r_t)\widetilde{X}_tdt+\sigma \widetilde{X}_td{W}_t,&& \label{eq:Vasicek_asset}\\
&&  &d\Delta r_t=-\gamma \Delta r_tdt+\widetilde{\sigma} d\widetilde W_t, \,\,\,\Delta r_0=0,\,\,\gamma>0,\,\,\widetilde{\sigma}>0, && \label{eq:Vasicek_rate}
\end{flalign}
where $W$ and $\widetilde{W}$ are independent Brownian motions.  The baseline model is still given by \req{eq:options_base}. Robustness bounds under such perturbations can be viewed as a model stress test for the type of financial instrument (QoI) studied here; see \cite{engelmann2011basel} for a detailed discussion of stress testing.

  The SDE \eqref{eq:Vasicek_rate} defines a 2-dimensional parametric family of alternative models, parameterized by $\gamma>0$ and $\widetilde{\sigma}>0$. Infinite-dimensional alternative families of $\Delta r$-models can also be considered; see Appendix \ref{app:options} for one such class of examples. \req{eq:Vasicek_rate}   has the exact solution
\begin{align}
\Delta r_t=\widetilde\sigma e^{-\gamma t}\int_0^t e^{\gamma s}d\widetilde{W}_s.
\end{align}
Note that $\Delta r_t$ is unbounded, and so a comparison-principle bound is is not  possible here. Under the model \eqref{eq:Vasicek_asset}, $r+\Delta r$ can become negative. Here, for both modeling and mathematical reasons, we eliminate this possibility by conditioning on the event $\int_0^{\tau\circ\widetilde{X}} r+\Delta r_sds\geq 0$, i.e., we only consider those paths with non-negative average interest rate up to the stopping time. This amounts to using the modified QoI
\begin{align}\label{eq:vasicek_QoI_non-neg}
\widetilde{F}_t[x,\Delta r]\equiv (K-L)\exp\left(-\int_0^t r+\Delta r_sds\right)1_{t<\infty}1_{\int_0^t r+\Delta r_sds\geq 0}
\end{align}
and bounding
\begin{flalign} \label{eq:Vasicek_conditional}
\text{\bf      Alternative QoI:} && \widetilde{E}[ F_\tau]\equiv{E}_{\widetilde{X},\Delta r}\!\left[\widetilde{F}_{\tau}\right]\!/{P}_{\widetilde{X},\Delta r}\left(\int_0^\tau \!\!r+\Delta r_sds\geq 0\right).&&
\end{flalign}
Other more complex interest rate models exist that enforce a positive rate via the dynamics; see, for example, page 275 in \cite{shreve2004stochastic}.

We also restrict to the parameter values $r>\frac{\widetilde\sigma^2}{2\gamma^2}$, for which a negative average interest rate is sufficiently rare. More precisely, this assumption implies $\lim_{n\to\infty}\int_0^n r+\Delta r_sds=\infty$ a.s.; see Appendix \ref{app:vasicek} for details. As a consequence, we also obtain $\widetilde{F}_{n}\to \widetilde{F}_\infty=0$ a.s. (under both the base and alternative models).  Therefore, Corollary \ref{cor:targeted_info_stopped} and Lemma \ref{lemma:lim_rel_ent} (see Appendix \ref{sec:rel_ent}) can both be used, yielding
\begin{align}\label{eq:Vasicek_UQ}
&\pm {E}_{\widetilde{X},\Delta r}\!\!\left[\widetilde{F}_{\tau}\right]\!\leq\inf_{c>0}\left\{\frac{1}{c}\Lambda_{P_{X^r,\Delta r}}^{\widetilde{F}_{\tau}}\!(\pm c)\!+\frac{1}{c}\liminf_{n\to\infty}R({P}_{\widetilde{X},\Delta r}|_{\mathcal{F}_{\tau\wedge n}}\|P_{X^r,\Delta r}|_{\mathcal{F}_{\tau\wedge n}})\right\}.
\end{align}

For this example, we do not utilize Theorem \ref{thm:goal_info_ineq3}, as the resulting expectation is difficult to evaluate.  Instead, we work with the cumulant generating function and relative entropy terms separately and use Corollary \ref{lemma:gen_rel_ent_bootstrap} together with Theorem \ref{thm:obs_adap_info_ineq}.

First we compute the cumulant generating function.  $\tau\circ X^r$ and $\Delta r$ are independent, hence we can use the result of \cite{MarioAbundo2015} to find
\begin{align}\label{eq:Vasicek_Lambda_F}
&\Lambda_{P_{X^r,\Delta r}}^{\widetilde F_{\tau}}(\pm c)\\
=&\log  \int \int_{-\infty}^\infty (2\pi\widetilde{\sigma}_t^2)^{-1/2} \exp\left(\pm c (K-L)e^{-z}1_{t<\infty}1_{z\geq 0}\right) e^{-\frac{(z-rt)^2}{2\widetilde{\sigma}_t^2}}dz P_{\tau\circ X^r}(dt),\notag\\
&\widetilde{\sigma}_t^2\equiv\frac{\widetilde{\sigma}^2}{2\gamma^3}\left(2\gamma t+4 e^{-\gamma t} - e^{-2\gamma t} - 3\right).\notag
\end{align}
 The integral over $t$ in \req{eq:Vasicek_Lambda_F} can then be computed as discussed in the text surrounding \req{eq:tau_Xr}, with $a\equiv -\sigma^{-1}\log(X_0/L)$ and $\mu\equiv r/\sigma-\sigma/2$. (When $t=\infty$, the inner integral should be interpreted as equaling $1$.)

\begin{remark}
Note that if the integrated rate ($z$ in \req{eq:Vasicek_Lambda_F}) is allowed to be negative (i.e., we remove $1_{z\geq 0}$ from the exponential) then $\Lambda_{P_{X^r,\Delta r}}^{\widetilde F_{\tau}}(c)$ is infinite and the upper UQ bound is  the trivial bound, $+\infty$.  This is the mathematical reason for conditioning on the event $\int_0^\tau r+\Delta r_sds\geq 0$.
 \end{remark}

 The relative entropy can be computed via Girsanov's theorem (Appendix \ref{app:vasicek} for a proof that the Girsanov factor is an exponential martingale)
\begin{align}
&R({P}_{\widetilde{X},\Delta r}|_{\mathcal{F}_{\tau\wedge n}}\|P_{X^r,\Delta r}|_{\mathcal{F}_{\tau\wedge n}})\leq E_{\widetilde{X},\Delta r}\left[G_n\right],\,\,\,\,\,G_n[x,\Delta r]\equiv \frac{1}{2}\int_0^{\tau(x)\wedge n}\!\!\!\!\sigma^{-2}|\Delta r_s|^2ds.
\end{align}
Note that $G_n$ is $\mathcal{F}_{\tau\wedge n}$-measurable and has finite expectation, as one can bound $\tau\wedge n$ by $n$ and then use the fact that $\Delta r_s$ is normal with mean $0$ and variance $\frac{\widetilde{\sigma}^2}{2\gamma}(1-e^{-2\gamma s})$.  Hence,  Corollary \ref{lemma:gen_rel_ent_bootstrap} is applicable, resulting in
\begin{align}\label{eq:Vasicek_R_bound}
&R({P}_{\widetilde{X},\Delta r}|_{\mathcal{F}_{\tau\wedge n}}\|P_{X^r,\Delta r}|_{\mathcal{F}_{\tau\wedge n}})\leq \inf_{\lambda>1}\left\{(\lambda-1)^{-1}\Lambda_{P_{X^r,\Delta r}}^{G_n}(\lambda)\right\}.
\end{align}
Again using the independence of $\tau\circ X^r$ and $\Delta r$, we have
\begin{align}\label{eq:Lambda_G}
\Lambda_{P_{X^r,\Delta r}}^{G_n}(\lambda)=\log \int E\left[e^{\lambda\frac{\sigma^{-2}}{2}\int_0^{t\wedge n}|\Delta r_s|^2ds}\right]P_{\tau\circ X^r}(dt).
\end{align}
Therefore
\begin{align} 
&\liminf_{n\to\infty}R({P}_{\widetilde{X},\Delta r}|_{\mathcal{F}_{\tau\wedge n}}\|P_{X^r,\Delta r}|_{\mathcal{F}_{\tau\wedge n}})\leq \inf_{\lambda>1}\left\{(\lambda-1)^{-1}\log \int E\left[e^{\lambda\frac{\sigma^{-2}}{2}\int_0^{t }|\Delta r_s|^2ds}\right]P_{\tau\circ X^r}(dt)\right\}
\end{align}
and the inner expectation can be evaluated using Theorem 1 (B) of \cite{10.2307/2102129}:
\begin{align}\label{eq:Vasicek_dr2}
&E\left[e^{\lambda\frac{\sigma^{-2}}{2}\int_0^{t}|\Delta r_s|^2ds}\right]\\
=&e^{\gamma t/2}\left[\frac{1}{\sqrt{1-\frac{\lambda\widetilde\sigma^2}{\sigma^2\gamma^2}}}\sinh\left(\gamma t\sqrt{1-\frac{\lambda\widetilde\sigma^2}{\sigma^2\gamma^2}}\right)+\cosh\left(\gamma t\sqrt{1-\frac{\lambda\widetilde\sigma^2}{\sigma^2\gamma^2}}\right)\right]^{-1/2}.\notag
\end{align}
By analytic continuation, this formula is valid for $\lambda<\sigma^2\gamma^2/\widetilde{\sigma}^2$.  Also, note that for the $n\to\infty$ limit of \req{eq:Lambda_G} to be finite, we need $\tau\circ X^r<\infty$ $P$-a.s., i.e., $\mu$ and $a$ must have the same signs.

\begin{figure}
 \begin{minipage}[b]{0.45\linewidth}
\centerline{\includegraphics[height=6cm]{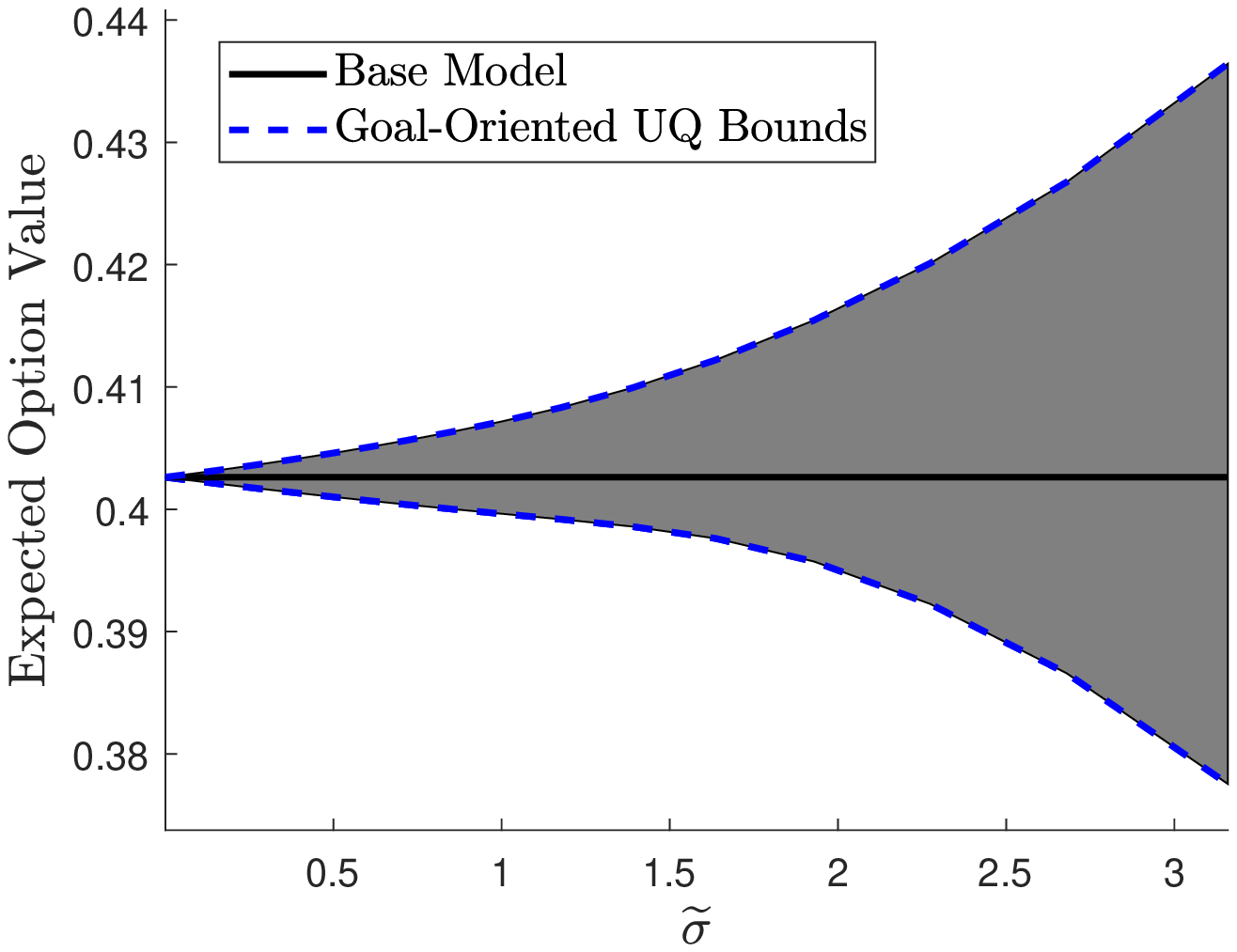}}
 \end{minipage}
 \hspace{0.5cm}
 \begin{minipage}[b]{0.5\linewidth}
\centerline{\includegraphics[height=6cm]{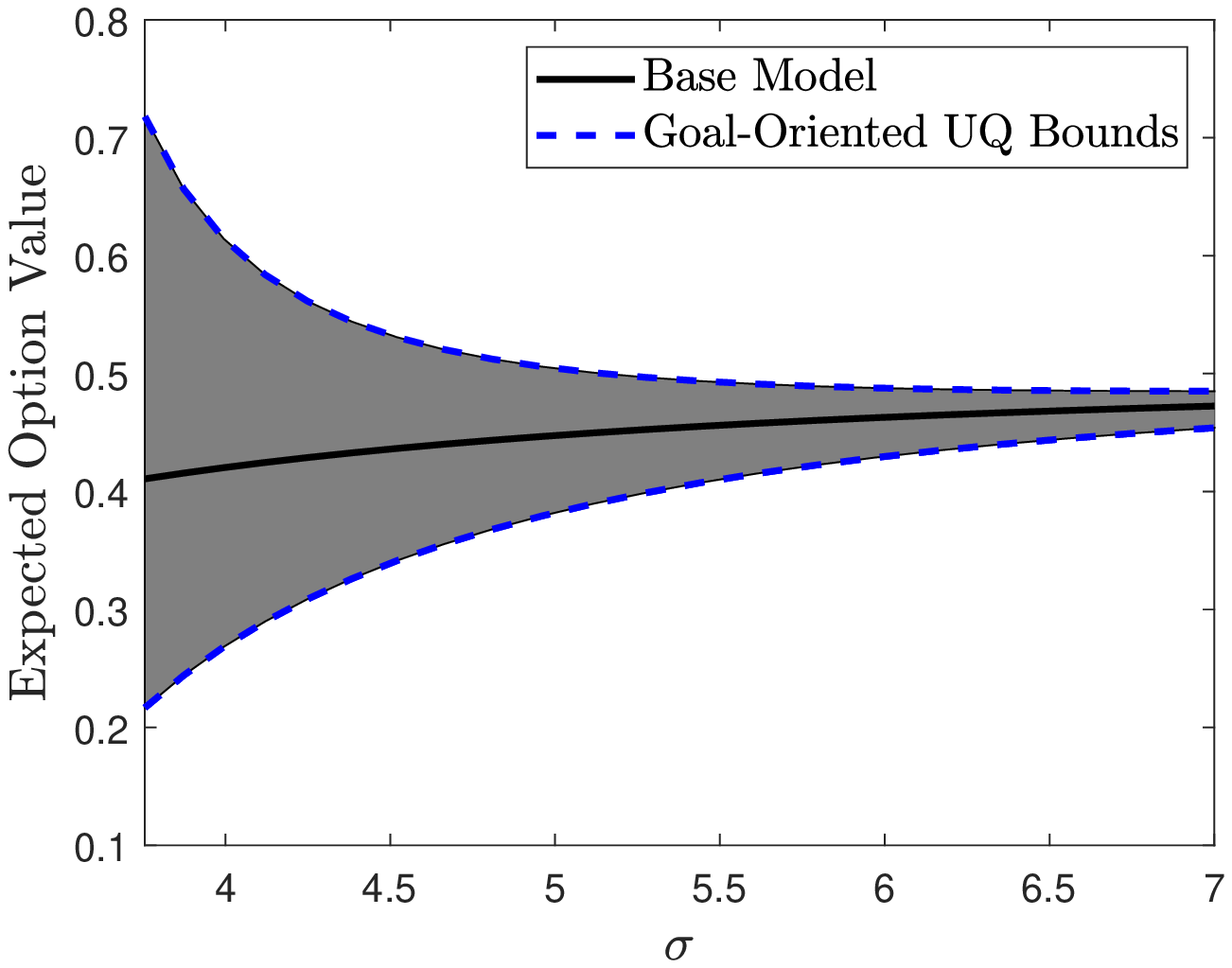}}
 \end{minipage}
 \caption{Bound on the expected option value under the Vasicek model. Black curves are for the baseline model (see \req{eq:finance_base_QoI}) and blue dashed curves are the bounds on the alternative model QoI, \req{eq:Vasicek_conditional}, that result from our method (see \req{eq:Vasicek_UQ} - \req{eq:Vasicek_P}); the QoI is constrained to the gray regions. Left:  As a function of the volatility of the interest-rate perturbation, with $\gamma=2$, $r=5/4$, $K=1$, $L=1/2$, $X_0=2$, $\sigma=4$. Right: As a function of the volatility of the asset price, with $\gamma=2$, $r=1$, $K=1$, $L=1/2$, $X_0=2$, $\widetilde{\sigma}=6$.}\label{fig:option_value_Vasicek}
 \end{figure}

\begin{figure}
 \end{figure}

We can  use Corollary \ref{cor:event_prob_UQ} to bound  ${P}_{\widetilde{X},\Delta r}\left(\int_0^\tau r+y_sds\geq 0\right)$:
\begin{align}\label{eq:Vasicek_Ptilde_bound}
&\pm {P}_{\widetilde{X},\Delta r}\left(\int_0^\tau r+y_sds\geq 0\right)\\
\leq& \inf_{c>0}\bigg\{\frac{1}{c}\log\left(1+(e^{\pm c}-1)P_{X^r,\Delta r}\left(\int_0^\tau r+\Delta r_sds\geq 0\right)\right)\notag\\
&\hspace{.2cm}+\frac{1}{c}\liminf_{n\to\infty}R(P_{\widetilde{X},\Delta r}|_{\mathcal{F}_{\tau\wedge n}}\|P_{X^r,\Delta r}|_{\mathcal{F}_{\tau\wedge n}})\bigg\},\notag
\end{align}
where we computed the  MGF as in \req{eq:event_UQ_bound} and 
\begin{align}\label{eq:Vasicek_P}
P_{X^r,\Delta r}\left(\int_0^\tau r+\Delta r_sds\geq 0\right)=&\int \int_0^\infty (2\pi\widetilde{\sigma}_t^2)^{-1/2} e^{-\frac{(z-rt)^2}{2\widetilde{\sigma}_t^2}} dz P_{\tau\circ X_r}(dt)\\
=&\int \frac{1}{2}\left(\text{erf}\left(rt/\sqrt{2\widetilde{\sigma}^2_t}\right)+1\right) P_{\tau\circ X_r}(dt).\notag
\end{align}

Figure \ref{fig:option_value_Vasicek} shows numerical results for the bounds on \req{eq:Vasicek_conditional} that result from combining \req{eq:Vasicek_UQ} through  \req{eq:Vasicek_P} (the full statement of the bound is quite lengthy so, rather than stating it here, we have placed it  in Appendix \ref{app:vasicek},  \req{eq:Vasicek_final_bound}). The baseline model is shown in the black solid curves (see \req{eq:finance_base_QoI}) and the blue dashed curves show the bounds \eqref{eq:Vasicek_final_bound} on the alternative model QoI that result from our method, which constrain the QoI to the gray regions; we emphasize that  the alternative model uses the modified QoI, \req{eq:vasicek_QoI_non-neg}-\eqref{eq:Vasicek_conditional}. In the left plot we see that the uncertainty approaches zero as $\widetilde{\sigma}$ approaches zero, corresponding to zero rate perturbation.  However, $\widetilde{\sigma}$ does not need to be small for our method to apply.

\begin{remark}
We are able to perturb the diffusion parameter, $\widetilde{\sigma}$, without losing absolute continuity because, technically, the baseline model is the joint distribution of $(X^r_t,\Delta r_t)$.  Of course, $X^r$ is not coupled to $\Delta r_t$, nor does  the baseline QoI depend on it, but regardless, this makes it clear that there is no loss of absolute continuity between the distributions of  $(X^r_t,\Delta r_t)$ and $(\widetilde{X}_t,\Delta r_t)$ when $\widetilde{\sigma}$ is changed.
\end{remark}

\appendix

\section{ Proof of the UQ Information Inequality}\label{app:base_inequality_proof}
For completeness, here we provide a proof of the UQ bound (\ref{goal_oriented_bound}), which can be found in \cite{chowdhary_dupuis_2013,DKPP}:
\begin{proposition}
Let $P$, $\widetilde{P}$ be probability measures and   $F:\Omega\to\mathbb{R}$, $F\in L^1(\widetilde P)$.  Then
\begin{align}\label{goal_oriented_bound2}
\pm E_{\widetilde P}[F]\leq \inf_{c>0}\left\{\frac{1}{c} \Lambda_{P}^{F}(\pm c)+\frac{1}{c}R(\widetilde P\|P)\right\}.
\end{align}
\end{proposition}
\begin{proof}

 The key tool  is the Gibbs variational principle, which relates the cumulant generating function and relative entropy as follows: If  $G:\Omega\to\mathbb{R}$ is measurable and there exists $c_0>0$ such that $E_P[e^{c_0G}] < \infty$ then 
\begin{equation}\label{gibbs}
\log E_P[e^G] = \sup_{Q: R(Q\|P)< \infty} \left\{  E_{Q}[G]  - R(Q \| P) \right\} \,.
\end{equation}  
A (slightly weaker) version of this result is found in Proposition 4.5.1 in \cite{dupuis2011weak}, where it is assumed that $G$ is bounded above. However, the above variant can then easily be obtained by applying \req{gibbs} to $G_n\equiv G\wedge n$ and then taking $n\to\infty$ via the monotone convergence theorem (the assumption $E_P[e^{c_0G}] < \infty$ is necessary to ensure that $E_Q[G^+]<\infty$ for all $Q$ with $R(Q\|P)<\infty$, where $G^+\equiv \max\{G,0\}$ is the positive part of $G$).

Now fix $c>0$ and consider the case of the upper signs in \req{goal_oriented_bound2} (the proof for the lower signs proceeds similarly).   If $R(\widetilde{P}\|P)=\infty$ or $\Lambda_P^{{F}}(c)=\infty$ then the claimed bound (at   $c$) is trivial, so suppose not. In this case, we have $E_P[e^{c{F}}]<\infty$ and so we can apply  \req{gibbs} to $G\equiv c{F}$.  Bounding the supremum in \req{gibbs} below by the value at $Q=\widetilde{P}$, solving for $E_{\widetilde{P}}[c{F}]$, and then dividing by $c$ yields the  result.
\end{proof}

\section{Tightness of the Bounds from Theorem \ref{thm:goal_info_ineq3}}\label{app:tight}

The bounds (\ref{eq:goal_info_ineq3}) are  tight, in the following sense:
\begin{corollary}\label{cor:tightness}
Let $F$ and $\mathcal{G}$ be as in Theorem \ref{thm:goal_info_ineq3} and assume that $R(\widetilde{P}|_{\mathcal{G}}\| P|_{\mathcal{G}})<\infty$ and $c\mapsto\log E_{\widetilde{P}}\left[\exp\left(c F\right)\right]$ is finite on a neighborhood of $0$. Then there exists a $\mathcal{G}$-measurable $G$ such that \req{eq:goal_info_ineq3} is an equality.
 \end{corollary}
 \begin{proof}

 Let $G=\log( d\widetilde{P}|_{\mathcal{G}}/dP|_{\mathcal{G}})$.  The assumptions,  together with  \req{eq:goal_info_ineq3}, imply $G\in L^1(\widetilde{P})$ and we have
\begin{align}
&\pm E_{\widetilde P}[F]\leq \inf_{c>0}\left\{\frac{1}{c} \log E_{{P}}\left[\exp\left(\pm c F+G\right)\right]\right\}=\lim_{c\searrow 0}\frac{1}{c} \log E_{\widetilde{P}}\left[\exp\left(\pm c F\right)\right]\\
=&\frac{d}{dc}\bigg|_{c=0}\!\!\!\!\!\log E_{\widetilde P}\left[\exp\left(\pm c F\right)\right]=\pm E_{\widetilde{P}}[F].\notag
\end{align}
Note that in the first line, we used the fact that $c\mapsto \Lambda_P^F(\pm c)/c$ are increasing on $c>0$, as can be proven using H{\"o}lder's inequality.
\end{proof}

\begin{remark}
   Again, we emphasize that the choice $G=\log( d\widetilde{P}|_{\mathcal{G}}/dP|_{\mathcal{G}})$  is rarely  tractable and so Corollary \ref{cor:tightness} is not useful in practice.  Rather, it shows that any loss of tightness in the bound (\ref{eq:goal_info_ineq3}) is the result of one's  choice of $\mathcal{G}$ and $G$, and is not intrinsic to the use of that bound.
\end{remark}

\section{Tools for Computing the Goal-Oriented Relative Entropy}\label{sec:rel_ent}
In this appendix, we present several lemmas that are generally useful when computing the goal-oriented relative entropy. First, we address the question of processes with  different initial distributions:
\begin{lemma}\label{lemma:chain_rule}
Suppose that the base and alternative models are of the form
\begin{align}
P=P^\mu=\int P^x(\cdot) \mu(dx),\,\,\,\,\,\,\widetilde{P}=\widetilde{P}^{\widetilde{\mu}}=\int \widetilde{P}^x(\cdot) \widetilde{\mu}(dx)
\end{align}
for some probability kernels, $P^x$ and $\widetilde{P}^x$, from $(\mathcal{X},\mathcal{M})$ to $(\Omega,\mathcal{F})$, and some probability measures (e.x. initial distributions), $\mu$ and $\widetilde{\mu}$, on $(\mathcal{X},\mathcal{M})$. 

 In this case, one can use the chain-rule for relative entropy (see, for example, \cite{dupuis2011weak}) to bound
\begin{align}\label{eq:rel_ent_ics}
R(F_*\widetilde{P}^{\widetilde{\mu}}\|F_*P^\mu)\leq  R(\widetilde{\mu}\|\mu)+\int R(F_*\widetilde{P}^x\| F_*P^x)\widetilde{\mu}(dx).
\end{align}
\end{lemma}
\begin{remark}
 Note that without further assumptions, \req{eq:rel_ent_ics} is not an equality.  For example, let $\nu$ be a probability measure on $\mathbb{R}$,  $F=\pi_2$ (the projection onto the second component in $\mathbb{R}\times\mathbb{R}$), and for $x\in\mathbb{R}$ define  $P^x=\widetilde P^x=\delta_x\times \nu$.  Then for $\mu\neq\widetilde\mu$ we have $P^\mu=\mu\times\nu$, $\widetilde{P}^{\widetilde{\mu}}=\widetilde{\mu}\times \nu$ and
\begin{align}
&R(F_* \widetilde{P}^{\widetilde{\mu}}\|F_*P^\mu)=R(\nu\|\nu)=0<R(\widetilde\mu\|\mu)+\int R(F_* \widetilde{P}^x\|F_*P^x)\widetilde{\mu}(dx).
\end{align}
Intuitively, inequality in \req{eq:rel_ent_ics} arises when the QoI $F$ `forgets' some of the information regarding the discrepancy between the initial condition, as the right-hand-side in \req{eq:rel_ent_ics} incorporates the full discrepancy.
\end{remark}
\begin{remark}
Recall that $R(F_*\widetilde{P}^{\widetilde{\mu}}\|F_*P^\mu)=R(\widetilde{P}^{\widetilde{\mu}}|_{\sigma(F)}\|P^\mu|_{\sigma(F)})$, where $\sigma(F)$ is the sigma algebra generated by $F$. If $F$ is measurable with respect to a sub sigma algebra, $\mathcal{G}$, then $\sigma(F)\subset\mathcal{G}$ and one can weaken the bound via $R(F_*\widetilde{P}^x\| F_*P^x)\leq R(\widetilde{P}^x|_{\mathcal{G}}\| P^x|_{\mathcal{G}})$, as discussed in Theorem \ref{thm:obs_adap_info_ineq}.
\end{remark}

The main tool we use to compute the relative entropy up to a stopping time, Lemma \ref{lemma:radon_mart} below, requires one to work with bounded stopping times.  This limitation can often be overcome by taking limits. For example:
\begin{lemma}\label{lemma:lim_rel_ent}
Suppose that one is working within the framework of Assumption \ref{assump:time_dep_QoI} and one of the following two conditions holds:
\begin{enumerate}
\item $\tau$ is finite $P$-a.s. and $\widetilde{P}$-a.s.
\item $F_n\to F_\infty$ as $n\to\infty$, $n\in\mathbb{Z}^+$, both $P$ and $\widetilde{P}$-a.s.
\end{enumerate}
Then 
\begin{align}
R( (F_\tau)_*\widetilde{P}\|(F_\tau)_*P)\leq& \liminf_{n\to\infty} R((F_{\tau\wedge n})_*\widetilde{P}\|(F_{\tau\wedge n})_*P)\\
\leq&\liminf_{n\to\infty} R(\widetilde{P}|_{\mathcal{F}_{\tau\wedge n}}\|P|_{\mathcal{F}_{\tau\wedge n}}).\notag
\end{align}
\end{lemma}
\begin{proof}

Either condition implies that $(F_{\tau\wedge n})_*\widetilde{P}\to (F_{\tau})_*\widetilde{P}$ and $(F_{\tau\wedge n})_*P\to (F_\tau)_*P$ weakly as $n\to\infty$.  The result then follows from lower-semicontinuity of relative entropy (see \cite{dupuis2011weak}).
\end{proof}

The general tool we use to compute relative entropy up to a stopping time is the optional sampling theorem (see, for example, \cite{karatzas2014brownian}), as summarized by the following lemma:
\begin{lemma}\label{lemma:radon_mart}
Let $P,\widetilde{P}$ be probability measures on the filtered space $(\Omega,\mathcal{F},\{\mathcal{F}_t\}_{t\in\mathcal{T}})$ such that $\widetilde{P}|_{\mathcal{F}_t}\ll P|_{\mathcal{F}_t}$  for all $t$. 

Then $\rho_t=\frac{d\widetilde{P}|_{\mathcal{F}_t}}{dP|_{\mathcal{F}_t}}$ is an $(\mathcal{F}_t,P)$-martingale.  In the continuous-time case, we also assume that we have a right-continuous version of $\rho_t$.

Given this, for any stopping time, $\tau$, we have $\frac{d\widetilde{P}|_{\mathcal{F}_{\tau\wedge t}}}{dP|_{\mathcal{F}_{\tau\wedge t}}}=E(\rho_t|\mathcal{F}_{\tau\wedge t})=\rho_{\tau\wedge t}$ and
\begin{align}
R(\widetilde{P}|_{\mathcal{F}_{\tau\wedge t}}\|P|_{\mathcal{F}_{\tau\wedge t}})=&E_{\widetilde{P}}\left[\log(\rho_{\tau\wedge t})\right].
\end{align}
\end{lemma}

Using Lemma \ref{lemma:radon_mart}, the relative entropy up to a stopping time for various types of Markov processes can be computed via standard techniques;  see Appendix \ref{sec:markov_rel_ent}.

\section{Relative Entropy up to a Stopping Time for Markov Processes}\label{sec:markov_rel_ent}

Below, we use Lemma \ref{lemma:radon_mart} to give explicit formulas for the relative entropy up to a stopping time for several classes of Markov processes in both discrete and continuous time.
\subsection{Discrete-Time Markov Processes}
Let $\mathcal{X}$ be a Polish space, $p(x,dy)$ and $\widetilde{p}(x,dy)$ be transition probabilities on $\mathcal{X}$ and $\mu$, $\widetilde{\mu}$ be probability measures on $\mathcal{X}$.  Let $P^\mu$ and $\widetilde{P}^{\widetilde{\mu}}$ be the induced probability measures on $\Omega\equiv\prod_{i=0}^\infty\mathcal{X}$  with transition probabilities $p$ and $\widetilde p$ respectively, and initial distribution $\mu$ and $\widetilde{\mu}$.  Let $\pi_n:\Omega\to\mathcal{X}$ be the natural projections,  $\mathcal{F}_n=\sigma(\pi_m:m\leq n)$, and $\tau$ be a $\mathcal{F}_n$-stopping time.

We have the following bound on the relative entropy,  via the optional sampling theorem:
\begin{lemma}
Suppose that $\widetilde p(x,dy)=g(x,y)p(x,dy)$ for some measurable $g:\mathcal{X}\times\mathcal{X}\to[0,\infty)$ and let $N\geq 0$. Then
\begin{align}
R(\widetilde{P}^x|_{\mathcal{F}_{\tau\wedge N}}\|P^x|_{\mathcal{F}_{\tau\wedge N}})=&\widetilde{E}^x\left[\log\left(\prod_{i=1}^{\tau\wedge N} g(\pi_{i-1},\pi_i)\right)\right].
\end{align}
\end{lemma}

\subsection{Continuous-Time Markov Chains}

Let $\mathcal{X}$ be a countable set, $P$ and $\widetilde{P}$ be probability measures on $(\Omega,\mathcal{F})$, and $X_t:\Omega\to\mathcal{X}$, $t\in[0,\infty)$, such that $(\Omega,\mathcal{F},P,X_t)$ and  $(\Omega,\mathcal{F},\widetilde{P},X_t)$ are continuous-time Markov chains (CTMCs) with transition probabilities $a(x,y)$, $\widetilde{a}(x,y)$, jump rates $\lambda(x)$, $\widetilde{\lambda}(x)$, and initial distributions $\mu$, $\widetilde{\mu}$ respectively.  Let $\mathcal{F}_t$ be the natural filtration for $X_t$, and $X^J_n$ be the embedded jump chain with jump times $J_n\in[0,\infty]$.  We assume $J_n(\omega)\to\infty$ as $n\to\infty$ for all $\omega$.

Suppose $\widetilde{\mu}\ll\mu$ and  that whenever $a(x,y)=0$ we also have $\widetilde{a}(x,y)=0$ (note that this also implies that $\widetilde\lambda(x)=0$ whenever $\lambda(x)=0$).  Then  $\widetilde{P}|_{\mathcal{F}_t}\ll P|_{\mathcal{F}_t}$ for all $t\geq 0$ and
\begin{align}
\frac{d \widetilde{P}|_{\mathcal{F}_t}}{dP|_{\mathcal{F}_t}}=\frac{\widetilde{\mu}(X_0)}{\mu(X_0)}\exp\left(\int_0^t \log\left(\frac{\widetilde{\lambda}(X_{s^-})\widetilde{a}(X_{s^-},X_{s})}{\lambda(X_{s^-})a(X_{s^-},X_{s})} \right)N(ds)-\int_0^t \widetilde{\lambda}(X_s)-\lambda(X_s)ds\right),
\end{align}
where $N_t=\max\{n: J_n\leq t\}$ is the number of jumps up to time $t$ (see, for example, \cite{heinz2015stochastic}). Note that in defining the functions $\widetilde{\mu}(x)/\mu(x)$ and $\widetilde{\lambda}(x)\widetilde{a}(x,y)/(\lambda(x)a(x,y))$, we  use the convention $0/0\equiv 1$.

Lemma  \ref{lemma:radon_mart} then yields the following:
\begin{lemma}
\begin{align}
&R(\widetilde{P}|_{\mathcal{F}_{\tau\wedge t}}\|P|_{\mathcal{F}_{\tau\wedge t}})\\
=&R(\widetilde{\mu}\|\mu)+E_{\widetilde{P}}\left[\int_0^{\tau\wedge t} \log\left(\frac{\widetilde{\lambda}(X_{s^-})\widetilde{a}(X_{s^-},X_{s})}{\lambda(X_{s^-})a(X_{s^-},X_{s})} \right)N(ds)-\int_0^{\tau\wedge t} \widetilde{\lambda}(X_s)-\lambda(X_s)ds\right].\notag
\end{align}

\end{lemma}

\subsection{Change of Drift for SDEs}\label{sec:rel_ent_SDE}
Here we consider the case where $P^x$ and $\widetilde P^x$ are the distributions on path-space, $C([0,\infty),\mathbb{R}^n)$, of the solution flows $X_t^x$ and $\widetilde X_t^x$ of a pair of SDEs starting at $x\in\mathbb{R}^n$.  More precisely:
\begin{assumption}\label{assump:SDE} 
Assume:
\begin{enumerate}
\item We have filtered probability spaces, $(M,\mathcal{G}_\infty,\mathcal{G}_t,P)$, $(\widetilde{M},\widetilde{\mathcal{G}}_\infty,\widetilde{\mathcal{G}}_t,\widetilde{P})$ that satisfy the usual conditions \cite{karatzas2014brownian} and are equipped with $\mathbb{R}^m$-valued Wiener processes  $W_t$ and $\widetilde W_t$ respectively (Wiener processes with respect to the respective filtrations).

\item We  have another probability space $(N,\mathcal{N},\nu)$, a  $\mathcal{G}_0$-measurable random quantity $Z:M\to N$, and a  $\widetilde{\mathcal{G}}_0$-measurable random quantity $\widetilde{Z}:\widetilde{M}\to N$, and they are both distributed as $\nu$.

\begin{remark}
Intuitively, we are thinking of $Z$ as representing some outside data, independent of  $\{W_t\}_{t\geq 0}$, that is informing the model.  $\mathcal{G}_t$ can then be taken to be the completed sigma algebra generated by $Z$ and the Wiener process up to time $t$.  Since $Z$ and the Wiener process are independent, $W_t$ is still a Wiener process with respect to the resulting filtration.)
\end{remark}

\item We have $x\in\mathbb{R}^n$ and $\mathbb{R}^n$-valued processes  $X^x_t$ and $\widetilde{X}^x_t$ that are adapted to $\mathcal{G}_t$ and $\widetilde{\mathcal{G}}_t$ respectively, and satisfy the SDEs
\begin{align}
&dX^x_t=b(t,X_t^x,Z)dt+\sigma(t,X_t^x,Z)dW_t,\,\, X^x_0=x,\\
& d\widetilde X^x_t=\widetilde b(t,\widetilde X_t^x,\widetilde{Z})dt+\sigma(t,\widetilde X_t^x,\widetilde{Z})d\widetilde W_t,\,\, \widetilde X^x_0=x,\label{tilde_SDE}
\end{align}
 where  $b:[0,\infty)\times \mathbb{R}^n\times N\to \mathbb{R}^n$, $\sigma:[0,\infty)\times\mathbb{R}^n\times N\to\mathbb{R}^{n\times m}$, and $\beta:[0,\infty)\times\mathbb{R}^n\times N\to \mathbb{R}^m$ are measurable and the modified drift is
 \begin{align}
\widetilde b=b+\sigma\beta.
 \end{align}
\item  $\beta(s,X_s^x)\in L^2_{loc}(W)$ and the following process is a martingale under $P$:
\begin{align}
\rho^x_T\equiv\exp\left(\int_0^T \beta(s,X_s^x,Z)\cdot dW_s-\frac{1}{2}\int_0^T \|\beta(s,X_s^x,Z)\|^2ds\right).
\end{align}
This holds if the Novikov condition is satisfied; see, for example, page 199 of \cite{karatzas2014brownian}.
\item The SDE \req{tilde_SDE} satisfies the following uniqueness in law property for all $T>0$:

Suppose $X^\prime_t$ is a continuous, adapted process on another filtered probability space satisfying the usual conditions, $(M^\prime,\mathcal{G}^\prime_\infty,\mathcal{G}_t^\prime,P^\prime)$, that is also equipped with $\mathcal{G}_t^\prime$-Wiener process $W_t^\prime$ and a $N$-valued $\mathcal{G}^\prime_0$-measurable random quantity, $Z^\prime$, distributed as $\nu$.  Finally, suppose also that  $X^\prime_t$  satisfies
\begin{align}
&X^\prime_t=x+\int_0^t\widetilde{b}(s,X_s^\prime,Z^\prime)ds+\int_0^t\sigma(s,X_s^\prime,Z^\prime)dW^\prime_t
\end{align}
for $0\leq t\leq T$. Then the joint distribution of $(X^\prime|_{[0,T]},Z^\prime)$ equals the joint distribution of $(\widetilde{X}^x|_{[0,T]},\widetilde{Z})$.

\end{enumerate}
\end{assumption}
Given this, we define $P^x=(X^x)_* P$ and $\widetilde P^x=(\widetilde X^x)_*\widetilde P$, i.e., the distributions on path-space,
\begin{align}
(\Omega,\mathcal{F}_\infty,\mathcal{F}_t)\equiv (C([0,\infty),\mathbb{R}^n),\mathcal{B}(C([0,\infty),\mathbb{R}^n)),\sigma(\pi_s,s\leq t)),
\end{align}
where $\pi_t$ denotes evaluation at time $t$ and $C([0,\infty),\mathbb{R}^n)$ is given the topology of uniform convergence on compact sets.   For $T\geq 0$, we let $P^x_T\equiv P^x|_{\mathcal{F}_T}$ and  $\widetilde{P}^x_T\equiv \widetilde{P}^x|_{\mathcal{F}_T}$.

The relative entropy can be computed via Girsanov's theorem, together with Lemma \ref{lemma:radon_mart}:
\begin{lemma}
Let $x\in\mathbb{R}^n$ and suppose  Assumption \ref{assump:SDE} holds. For all $T>0$ we have
\begin{align}
&R(\widetilde P^x|_{\mathcal{F}_{\tau\wedge T}}\|P^x|_{\mathcal{F}_{\tau\wedge T}})\leq\frac{1}{2} E_{\widetilde{P}}\left[\int_0^{\widetilde\sigma_T^x} \|\beta(s,\widetilde{X}^x_s,\widetilde{Z})\|^2ds\right],
\end{align}
where $\widetilde\sigma_T^x=(\tau\circ\widetilde{X}^x)\wedge T$.
In terms of the base process, we have
\begin{align}
&R(\widetilde P^x|_{\mathcal{F}_{\tau\wedge T}}\|P^x|_{\mathcal{F}_{\tau\wedge T}})\leq \frac{1}{2} E_{P}\left[\int_0^{\sigma^x_T} \|\beta(s,X^x_s,Z)\|^2ds\,\,\rho^x_{\sigma^x_T}\right],
\end{align}
where $\sigma^x_T=(\tau\circ X^x)\wedge T$ and
\begin{align}
\rho^x_{\sigma^x_T}=\exp\left(\int_0^{\sigma^x_T}\beta(s,X^x_s,Z)\cdot dW_s-\frac{1}{2}\int_0^{\sigma_T^x} \|\beta(s,X^x_s,Z)\|^2ds\right).
\end{align}
\end{lemma}

A class of alternative models that are of particular case of interest is covered by the following corollary:
 \begin{corollary}
Suppose the baseline model is the solution to the following SDE on $\mathbb{R}^n$,
\begin{align}
dX^x_t=b(t,X^x_t)dt+\sigma(t,X^x_t)dW_t,\,\,\,\, X^x_0=x,
\end{align}
and the alternative model is given by
\begin{align}
&d\widetilde{X}^x_t=b(t,\widetilde{X}^x_t)dt+\sigma(t,\widetilde{X}^x_t)\beta(t,\widetilde{X}^x_t,\widetilde{Y}^y_t,\widetilde{Z})dt+\sigma(t,\widetilde{X}^x_t)d\widetilde{W}^1_t,\,\,\,\, \widetilde{X}^x_0=x,
\end{align}
which includes a modified  drift that depends on external data, $\widetilde{Z}$, and is also coupled to another SDE
\begin{align}
&d\widetilde{Y}^y_t=f(t,\widetilde{X}^x_t,\widetilde{Y}^y_t,\widetilde{Z})dt+\eta(t,\widetilde{X}^x_t,\widetilde{Y}^y_t,\widetilde{Z}) d\widetilde{W}^2_t,
\end{align}
on $\mathbb{R}^k$, where $\widetilde{W}^1_t$ and $\widetilde{W}^2_t$ are independent Wiener processes.

Suppose also we have $Y^y_t$ such that
\begin{align}
dY^y_t=f(t,X_t^x,Y^y_t,Z)dt+\eta(t,X_t^x,Y^y_t,Z)dW_t^\prime,\,\,\, Y^y_0=y,
\end{align}
where $W_t^\prime$ is a Wiener process independent from $W$.

If the base, $(X^x_t,Y^y_t)$, and alternative, $(\widetilde{X}^x_t,\widetilde{Y}^y_t)$, systems satisfy Assumption \ref{assump:SDE}, we have
\begin{align}
R( (\widetilde{X}^x_*\widetilde{P})|_{\mathcal{F}_{\tau\wedge T}}\|(X^x_*P)|_{\mathcal{F}_{\tau\wedge T}})\leq&R( ((\widetilde{X}^x,\widetilde{Y}^y)_* \widetilde{P})|_{\mathcal{F}_{(\tau\circ\pi_1)\wedge T}}\|(({X}^x,{Y}^y)_* {P})|_{\mathcal{F}_{(\tau\circ\pi_1)\wedge T}})\\
\leq&\frac{1}{2}E_{\widetilde{P}}\left[\int_0^{(\tau\circ\widetilde{X}^x)\wedge T}\|\beta(s,\widetilde{X}^x_s,\widetilde{Y}^y_s,\widetilde{Z})\|^2ds\right]\notag
\end{align}
for any stopping time $\tau$ on $C([0,\infty),\mathbb{R}^n)$.  Here $\pi_1$ denotes the map that takes a path in $C([0,\infty),\mathbb{R}^{n}\times\mathbb{R}^k)$ and returns the path of the $\mathbb{R}^n$-valued component.
\end{corollary}

\section{UQ for General Discounted Observables}\label{app:discounted}
Theorem \ref{thm:UQ_integral_QoI} can be generalized beyond the case of exponential discounting.  The proof is very similar and so it is omitted.
 \begin{theorem} \label{thm:UQ_integral_QoI2}
Let $\pi$ be a sigma-finite positive measure on $\mathcal{T}$, $\tau:\Omega\to \overline{\mathcal{T}}$ be a $\mathcal{F}_t$-stopping time, $f:\mathcal{T}\times\Omega\to\overline{\mathbb{R}}$ be progressively measurable, and suppose we also have one of the following two conditions:
\begin{enumerate}
\item $\pi$ is a finite measure and $f_s\geq 0$  for $\pi$-a.e. $s$.
\item $E_{\widetilde{P}}\left[\int_{[0,\tau]} |f_s|d\pi\right]<\infty$.
\end{enumerate}
For $t\in\overline{\mathcal{T}}$ define the  progressive process $F_t=\int_{[0,t]} f_s \pi(ds)$. Then
\begin{align}
\pm E_{\widetilde{P}}\left[F_\tau\right]\leq&\int_{\mathcal{T}}\inf_{c>0}\left\{\frac{1}{c}\Lambda_P^{{1_{s\leq \tau} f_{s}}}(\pm c)+\frac{1}{c} R(\widetilde{P}|_{\mathcal{F}_{s\wedge\tau}}\|P|_{\mathcal{F}_{s\wedge\tau}})\right\}\pi(ds).\notag
\end{align}
\end{theorem}

\section{Hitting Times of Brownian Motion with Constant Drift}\label{app:BM_hitting_times}
For the reader's convenience, here we recall the  moment generating function and distribution of hitting times for Brownian  motion with constant drift; see, for example, Chapter 8 of \cite{shreve2004stochastic} and page 196 in \cite{karatzas2014brownian}.
\begin{lemma}
Let $W_t$ be a $\mathbb{R}$-valued Wiener process (i.e., Brownian motion)  on $(\Omega,\mathcal{F},P)$, $\mu\in\mathbb{R}$, $a\in\mathbb{R}\setminus\{0\}$.  Let $X_t=\mu t+W_t$  be  Brownian motion with constant drift $\mu$.  Define the level-$a$ hitting time $\tau_a=\inf\{t\geq 0: X_t=a\}$. For any measurable  $A\subset (0,\infty]$ we have
\begin{align}
P(\tau_a\in A)=\frac{|a|}{\sqrt{2\pi}}\int_{A\cap(0,\infty)} t^{-3/2} e^{-\frac{(a-\mu t)^2}{2t}}dt+\left(1-e^{\mu a-|\mu a|}\right)1_{\infty\in A}.
\end{align}
In addition, for $\lambda>0$ the MGF is
\begin{align}
E[e^{-\lambda \tau_a}]=\exp\left(a\mu-|a|\sqrt{\mu^2+2\lambda}\right).
\end{align}
\end{lemma}

\section{Expected Hitting Times of Perturbed Brownian Motion}\label{app:BM_Etau}
In this appendix we bound the expected hitting times (\ref{eq:BM_QoI2}) for perturbations of Brownain motion of the form (\ref{eq:BM_alt}), using the techniques of Section \ref{subsec:exp_stop_time}. We  assume that $\mu$ and $a$  have the same sign in order to ensure the hitting time is a.s.-finite under the baseline model.   For a given $\alpha>0$ we  specialize \req{eq:BM_alt} to the following:
\begin{flalign} \label{eq:BM_Ambiguity2}
\text{\bf      Alternative Models:} &&\text{ perturbations by drifts, }\beta, \text{ with }\|\beta\|_\infty\leq\alpha.&& &&
\end{flalign}
 This example uses the technique of Section \ref{subsec:exp_stop_time}.

 The optimal bounds can be obtained directly (i.e., without using the UQ methods developed above) by combining the known distribution of $\tau_a$ \cite{karatzas2014brownian} with the comparison principle, i.e., with
 \begin{align}\label{eq:BM_comparison}
(\mu-\|\beta\|_\infty)t+\widetilde W_t \leq \widetilde X_t\leq (\mu+\|\beta\|_\infty)t+\widetilde W_t.
 \end{align}
We assume that $\alpha<\mu$ so  that the constant-drift lower bound, and hence also the perturbed model, have a.s.-finite hitting times. This yields
 \begin{align}\label{eq:BM_hitting_optimal}
\frac{a}{\mu+\alpha}\leq  {E}_{\widetilde{P}^0}[\tau_a]\leq \frac{a}{\mu-\alpha},
 \end{align}
 where upper and lower bounds are achieved in the cases of constant drift perturbations  $\alpha$ and $-\alpha$ respectively.  

The optimal bounds \req{eq:BM_hitting_optimal} provide  a useful test-case for our  goal-oriented UQ bounds.   \req{eq:BM_Ambiguity2} together with Girsanov's theorem gives a bound on the relative entropy:
\begin{align}
&R( \widetilde{P}^0|_{\mathcal{F}_{\tau_a\wedge n}}\|P^0|_{\mathcal{F}_{\tau_a\wedge n}})\leq\frac{\alpha^2 }{2}E_{\widetilde{P}^0}\left[ \tau_a\wedge n\right]
\end{align} 
for all $n$; in the language of Theorem \ref{thm:goal_info_ineq3}, we can take $\mathcal{G}=\mathcal{F}_{\tau_a\wedge n}$ and $G=\frac{\alpha^2 }{2}(\tau_a\wedge n)$.

Using  the known formula for the cumulant generating function of $\tau_a$ (see Chapter 8 of \cite{shreve2004stochastic})  and analytic continuation, the cumulant generating function in the baseline model can be computed:
\begin{align}\label{eq:BM_base_cumulant}
\Lambda_{P^0}^{\tau_a}(c)=a\mu-a\sqrt{\mu^2-2c},\,\,\, c<\mu^2/2.
\end{align}
Therefore Corollary \ref{cor:E_tau_bound} yields
\begin{align}\label{eq:BM_hitting_method}
& -\inf_{c>0}\left\{ (c+K)^{-1}\left(a\mu-a\sqrt{\mu^2+2c}\right)\right\} \\
\leq& E_{\widetilde{P}^0}[\tau_a]\notag\\
\leq&\inf_{K<\lambda<\mu^2/2}\left\{ (\lambda-K)^{-1}\left(a\mu-a\sqrt{\mu^2-2\lambda}\right)\right\},\notag
\end{align}
where $K\equiv\frac{1}{2}\alpha^2$. The above optimization problems can be solved explicitly, with minimizers $c^*=K+\sqrt{2K}\mu$ and $\lambda^*=-K+\sqrt{2K} \mu$ respectively.  Computing the corresponding minimum values, one finds that the UQ bounds resulting from our method, \req{eq:BM_hitting_method}, are the same as the optimal bounds, \req{eq:BM_hitting_optimal}, that were obtained from the comparison principle.    We also note that, while the comparison principle is only available in very specialized circumstances, our UQ method is quite general; in particular, our method is not restricted to 1-D systems.  Finally, the non-goal-oriented UQ bounds, \req{goal_oriented_bound}, are not effective here;  the non-goal-oriented relative entropy over an infinite time-horizon is $R(\widetilde{P}^0\|P^0)$, which  is infinite for the constant drift perturbations $\beta=\pm \alpha$.

\section{Vasicek Model UQ Bounds}\label{app:vasicek}
Here we record a pair of lemmas that are needed for the analysis of option values under the Vasicek interest rate model, as well as the final UQ bound obtained by our method. See Section \ref{sec:Vasicek} for definitions of the notation.
\begin{lemma}
Suppose $r>\frac{\widetilde\sigma^2}{2\gamma^2}$.  Then $\lim_{n\to\infty}\int_0^n r+\Delta r_s ds=\infty$ a.s. under both the base and the alternative models. 
\end{lemma}
\begin{proof}

 The distribution of $\Delta r$ is the same under both the base and alternative models, so it does not matter which one we consider. As shown in \cite{MarioAbundo2015}, $\int_0^t r+\Delta r_s ds$ is normally distributed  for all $t\geq 0$, with mean $m_t\equiv rt$ and variance 
\begin{align}
\widetilde{\sigma}_t^2=\frac{\widetilde{\sigma}^2}{2\gamma^3}\left(2\gamma t+4 e^{-\gamma t} - e^{-2\gamma t} - 3\right).
\end{align}
\begin{comment}
$\widetilde{\sigma}_t^2=g(\phi(t))$, where
\begin{align}
g(t)\equiv&\frac{\widetilde\sigma^2t}{\gamma^2(2\gamma t+\widetilde\sigma^2)}+\frac{\widetilde\sigma^2}{2\gamma^3}\log(1+2\gamma t/\widetilde\sigma^2)-\frac{2\widetilde\sigma^2}{\gamma^3}\left(1-(1+2\gamma t/\widetilde\sigma^2)^{-1/2}\right),\notag\\
\phi(t)\equiv&\frac{\widetilde\sigma^2}{2\gamma}(e^{2\gamma t}-1).
\end{align}
Note that there exists a constant $C>0$ such that $\widetilde{\sigma}_t^2\leq C+\frac{\widetilde{\sigma}^2}{\gamma^2}t$ for all $t\geq 0$.
\end{comment}
Using this, for any $R>0$ we can compute
\begin{align}
P\left(\liminf_{n\to\infty}\int_0^n r+\Delta r_s ds<R\right)\leq&\lim_{N\to\infty}\sum_{n\geq N} P\left(\int_0^n r+\Delta r_s ds<R\right)\\
\leq&\lim_{N\to\infty}\sum_{n\geq N} P\left(\exp\left(-\int_0^n r+\Delta r_s ds\right)>\exp(-R)\right)\notag\\
\leq&e^R\lim_{N\to\infty}\sum_{n\geq N} e^{-rn+\frac{1}{2}\widetilde \sigma_n^2}\notag\\
\leq& e^{R+\widetilde{\sigma}^2/\gamma^3}\lim_{N\to\infty}\sum_{n\geq N} e^{-\left(r-\frac{\widetilde{\sigma}^2}{2\gamma^2}\right)n}=0.\notag
\end{align}
Therefore
\begin{align}
P\left(\liminf_{n\to\infty}\int_0^n r+\Delta r_s ds<\infty\right)=\lim_{R\to\infty} P\left(\liminf_{n\to\infty}\int_0^n r+\Delta r_s ds<R\right)=0.
\end{align}
\end{proof}

\begin{lemma}
\begin{align}
\rho_t\equiv\exp\left(\int_0^t \sigma^{-1}\Delta r_s dW_s-\frac{1}{2}\int_0^t\sigma^{-2}|\Delta r_s |^2ds\right).
\end{align}
is a $P$-martingale.
\end{lemma}
\begin{proof}

  As shown in \cite{karatzas2014brownian} (see Corollary 5.14), it suffices to prove that there exists $\{t_k\}_{k=0}^\infty$ with $0=t_0<t_1<...<t_n\nearrow \infty$ such that
\begin{align}
E\left[\exp\left(\frac{1}{2}\int_{t_{n-1}}^{t_n}\sigma^{-2}|\Delta r_s |^2ds\right)\right]<\infty \text{ for all }n.
\end{align}
Letting $\Delta t_n=t_n-t_{n-1}$ and using Jensen's inequality, we have
\begin{align}\label{eq:sup_E}
E\left[\exp\left(\frac{1}{2}\int_{t_{n-1}}^{t_n}\sigma^{-2}|\Delta r_s |^2ds\right)\right]\leq& E\left[\Delta t_n^{-1}\int_{t_{n-1}}^{t_n}\exp\left(\frac{\Delta t_n}{2}\sigma^{-2}|\Delta r_s |^2\right)ds\right]\\
\leq&\sup_{t_{n-1}\leq s\leq t_n}E\left[\exp\left(\frac{\Delta t_n}{2}\sigma^{-2}|\Delta r_s |^2\right)\right].\notag
\end{align}
$\Delta r_t$ is normal with mean $0$ and variance $\widetilde\sigma^2(1-e^{-2\gamma t})/(2\gamma)$. Hence, if we fix $\Delta t_n=\epsilon$ small enough then the upper bound in \req{eq:sup_E} is finite for all $n$.
\end{proof}

The final UQ bound on the Vasicek model  obtained by the calculations in Section  \ref{sec:Vasicek} is the following:
\begin{align}  \label{eq:Vasicek_final_bound}
 \pm \widetilde{E}[F_\tau]\leq \inf_{c>0}\left\{\frac{1}{c}\log  \int \int_{-\infty}^\infty (2\pi\widetilde{\sigma}_t^2)^{-1/2} \exp\left(\pm c (K-L)e^{-z}1_{t<\infty}1_{z\geq 0}\right) e^{-\frac{(z-rt)^2}{2\widetilde{\sigma}_t^2}}dz P_{\tau\circ X^r}(dt)\!+\frac{D}{c}\right\}\!/K_{\mp},
\end{align}
where
\begin{align} 
& K_\pm=\pm\inf_{c>0}\bigg\{\frac{1}{c}\log\left(1+(e^{\pm c}-1)\int \frac{1}{2}\left(\text{erf}\left(rt/\sqrt{2\widetilde{\sigma}^2_t}\right)+1\right) P_{\tau\circ X_r}(dt)\right)+\frac{D}{c}\bigg\},\\
&D=\inf_{\lambda>1}\left\{(\lambda-1)^{-1}\log \int e^{\gamma t/2}\left[\frac{1}{\sqrt{1-\frac{\lambda\widetilde\sigma^2}{\sigma^2\gamma^2}}}\sinh\left(\gamma t\sqrt{1-\frac{\lambda\widetilde\sigma^2}{\sigma^2\gamma^2}}\right)+\cosh\left(\gamma t\sqrt{1-\frac{\lambda\widetilde\sigma^2}{\sigma^2\gamma^2}}\right)\right]^{-1/2}P_{\tau\circ X^r}(dt)\right\},\notag\\
&P_{\tau\circ X_r}(dt)= \frac{|a|}{\sqrt{2\pi}}t^{-3/2}e^{-(a-\mu t)^2/(2t)}1_{(0,\infty)}dt,\,\,\,a= -\sigma^{-1}\log(X_0/L),\,\,\, \mu= r/\sigma-\sigma/2,\notag\\
&\widetilde{\sigma}_t^2=\frac{\widetilde{\sigma}^2}{2\gamma^3}\left(2\gamma t+4 e^{-\gamma t} - e^{-2\gamma t} - 3\right).\notag
\end{align}
In creating Figure \ref{fig:option_value_Vasicek}, the integrals over $t$ and $z$ were computed numerically via quadrature methods and the optimization over $c$ was also performed numerically; see Remark \ref{remark:unimodal}. Finally, note that $K_+$ and $K_-$ are, respectively, upper and lower bounds on a probability and they reflect that fact: 
\begin{align}
&K_+\leq\lim_{c\to\infty}\bigg(\frac{1}{c}\log\left(1+(e^{ c}-1)\int \frac{1}{2}\left(\text{erf}\left(rt/\sqrt{2\widetilde{\sigma}^2_t}\right)+1\right) P_{\tau\circ X_r}(dt)\right)+\frac{D}{c}\bigg)=1,\\
&K_-\geq -\lim_{c\to\infty}  \bigg(\frac{1}{c}\log\left(1+(e^{- c}-1)\int \frac{1}{2}\left(\text{erf}\left(rt/\sqrt{2\widetilde{\sigma}^2_t}\right)+1\right) P_{\tau\circ X_r}(dt)\right)+\frac{D}{c}\bigg)=0.\notag
\end{align}

\section{American Put Options: Bounded Rate Perturbations}\label{app:options}
In this appendix, we derive UQ bounds on the value of American put options for another class of rate perturbations.  This is a simpler example to analyze than the Vasicek model from Section \ref{sec:Vasicek} and provides a good benchmark case, as we will also have access to comparison principle bounds.  We use the terminology and notation introduced in Section \ref{sec:options}.

Specifically, here we consider interest rate perturbations of the following form: 
\begin{flalign}\label{eq:Delta_r_bound}
\text{\bf      Alternative Models:} &&\Delta r_-\leq h_{-}(t)\leq \Delta r(t,\cdot,\cdot) \leq h_{+}(t)\leq \Delta r_+, \,\,t\geq 0,&& &&
\end{flalign}
 where $\Delta r_\pm\in\mathbb{R}$, $r+\Delta r_-> 0$, and  $h_{-},h_+:[0,\infty)\to\mathbb{R}$. The parameters $\Delta r_-$ and $\Delta r_+$ define a fluctuation range, with $h_\pm$ allowing for specification of a time-dependent envelope on the fluctuation.

 An elementary bound on $\widetilde{X}_t$ follows  from the uniform bounds $\Delta r_-\leq \Delta r(t,x,y)\leq \Delta r_+$ together with the comparison principle: 
 \begin{align}\label{eq:stop_time_ineq}
X^{r+\Delta r_-}_t\leq\widetilde X_t\leq X^{r+\Delta r_+}_t, \text{ and   hence
 } \tau\circ X^{r+\Delta r_-}\leq \tau\circ\widetilde X\leq \tau\circ X^{r+\Delta r_+}.
\end{align}

Using only the information $\Delta r_-\leq \Delta r(t)\leq \Delta r_+$, the optimal UQ bounds can be obtained from \req{eq:Delta_r_bound} and  \req{eq:stop_time_ineq}, together with the exact value for the constant-rate processes, \req{eq:finance_base_QoI}:
\begin{align}\label{eq:finance_naive_bound}
 (K-L)\left({L}/{X_0}\right)^{2(r+\Delta r_+)/\sigma^2}\leq{E}\left[\widetilde{V}_{\tau\circ\widetilde{X}}[\widetilde{X},\widetilde{Y}]\right]\leq (K-L)\left({L}/{X_0}\right)^{2(r+\Delta r_-)/\sigma^2}.
\end{align}
This  provides a useful comparison for the UQ bounds derived via the methods developed above.

To use the goal-oriented method, first note the relative entropy bounds (obtained from Girsanov's theorem):
\begin{align}\label{eq:finance_rel_ent}
&R({P}_{\widetilde{X},\widetilde{Y}}|_{\mathcal{F}_{\tau\wedge n}}\|P_{X^r,Y}|_{\mathcal{F}_{\tau\wedge n}})\leq {E}_{\widetilde{X},\widetilde{Y}}\left[G_n\right],\,\,\,\,\,\,\,G_n[x,y]\equiv \frac{\sigma^{-2}}{2}\int_0^{\tau( x)\wedge n}|\Delta r(s,{x}_s,{y}_s)|^2ds.
\end{align}
$G_n$ is $\mathcal{F}_{\tau\wedge n}$-measurable, hence we can   use Theorem \ref{thm:goal_info_ineq3} to obtain
\begin{align}\label{eq:option_info_ineq}
\pm {E}_{\widetilde{X},\widetilde{Y}}\left[F_{\tau\wedge n}\right]\leq&\inf_{c>0}\left\{\frac{1}{c}\log E_{X^r,Y}\left[\exp\left(\pm cF_{\tau( x)\wedge n}+\frac{\sigma^{-2}}{2}\int_0^{\tau( x)\wedge n}|\Delta r(s,{x}_s,{y}_s)|^2ds\right)\right]\right\}.
\end{align}
The above steps are justified by the bounds on $\Delta r$.

\begin{figure}
\begin{minipage}[b]{0.5\linewidth}
  \includegraphics[height=6cm]{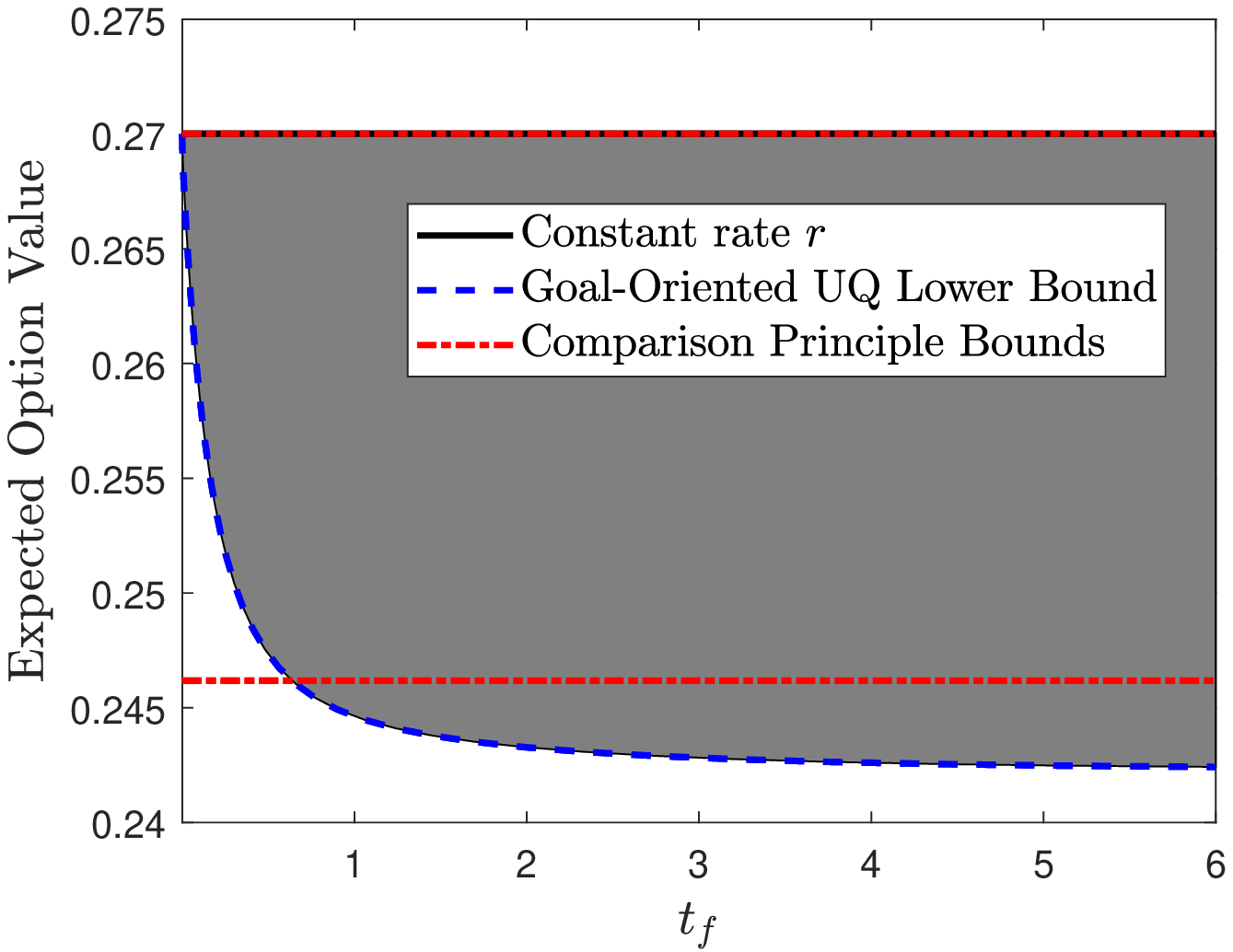}
 \end{minipage}
 \begin{minipage}[b]{0.5\linewidth}
  \includegraphics[height=6cm]{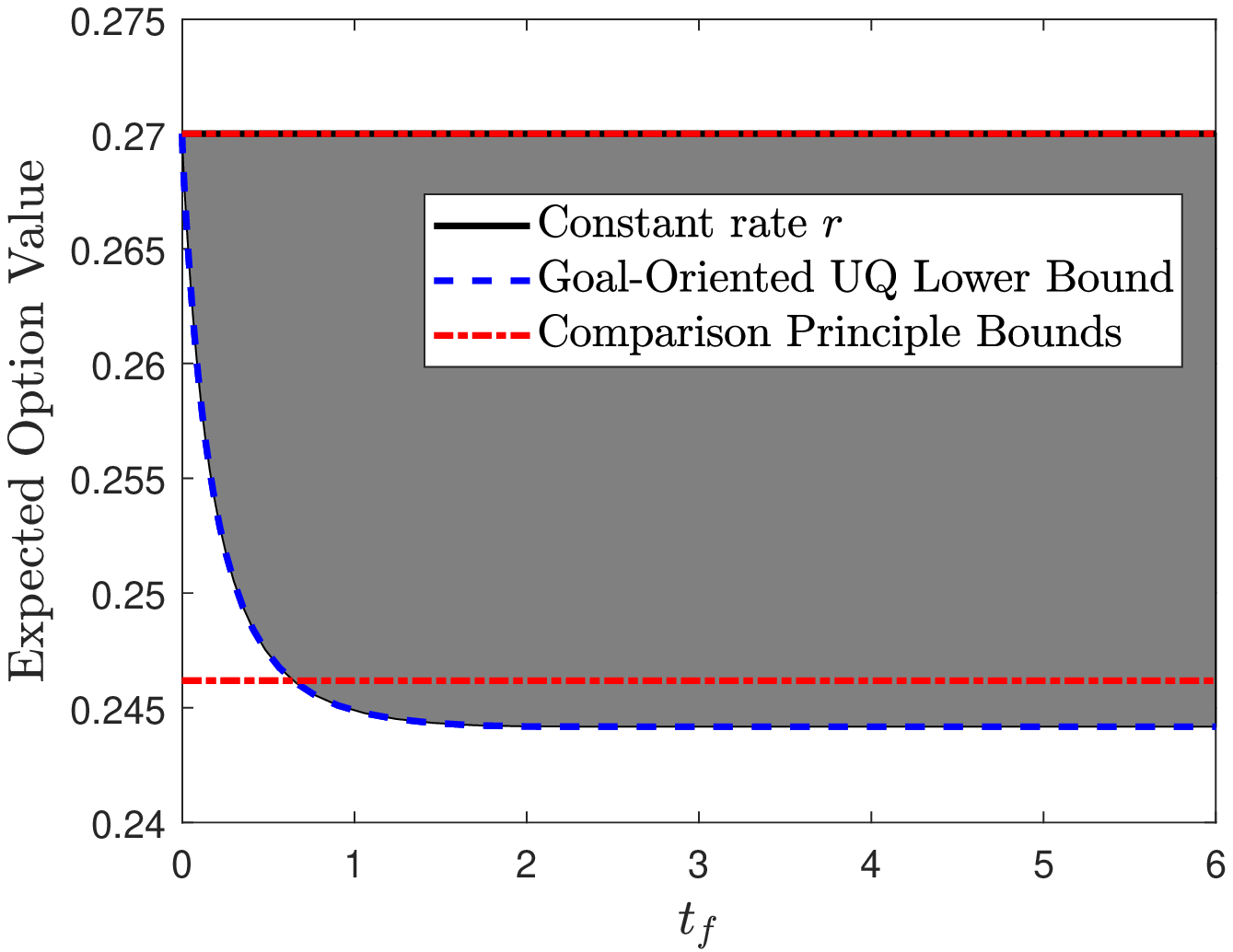}
 \end{minipage}
\caption{ Bounds on the expectation of the option value as a function of $t_f$, the time by which the interest rate is assumed to have completed the drop from $r+\Delta r_+$ to $r$; parameter values are  $r=2$, $K=1$, $L=1/2$, $\sigma=3$, $X_0=2$, $\Delta r_+=0.3$.  The dashed blue curve is the lower bound obtained from our method, \req{eq:option_UQ_bounded_dr}. The upper and lower red  lines are the  comparison-principle bounds based solely on knowledge of $\|\Delta r\|_\infty$; see \req{eq:finance_naive_bound}. The expected value in the constant-rate, $r$, model (black line, \req{eq:finance_base_QoI}) is the same as the optimal upper bound. The optimal upper bound together with the goal-oriented lower bound constrain the QoI to the gray region.  The left plot was obtained using the base-model rate fixed at $r$ and the right plot was obtained by minimizing over the rate parameter that is assigned to the baseline model, as described in Remark \ref{remark:kappa_min} (i.e., replacing $r$ with $\kappa$ on the right-hand-side of \req{eq:option_UQ_bounded_dr}, replacing $h_\pm$ with $h^\kappa_-=r-\kappa$, $h^\kappa_+=r-\kappa+\Delta r_+1_{[0,t_f]}$, and minimizing over $\kappa$). Note that here, the best performance is achieved by  a combination the comparison-principle bounds with our goal-oriented bounds.}  
\label{fig:options_plot_r_drop}
\end{figure}

The bounds on $\Delta r$ also allow us to take $n\to\infty$ in \req{eq:option_info_ineq} to obtain
 \begin{align}\label{eq:option_UQ}
\pm E\left[\widetilde{V}_{\tau\circ\widetilde{X}}[\widetilde{X},\widetilde{Y}]\right]=&\pm {E}_{\widetilde{X},\widetilde{Y}}\left[F_{\tau}\right]=\lim_{n\to\infty}\pm {E}_{\widetilde{X},\widetilde{Y}}\left[F_{\tau\wedge n}\right]\\
\leq&\inf_{c>0}\left\{\frac{1}{c}\log E_{X^r,Y}\left[\exp\left(\pm cF_{\tau(x)}+\frac{\sigma^{-2}}{2}\int_0^{\tau( x)}|\Delta r(s,{x}_s,{y}_s)|^2ds\right)\right]\right\}.\notag
\end{align}
  Bounding $\Delta r$ in terms of $h_\pm$ then yields
 \begin{align}\label{eq:option_UQ_bounded_dr}
&\pm{E}\left[\widetilde{V}_{\tau\circ\widetilde{X}}[\widetilde{X},\widetilde{Y}]\right]\leq\inf_{c>0}\left\{\frac{1}{c}\log E_{X^r}\left[M_c^{\pm}\right]\right\},\\
&M^\pm_c[x]\equiv\exp\left(\frac{\sigma^{-2}}{2}\!\int_0^{\tau( x)}\!|h(s)|^2ds\pm c(K-L)e^{-\,\int_0^{\tau(x)} r+h_{\mp}(s)ds}1_{\tau(x)<\infty}\right),\notag
\end{align}
where $h(t)\equiv\max\{h_+(t),-h_-(t)\}$.  The expectation can be evaluated using   the known distribution of $\tau\circ X^r$ (see the discussion surrounding \req{eq:tau_Xr}):
\begin{align}\label{eq:option_bounded_dr_expectation}
E_{X^r}[M^\pm_c]=&\frac{|a|}{\sqrt{2\pi}}\int_0^\infty \!\!\! \exp\!\left(\frac{\sigma^{-2}}{2}\!\int_0^th_s^2ds\pm c(K-L)e^{-\int_0^t r+h_{\mp}(s)ds}-\frac{(a-\mu t)^2}{2t}\right)t^{-3/2}dt\\
&+\exp\!\left(\frac{\sigma^{-2}}{2}\int_0^{\infty} h_s^2ds\right)(1-e^{\mu a-|\mu a|}).\notag
\end{align}
Note that all reference to the driving process $Y$, which is not coupled to $X^r$, has been eliminated.

\begin{remark}\label{remark:kappa_min}
One can obtain somewhat tighter bounds in \req{eq:option_UQ_bounded_dr} by optimizing over the parameter that defines the baseline model, similarly to \req{eq:goal_info_ineq_base_param}; in other words,  replace $r$ with a parameter, $\kappa$, on the right-hand side, find the corresponding $h^\kappa_\pm$ satisfying 
$h_-^\kappa\leq r-\kappa+\Delta r\leq h_+^\kappa$, and minimize over $\kappa$; see Figure \ref{fig:options_plot_r_drop}.
\end{remark}

 We show numerical results for the following scenario: Suppose the rate drops from $r+\Delta r_+$ to $r$, and is certain to have completed this drop by time $t_f$ (one could also study a similar scenario where $r$ increases), but the timing and profile of the drop is otherwise unknown, i.e., we consider $\Delta r_+$ and $t_f$ to be known, and the only constraint on $\Delta r$ is  $0\leq \Delta r(t,\cdot,\cdot)\leq \Delta r_+ 1_{[0,t_f]}(t)$. Using $h_{-}=0$ and $h_+(t)=\Delta r_+ 1_{[0,t_f]}(t)=h(t)$,  UQ bounds are obtained by combining \req{eq:option_UQ_bounded_dr} with \req{eq:option_bounded_dr_expectation}.

In Figure \ref{fig:options_plot_r_drop} we show the resulting bounds on the option value, as a function of $t_f$. Note that the bounds resulting from our UQ method are an improvement over the comparison-principle bound  \req{eq:finance_naive_bound} for an initial range of $t_f$, and stabilizes at a value near the comparison-principle bound for large $t_f$. For large $t_f$ our method is sub-optimal, but still competitive. Minimizing over the parameters assigned to the baseline model (right plot) improves the bound.  We find similar behavior when assuming other time-dependent envelopes of $\Delta r$. It should again be noted that  the comparison principle is available only in very special circumstances and is used here for benchmarking purposes; see Section \ref{sec:Vasicek} for a more realistic  model where the comparison principle is not an effective tool.

\section*{Acknowledgments}
The research of J.B., M.K., and L. R.-B. was partially supported by NSF TRIPODS CISE-1934846. The research of M.K and L. R.-B. was partially supported by the National Science Foundation (NSF) under the grant DMS-1515712 and the Air Force Office of Scientific Research (AFOSR) under the grant FA-9550-18-1-0214.

\bibliography{UQStoppingTime_arXiv}

%\bibliographystyle{amsplain}

%\bibliography{refs}

\end{document}